\documentclass[a4paper,12pt]{amsart}
\usepackage{amssymb,amsfonts,amsthm}
\usepackage{amstext}
\usepackage{comment}
\usepackage{amsmath}
\usepackage{hyperref}  

\usepackage{color}
\usepackage[latin1]{inputenc}
\usepackage[active]{srcltx}
\usepackage{mathrsfs} 

\usepackage{tensor}
\usepackage{color}
\usepackage{mathptmx}

\theoremstyle{plain}

\newtheorem{theorem}{Theorem}[section]
\newtheorem{lem}[theorem]{Lemma}
\newtheorem{defn}[theorem]{Definition}
\newtheorem{cor}[theorem]{Corollary}

\usepackage{chngcntr} 

\usepackage{mathrsfs} 

\usepackage{comment} 

\usepackage{hyperref}  
\begin{document}
	
	\title{Generalization of the HSIC and distance covariance using  positive definite independent kernels}
	
	\author{J. C. Guella}
	\email{jcguella@ime.unicamp.br}
	\address{Institute of Mathematics, Statistics and Scientific Computing, University of Campinas}
	\begin{abstract} Hilbert-Schmidt  independence criterion and distance covariance are methods to describe independence of random variables using either the Kronecker product of positive definite kernels  or  the Kronecker product  of  conditionally negative definite kernels. In this paper we generalize both methods by providing an independence criteria using a new concept,  of positive definite independent kernels. We provide a characterization of the radial kernels that are positive definite independent on all  Euclidean spaces  and we present several examples.
	\end{abstract}
	

	\maketitle

	\tableofcontents
	
\section{\textbf{Introduction}}

A method to compare if two random variables $X$ and $Y$ are independent, is to make use of the well known relation between positive definite kernels and conditionally negative definite kernels with metrics on the space of probabilities, which is provided by the kernel mean embedding.	For that, we pick $\gamma: X \times X \to \mathbb{R}$ and   $\varsigma: Y \times Y \to \mathbb{R}$ continuous positive definite (or conditionally negative definite) kernels, then for a Radon regular probability $\pi$   in $X\times Y$  with marginals $\pi_{1}$ and $\pi_{2}$, and with good integrability properties with respect to the kernels  $\varsigma_{1}$, $\varsigma_{2}$ (described in Corollary \ref{integmarginkroeprodcor}), the following double integral  
$$
\int_{X\times Y} \int_{X\times Y} \gamma(x,x^{\prime}) \varsigma(y,y^{\prime})d[\pi - \pi_{1}\otimes \pi_{2}](x,y)d[\pi - \pi_{1}\otimes \pi_{2}](x^{\prime}, y^{\prime}),
$$	
is well defined and is nonnegative. For some kernels the double integral is zero if and only if $\pi = \pi_{1}\otimes \pi_{2}$. For instance, if $\varsigma_{1}$ and $\varsigma_{2}$ are Gaussian kernels on Euclidean spaces,  in this case both kernels are positive definite. Another example is when $\varsigma_{1}$ and $\varsigma_{2}$ are the standard metric on Hilbert  spaces,  in this case both kernels are conditionally negative  definite, \cite{lyons2013}. 

The function HSIC (Hilbert Schmidt independence criterion) \cite{sejdinovic2013equivalence, gretton2005measuring, gretton2008kernel}, takes a probability $\pi$ to the previous double integral. The distance covariance \cite{szekely2009brownian} \cite{lyons2013, lyons2018errata, lyons2021second}
,  is essentially the same  function, but is defined as the sum of $4$ functions that only depends on $\pi$. We present more details about the distance covariance and the generalization that we propose in Section \ref{Distance covariance}.      

The fact that HSIC or distance covariance is defined only on a Kronecker product of kernels is very restrictive, as it does not allow more intricate connections between the points of $X$ and the points of $Y$. Below we present  an example of a kernel, which is not a Kronecker product,  that is able to discern if a probability on a product space is the product of its marginals. For this, consider Hilbert spaces $\mathcal{H}$ and $\mathcal{H}^{\prime}$ and  the following kernel defined in   $\mathcal{H}\times \mathcal{H}^{\prime}$
$$
\mathfrak{I}((x,y),(x^{\prime}, y^{\prime})):=(\|x-x^{\prime}\|^{2} + \|y-y^{\prime}\|^{2})^{3/2}.
$$
By the integral relation 
$$
t^{3/2}=\frac{3}{4\pi^{1/2}}\int_{[0, \infty)}(e^{-rt}-1+rt)\frac{1}{r^{5/2}}dr, \quad t \geq 0
$$ 
if $\pi$ is a Radon regular probability on $\mathcal{H}\times \mathcal{H}^{\prime}$ (that satisfies the integrability assumption in Lemma \ref{integmargincompl2}) and $\Delta_{\pi}:= \pi-\pi_{1}\otimes \pi_{2}$, we have that
\begin{equation}\label{ex1}
	\begin{split}
		&\int_{\mathcal{H}\times \mathcal{H}^{\prime}}\int_{\mathcal{H}\times \mathcal{H}^{\prime}}(\|x-x^{\prime}\|^{2} + \|y-y^{\prime}\|^{2})^{3/2} d\Delta_{\pi}(x,y) d\Delta_{\pi}(x^{\prime},y^{\prime})\\
		& =\frac{3}{4\pi^{1/2}}\int_{[0, \infty)} \left [\int_{\mathcal{H}\times \mathcal{H}^{\prime}}\int_{\mathcal{H}\times \mathcal{H}^{\prime}}e^{-(\|x-x^{\prime}\|^{2} + \|y-y^{\prime}\|^{2})r} d\Delta_{\pi}(x,y) d\Delta_{\pi}(x^{\prime},y^{\prime}) \right ]\frac{1}{r^{5/2}}dr.
	\end{split}
\end{equation}

It is a well known fact that the Gaussian kernel on any Euclidean space is Integrally strictly positive definite, and  recently proved that it also holds on any Hilbert space in \cite{gaussinfi}. So the double integral in Equation \ref{ex1} is not only a nonnegative number, but it is zero if and only if $\pi= \pi_{1}\otimes \pi_{2}$. More generally, we can replace $\pi_{1}\otimes \pi_{2}$ by a another probability $\pi^{\prime}$, with the same marginals as $\pi$, and obtain the same conclusion.

Our  objective in this paper is to describe which symmetric kernels $\mathfrak{I}: (X \times Y) \times (X \times Y) \to \mathbb{R} $ are able to discern Radon regular probabilities $\pi, \pi^{\prime}$ defined in $X\times Y$, that have the same marginals, using the double integral
\begin{equation}\label{objective}
	\int_{X\times Y} \int_{X\times Y} \mathfrak{I}((x,y),(x^{\prime}, y^{\prime}))d[\pi - \pi^{\prime}](x,y)d[\pi - \pi^{\prime}](x^{\prime},y^{\prime}),
\end{equation}
under the restriction that the integrals are well defined. The concept that will allow us to understand this problem is of a  positive definite independent kernel (PDI). In Section \ref{positive definite independent kernels} we describe its basic properties and the relation between this new concept with positive definite and conditionally negative definite kernels. Under additional requirements, we obtain results that describes which Radon regular probabilities we can compare on the general case  and  that the square root of Equation \ref{objective} defines a pseudometric on the space of probabilities in $X\times Y$ with fixed marginals. When it defines a metric for every pair of marginals that integrates $\mathfrak{I}$, we say that $\mathfrak{I}$ is an PDI-Characteristic kernel.

In Corollary \ref{integmarginkroeprodcor}  we reprove the well known fact that the standard HSIC/distance covariance is well defined if and only if each conditionally negative definite kernel is able to discern distinct probabilities (CND-Characteristic by our terminology) by a different method than \cite{sejdinovic2013equivalence},  without the need of the Hilbert Schmidt  operator for which the name HSIC comes from.

In Section \ref{Bernstein  functions of two variables}
we move to a characterization problem, where the focus is to describe the set of continuous functions $f: [0, \infty) \times [0, \infty) \to \mathbb{R}$, for which the kernel 
$$
f( \|x-x^{\prime}\|,\|y - y^{\prime}\| ), \quad  x,x^{\prime} \in  \mathbb{R}^{d}, \quad x^{\prime}, y^{\prime} \in \mathbb{R}^{d^{\prime}}
$$
is positive definite independent   for every  $d, d^{\prime }\in \mathbb{N}$,  in a similar way that Schoenberg characterized the positive definite  case in \cite{schoenbradial}, and the conditionally negative definite version in \cite{neuschoen}. The proof is not constructive, instead we rely on the fact that the representation for the conditionally negative definite case is unique (Theorem \ref{reprcondneg}), and some techniques based in \cite{jeanrad} concerning measure valued positive definite radial  kernels on Euclidean spaces.  

In Section \ref{Bernstein  functions of two variables} we also present a  characterization for the   continuous functions $f: [0, \infty)  \to \mathbb{R}$, for which the kernel 
$$
f( \|x-x^{\prime}\|^{2}+\|y - y^{\prime}\|^{2} ), \quad  x,x^{\prime} \in  \mathbb{R}^{d}, \quad x^{\prime}, y^{\prime} \in \mathbb{R}^{d^{\prime}}
$$
is positive definite independent   for every  $d, d^{\prime }\in \mathbb{N}$, by relating  those functions with  completely monotone functions of order $2$, being $f(t)= t^{a}$, $a \in (1,2)$, an example. 

In Section \ref{Examples of PDI-Characteristic kernels} we provide several examples of PDI-Characteristic kernels (even on non-Hilbert spaces)  based on the set on the functions obtained in  Section \ref{Bernstein  functions of two variables}, the relations of conditionally negative definite kernels and norms on Hilbert spaces, and the already mentioned fact that the Gaussian kernel is integrally strictly positive definite in any Hilbert space.

In Section  \ref{Definitions} we formally describe the most important definitions and results that we use. In Section  \ref{Proofs} we present the proofs of the results in Sections \ref{positive definite independent kernels}, \ref{Bernstein  functions of two variables}, \ref{Examples of PDI-Characteristic kernels}, \ref{Distance covariance}  and a few technical results.

\section{\textbf{Definitions}}\label{Definitions}	

We recall that a  nonnegative measure $\lambda$ on a  Hausdorff space $X$ is    Radon regular (which we simply refer as Radon) when it is a Borel measure such that is finite on every compact set of $X$ and
\begin{enumerate}
	\item[(i)](Inner regular)$\lambda(E)= \sup\{\lambda(K), \ \ K \text{ is compact }, K\subset E\} $ for every Borel set $E$.
	\item[(ii)](Outer regular) $\lambda(E)= \inf\{\lambda(U), \ \ U \text{ is open }, E\subset U\} $ for every Borel set $E$.
\end{enumerate}

We then said that a real valued measure $\lambda$ of bounded variation is Radon if its variation is a Radon measure. The vector space of such measures is denoted by $\mathfrak{M}(X)$. Recall that every Borel measure of finite variation (in particular, probability measures) on a separable complete metric space is necessarily Radon. 

An semi-inner product on a real  vector space $V$  is a bilinear  real valued function $( \cdot, \cdot)_{V} $ defined on  $V\times V$ such that $(u,u)_{V} \geq 0$ for every $u \in V$. When this inequality is an equality only for $u=0$, we say that  $( \cdot, \cdot)_{V}$ is an inner-product. Similarly, an pseudometric on a set $X$ is a symmetric function $d: X\times X \to [0, \infty)$, such that $d(x,x)=0$ that satisfies the triangle inequality. If $d(x,y)=0$ only when $x=y$, $d$ is a metric on $X$. 

-\textbf{Positive definite kernels}


A symmetric kernel $K: X \times X \to \mathbb{R}$ is called \textbf{Positive Definite (PD)} if for every finite quantity of distinct points $x_{1}, \ldots, x_{n} \in X$ and scalars $c_{1}, \ldots, c_{n} \in \mathbb{R}$, we have that
$$
\sum_{i, j =1}^{n}c_{i}c_{j} K(x_{i}, x_{j}) \geq 0.
$$
The  \textbf{Reproducing Kernel Hilbert Space (RKHS)} of a positive definite kernel $K: X \times X \to \mathbb{R}$ is the Hilbert space $\mathcal{H}_{K} \subset \mathcal{F}(X, \mathbb{R})$, and it satisfies \cite{Steinwart}
\begin{enumerate}
	\item[$(i)$]  $x \in X \to K_{y}(x):= K(x,y) \in \mathcal{H}_{K}$;
	\item[$(ii)$] $\langle K_{x}, K_{y}\rangle = K(x,y) $
	\item[$(iii)$] $\overline{ span\{ K_{y}, \quad y \in X\} }= \mathcal{H}_{K}$.
\end{enumerate} In particular, if $X$ is a Hausdorff  space and $K$ is continuous it holds that $\mathcal{H}_{K} \subset C(X)$.

The following widely known result (usually called  Kernel Mean Embedding) describes how it is possible to define a semi-inner product structure on a subspace of $\mathfrak{M}(X)$ using a continuous positive definite kernel.

\begin{lem}\label{initialextmmddominio}
	If $K: X \times X \to \mathbb{R}$ is a continuous positive definite kernel and  $\mu \in \mathfrak{M}(X)$  with $\sqrt{K(x,x)} \in L^{1}(|\mu|)$ ($\mu \in \mathfrak{M}_{\sqrt{K}}(X)$), then 
	$$
	z \in X \to K_{\mu}(z):=\int_{X} K(x,z)d\mu(x) \in \mathbb{R}
	$$	
	is an element of $\mathcal{H}_{K}$, and if $\eta$ is another measure  with the same conditions as $\mu$, we have that
	$$
	\langle K_{\eta}, K_{\mu}\rangle_{\mathcal{H}_{K}}= \int_{X} \int_{X}K(x,y)d\eta(x)d\mu(y).
	$$
	In particular, $(\eta, \mu) \in \mathfrak{M}_{\sqrt{K}}(X) \times \mathfrak{M}_{\sqrt{K}}(X) \to \langle K_{\eta}, K_{\mu}\rangle_{\mathcal{H}_{K}} $ is an semi-inner product.
\end{lem}

Note that if $K$ is bounded, then $\mathfrak{M}_{\sqrt{K}}(X)=\mathfrak{M}(X)$. The kernel is \textbf{Integrally Strictly Positive Definite (ISPD)}, if $K$ is bounded and the semi-inner product in Lemma \ref{initialextmmddominio} is an inner product. If $K$ is bounded and the semi-inner product is an inner product on the subspace $\mu(X)=0$, we say that $K$ is \textbf{Characteristic}. The interesting aspect of a Characteristic kernel $K$ is that if $P, Q \in \mathfrak{M}(X)$, then
$$
D_{K}(P,Q):=\sqrt{\int_{X} \int_{X}k(x,y)d[P-Q](x)d[P-Q](y) }= \|K_{P} - K_{Q}\|
$$
is a metric on the space of probabilities. The psedometric $D_{K}$ is usually called the \textbf{Maximun Mean Discrepancy (MMD)}.

-\textbf{Conditionally negative definite kernels}

A symmetric kernel $\gamma: X \times X \to \mathbb{R}$ is called \textbf{Conditionally Negative Definite (CND)} if for every finite quantity of distinct points $x_{1}, \ldots, x_{n} \in X$ and scalars $c_{1}, \ldots, c_{n} \in \mathbb{R}$, with the restriction that $\sum_{i=1}^{n}c_{i}=0$, we have that
$$
\sum_{i, j =1}^{n}c_{i}c_{j} \gamma(x_{i}, x_{j}) \leq 0.
$$

The concept of CND kernels is intrinsically related to PD kernels, as a  symmetric kernel $\gamma: X\times X\to \mathbb{R}$ is CND if and only if  for any (or equivalently, for every) $w \in X$ the kernel
\begin{equation}\label{Kgamma}
	K_{\gamma}^{w}(x,y):=\gamma(x,w) + \gamma(w, y) - \gamma(x,y) - \gamma(w,w)
\end{equation}
is positive definite. With this result is possible to explain the relation between CND kernels and Hilbert spaces as  if   $\gamma: X \times X \to \mathbb{R}$ is CND it can  be written as 
\begin{equation}\label{condequa}
	\gamma(x,y)= \|h(x)- h(y)\|_{\mathcal{H}}^{2} + f(x) + f(y)
\end{equation}
where $\mathcal{H}$ is a real Hilbert space, $h: X \to \mathcal{H}$  and $f : X \to \mathbb{R}$. Note that $f(x)= \gamma(x,x)/2$. Another famous relation is that a  symmetric kernel $\gamma: X\times X\to \mathbb{R}$ is CND if and only if  for every $r>0$  the kernel
\begin{equation}\label{schoenmetriccond}
	(x,y) \in X\times X \to  e^{-r\gamma(x,y)}
\end{equation}
is PD.  Those  classical results are crucial for the development of the subject and can be found in Chapter $3$ at \cite{berg0}.

The concept of CND kernel is more general then the concept of an PD kernel, as if $K$ is PD then $-K$ is CND. It is also possible to define semi-inner products on subspaces of $\mathfrak{M}(X)$ using CND kernels, which is described in the next  Lemma.

\begin{lem}\label{estimativa} Let   $\gamma: X \times X \to \mathbb{R}$ be a continuous CND kernel such that $\gamma(x,x)$ is a bounded function and   $\mu \in \mathfrak{M}(X)$. The following assertions are equivalent
	\begin{enumerate}
		\item[$(i)$] $\gamma \in L^{1}(|\mu|\times |\mu|)$;
		\item[$(ii)$] The function $x \in X \to \gamma(x,z) \in L^{1}(|\mu|)$ for some $z \in X$;
		\item[$(iii)$] The function $x \in X \to \gamma(x,z) \in L^{1}(|\mu|)$ for every $z \in X$.
	\end{enumerate}
	Further,  the set of measures that satisfies these relations is a vector space. In particular, consider the vector space
	$$
	\mathfrak{M}_{1}(X; \gamma):= \{ \eta \in \mathfrak{M}(X), \quad   \gamma(x,y)\in L^{1}(|\eta|\times |\eta|) \text{ and } \eta(X)=0 \},
	$$
	then the function
	$$
	(\mu, \nu ) \in  \mathfrak{M}_{1}(X; \gamma) \times \mathfrak{M}_{1}(X; \gamma) \to I(\mu, \nu)_{\gamma}:=\int_{X} \int_{X} -\gamma(x,y)d\mu(x)d\nu(y)
	$$
	defines an semi-inner product on  $\mathfrak{M}_{1}(X; \gamma)$.\end{lem}

When the semi inner product on the previous Lemma is an inner product,  we say that the kernel  $\gamma$ is   \textbf{CND-Characteristic}. The interesting aspect of a CND-Characteristic kernel $\gamma$ is that 
$$
E_{\gamma}(P,Q):=\sqrt{-\int_{X} \int_{X}\gamma(x,y)d[P-Q](x)d[P-Q](y) }
$$
is a metric on the space of probabilities that satisfies any of the 3 equivalent conditions in the first part of Lemma \ref{estimativa}.  The pseudometric $E_{\gamma}$ is usually called the \textbf{Energy distance}. A proof of Lemma \ref{estimativa} can be found in  Section $3$ in \cite{energybernstein}.

The characterization of  the continuous CND radial kernels in all Euclidean spaces was proved in \cite{neuschoen}, and is the following:  
\begin{theorem}\label{reprcondneg} Let $\psi :[0, \infty)\to \mathbb{R}$ be a continuous function. The following conditions are equivalent
	\begin{enumerate}
		\item[$(i)$] The kernel
		$$
		(x,y) \in \mathbb{R}^{d}\times \mathbb{R}^{d} \to \psi(\|x-y\|^{2}) \in \mathbb{R}
		$$
		is CND for every $d \in \mathbb{N}$. 
		\item[$(ii)$] The function $\psi$ can be represented as
		$$
		\psi(t)=\psi(0)+ \int_{[0,\infty)}(1-e^{-rt})\frac{1+r}{r}d\sigma(r),
		$$
		for all $t \geq 0$, where  $\sigma$ is a nonnegative  measure on $\mathfrak{M}([0,\infty))$. The  representation is  unique.
		\item[$(iii)$] The function $\psi \in C^{\infty}(0,\infty))$ and $\psi^{(1)}$ is completely monotone, that is, $(-1)^{n}\psi^{(n+1)}(t) \geq 0$, for every $n\in \mathbb{Z}_{+}$ and $t>0$.
	\end{enumerate}
\end{theorem}

A continuous function $\psi:[0, \infty)\to \mathbb{R}$ that satisfies the relation $(iii)$ in  Theorem \ref{reprcondneg}  is called a \textbf{Bernstein function} (we do not need to assume that Bernstein functions are nonnegative), and the same theorem provides a representation for it. For more information on Bernstein functions see \cite{bers}. The value of the function $(1-e^{-rt})(1+r)/r$ at  $r=0$ is defined as the limit of $r \to 0$, that is, it value is $t$. Usually, the integral on the set $[0, \infty)$ is separated in the integral at $\{0\}$ plus the integral on the set $(0, \infty)$, we do not present it in this way as   the notation and terminology of the proofs in Section \ref{Bernstein  functions of two variables} are considerably simplified by using this simple modification.    

In \cite{gaussinfi} it was proved the following examples of ISPD kernels using Equation \ref{schoenmetriccond}.

\begin{theorem}\label{gengaussian} Let $\gamma: X \times X \to \mathbb{R}$ be a continuous CND kernel. Then the kernel 
	$$
	(x,x^{\prime}) \in X \times X \to e^{-\gamma(x,x^{\prime})} \in \mathbb{R}
	$$
	is ISPD if and only if there  exists  $\inf_{x \in X} \gamma(x,x)$ and the following relation holds
	\begin{equation}\label{metrizable}	
		\{(x,x^{\prime}) \in X\times X, \quad 2\gamma(x,y)= \gamma(x,x) + \gamma(y,y) \}= \{(x,x), \quad x \in X\}.
	\end{equation}	
\end{theorem}	
A CND kernel $\gamma$ that satisfies Equation \ref{metrizable}  is called \textbf{metrizable}, as the function 
$$
D_{\gamma}(x, x^{\prime}):= \sqrt{2\gamma(x,y) - \gamma(x,x) - \gamma(y,y)}
$$
defines a metric on $X$ if and only if the relation on Equation \ref{metrizable} is satisfied (equivalently, the function $h$ in Equation \ref{condequa} is injective).

Note that if $\psi:[0, \infty) \to \mathbb{R}$ is a Bernstein function and $\gamma: X \times X \to [0, \infty)$ is a CND kernel, then the kernel
\begin{equation}\label{bern+cnd}
	(x,y) \in X \times X \to \psi(\gamma(x,y)) \in \mathbb{R}
\end{equation}
is CND.	As a consequence of Theorem \ref{gengaussian}, Theorem $3.3$ in  \cite{energybernstein}  is proved that this kernel is CND-Characteristic  if and only if $\gamma$ is metrizable and either $\sigma((0, \infty))>0$ or $\sigma(\{0\})>0$ and $\gamma$ is an CND-Characteristic kernel.

More information about the use of PD and CND kernels and metrics in the space of probabilities can be found at \cite{MethodsofDistances2013,bergforst,berlinet2011reproducing}

-\textbf{Conditionally positive definite radial kernels of order $\ell$}

In \cite{guo}   it is proved a generalization of Theorem \ref{reprcondneg}, as  it describes, for any $\ell \in \mathbb{Z}_{+}$,  the set of  continuous functions  $\psi:[0, \infty) \to \mathbb{R}$, that  satisfies 
$$ 
\sum_{i,j=1}^{n}c_{i}c_{j} \psi(\|x_{i}-x_{j}\|^{2})\geq 0  
$$
for every points $x_{1}, \ldots, x_{n} \in \mathbb{R}^{d}$ and scalars $c_{1}, \ldots, c_{n}$, under no restriction on the dimension $d$, but such that  
$$ 
\sum_{i=1}^{n}c_{i}p(x_{i})=0, \quad p \text{ is a polynomial of degree less then  } \ell.
$$
If the function $\psi$ satisfies this requirement we say that the kernel $\psi(\|x-y\|^{2})$ is \textbf{Conditionally Positive Definite of Order $\ell$ ($CPD_{\ell}$)}, see Chapter 8 in \cite{wendland}. Note that the case $\ell =1$ we are dealing with constants (polynomials of degree zero), hence $-\psi(\|x-y\|^{2})$ is an  CND kernel. The case $\ell =0$ is of  PD kernels.  	

\begin{theorem}\label{compleelltimes}  The following conditions are equivalent for a continuous functions $\psi: [0, \infty) \to \mathbb{R}$ 
	\begin{enumerate}
		\item[$(i)$] The kernel
		$$
		(x,y) \in \mathbb{R}^{d}\times \mathbb{R}^{d} \to \psi(\|x-y\|^{2}) \in \mathbb{R}
		$$
		is $CPD_{\ell}$  for every $d \in \mathbb{N}$. 
		\item[$(ii)$] The function $\psi$ can be represented as
		$$	
		\psi(t)= \int_{(0,\infty)} (e^{-tr} - e_{\ell}(r)\omega_{\ell}(rt))\frac{1+r^{\ell}}{r^{\ell}} d\sigma(r) + \sum_{k=0}^{\ell}a_{k}t^{k}
		$$
		where $\sigma$ is a nonnegative  measure on $\mathfrak{M}((0,\infty))$, with
		$$
		\omega_{\ell} (s):= \sum_{l=0}^{\ell-1}(-1)^{l}\frac{s^{l}}{l!}, \quad e_{\ell}(s):= e^{-s}\sum_{l=0}^{\ell-1}\frac{s^{l}}{l!}
		$$
		and $a_{k} \in \mathbb{R}$, $(-1)^{\ell}a_{\ell} \geq 0$. The representation is unique.
		\item[$(iii)$] The function $\psi \in C^{\infty}(0,\infty))$ and the function  $(-1)^{\ell}\psi^{(\ell)}$ is a completely monotone, that is, $(-1)^{n+\ell}\psi^{(n+\ell)}(t) \geq 0$, for every $n\in \mathbb{Z}_{+}$ and $t>0$.
	\end{enumerate}
\end{theorem}

A function that satisfies the equivalence on Theorem \ref{compleelltimes}  is called a \textbf{completely monotone function of order $\ell$} ($CM_{\ell}$) . For instance, the functions
\begin{enumerate}
	\item[$i)$] $(-1)^{\ell}t^{a} $;
	\item[$ii)$] $(-1)^{\ell}t^{\ell-1}\log (t)$;
	\item[$iii)$] $(-1)^{\ell}( c+t )^{a}$;
	\item[$iv)$] $e^{-rt}$, 
\end{enumerate}
are elements of $CM_{\ell}$, for $\ell-1 < a \leq \ell$  and  $c>0$.

On \cite{energybernstein}, we used functions on $CM_{\ell}$ to define metrics on a subspace of the  space of probabilities that have the same vector mean ($\ell=2$), the same vector mean and same covariance matrix $(\ell=3)$, and so on. It was also presented how those ideas can be generalized using CND kernels. 

In this text we only use the case $\ell=2$, as in Corolllary \ref{2times} we connect those functions with a certain type of radial PDI kernels. In order to obtain a better description of which probabilities with the same marginals we can compare,  we often include the hypothesis that $f\in CM_{2} \cap C^{1}([0, \infty))$. We do so because we have a simpler representation of those functions compared to Theorem \ref{compleelltimes}.

\begin{lem}\label{compleelltimessimples}  The following conditions are equivalent for a continuous functions $\psi: [0, \infty) \to \mathbb{R}$ 
	\begin{enumerate}
		\item[$(i)$] The function $\psi$ can be represented as
		$$	
		\psi(t)= \int_{(0,\infty)} (e^{-tr} - \omega_{\ell}(rt))\frac{1+r}{r^{\ell}} d\sigma(r) + \sum_{k=0}^{\ell}a_{k}t^{k}
		$$
		where $\sigma$ is a nonnegative  measure on $\mathfrak{M}((0,\infty))$, with
		$$
		\omega_{\ell} (s):= \sum_{l=0}^{\ell-1}(-1)^{l}\frac{s^{l}}{l!}, \quad e_{\ell}(s):= e^{-s}\sum_{l=0}^{\ell-1}\frac{s^{l}}{l!}
		$$
		and $a_{k} \in \mathbb{R}$, $(-1)^{\ell}a_{\ell} \geq 0$. The representation is unique.
		\item[$(ii)$] The function $\psi \in C^{\infty}(0,\infty)) \cap C^{\ell-1}([0, \infty))$ and the function  $(-1)^{\ell}\psi^{(\ell)}$ is a completely monotone, that is, $(-1)^{n+\ell}\psi^{(n+\ell)}(t) \geq 0$, for every $n\in \mathbb{Z}_{+}$ and $t>0$.
	\end{enumerate}

\end{lem} 	

A proof of this fact can be found in Section $4$ at \cite{energybernstein}, where this representation was used for similar purposes. As proved in Section $6$ in \cite{energybernstein}, for any $\ell \in \mathbb{N}$, the function $E_{\ell}(s)= (-1)^{\ell}(e^{-s} -\omega_{\ell}(s))$  is nonnegative, increasing, convex  and there are $M_{1}, M_{2} > 0$ for which
\begin{equation}\label{ineqcm2c1}
	M_{1}\min(t^{\ell}, t^{\ell-1}) \leq E_{\ell}(rt)\frac{1+r}{r^{\ell}} \leq M_{2}(1 + t^{\ell}), \quad r,t \geq 0.
\end{equation}
In particular, for any $s \geq 1$ we have that  $E_{\ell}(rs) \leq M_{2}s^{\ell}E_{\ell}(r)/M_{1}$, for every $t \geq 0$. These properties are essential for the equivalence relations in Lemma \ref{integmargincompl2}.

A generalization of Lemma \ref{estimativa}  to functions in $CM_{\ell} \cap C^{\ell -1}([0, \infty))$ is possible  and can be found in Section $4$ of \cite{energybernstein}.

\begin{lem}\label{estimativa2} Let $\gamma : X \times X \to [0, \infty)$ be a continuous CND kernel such that $\gamma(x,x)$ is a bounded function, $\mu \in \mathfrak{M}(X)$ and  $\psi \in CM_{\ell} \cap C^{\ell -1}([0, \infty))$. The following assertions are equivalent
	\begin{enumerate}
		\item[$(i)$] $\psi(\gamma) \in L^{1}(|\mu|\times |\mu|)$;
		\item[$(ii)$] The function $x \to \psi(\gamma(x,z)) \in L^{1}(|\mu|)$ for some $z \in X$;
		\item[$(iii)$] The function $x \to \psi(\gamma(x,z)) \in L^{1}(|\mu|)$ for every  $z \in X$.	
	\end{enumerate}	
	Further, the set of measures that satisfies these relations is a vector space.	
\end{lem}

\section{\textbf{Positive definite independent kernels}}\label{positive definite independent kernels}

\begin{defn}\label{PDI}Let $X$ and $Y$ be non empty sets. We say that a symmetric kernel $\mathfrak{I}: (X\times Y) \times (X\times Y) \to \mathbb{R}$ is a  \textbf{Positive Definite Independent Kernel (PDI)}  if for every finite quantity of distinct points $x_{1}, \ldots, x_{n} \in X$, $y_{1}, \ldots , y_{m} \in Y$  and real scalars $c_{i, k}$, with the restrictions
	$$
	\sum_{i=1}^{n}c_{i,k}=0 , \quad \sum_{l=1}^{m}c_{j, l}=0, 
	$$
	for every $1\leq k \leq m$, $1\leq j \leq n$, it satisfies
	$$
	\sum_{i,j=1}^{n}\sum_{k,l=1}^{m}c_{i, k}c_{j, l}\mathfrak{I}((x_{i},y_{k}),(x_{j}, y_{l})) \geq 0.
	$$
	The kernel $\mathfrak{I}$ is called \textbf{Strictly Positive Definite Independent Kernel} (SPDI), if the previous inequality is an equality only when  all scalars $c_{i, k}$ are zero.
\end{defn}

This definition is inspired by Equation \ref{objective} on the case that $\pi, \pi^{\prime}$ are discrete measures, because there exists  points $x_{1}, \ldots, x_{n} \in X$, $y_{1}, \ldots , y_{m} \in Y$,  real scalars $c_{i, k}$ that satisfies the necessary restrictions of Definition \ref{PDI} and $M \in \mathbb{R}$ for which
$$
\pi- \pi^{\prime}= M\left ( \sum_{i=1}^{n} \sum_{k=1}^{m}c_{i,k}\delta_{(x_{i}, y_{k})} \right ).
$$

The most important example of an PDI kernel is the fact that the Kronecker product of CND kernels is PDI.  Indeed, let $\gamma: X \times X\to \mathbb{R}$ and $\varsigma: Y \times Y \to \mathbb{R}$ be nonzero  CND kernels and consider its  Kronecker product
$$
(\gamma\otimes \varsigma)((x,y)(x^{\prime}, y^{\prime})):= \gamma(x,x^{\prime})\varsigma(y,y^{\prime}).
$$
Let $x_{1}, \ldots, x_{n} \in X$, $y_{1}, \ldots , y_{m} \in Y$  and real scalars $c_{i, k}$ with the necessary restrictions, then by Equation \ref{condequa} we have that
\begin{align*}
	\sum_{i,j=1}^{n}\sum_{k,l=1}^{m}&c_{i, k}c_{j, l}(\gamma\otimes \varsigma)((x_{i},y_{k}),(x_{j},y_{l}))=\sum_{i,j=1}^{n}\sum_{k,l=1}^{m}c_{i, k}c_{j, l}\gamma(x_{i},x_{j})\varsigma(y_{k},y_{l})\\
	&=\sum_{i,j=1}^{n}\gamma(x_{i},x_{j}) \left [\sum_{k,l=1}^{m}c_{i, k}c_{j, l}(\|h_{\varsigma}(y_{k}) - h_{\varsigma}(y_{l})\|^{2} + \varsigma(y_{k},y_{k})/2 +\varsigma(y_{l},y_{l}))/2\right ]\\
	&= \sum_{i,j=1}^{n}\gamma(x_{i},x_{j}) \left [-2\sum_{k,l=1}^{m}c_{i, k}c_{j, l} \langle h_{\varsigma}(y_{k}), h_{\varsigma}(y_{l})\rangle \right ]\\
	&= -2\sum_{k,l=1}^{m}\langle h_{\varsigma}(y_{k}), h_{\varsigma}(y_{l})\rangle \left [\sum_{i,j=1}^{n} c_{i, k}c_{j, l}\gamma(x_{i},x_{j}) \right ]\\
	&= 4\sum_{i,j=1}^{n}\sum_{k,l=1}^{m}c_{i, k}c_{j, l}\langle h_{\gamma}(x_{i}), h_{\gamma}(x_{j})\rangle\langle h_{\varsigma}(y_{k}), h_{\varsigma}(y_{l})\rangle \geq 0.
\end{align*}

Although this is a simple example, for some families of PDI kernels that satisfies some simmetry relation (like radiality on both coordinates, that we present on   Section \ref{Bernstein  functions of two variables}) are a convex combination (on an integration sense) of kernels of this type. Also, note that the Kronecker product of positive definite kernels is an PDI kernel.

If  $\mathfrak{I}: (X\times Y) \times (X\times Y) \to \mathbb{R}$ is an PDI kernel and $y_{1}, \ldots, y_{m} \in Y$, $e_{1}, \ldots , e_{m} \in \mathbb{R}$ satisfies $\sum_{k=1}^{m}e_{k}=0$, then the kernel
\begin{equation}\label{CNDprojgeral}
	(x,x^{\prime}) \in X\times X \to -\sum_{k,l=1}^{m}e_{k}e_{j}\mathfrak{I}((x,y_{k}),(x^{\prime}, y_{l})), \quad 
\end{equation}
is CND on $X$ (a similar result replacing $X$ with $Y$ is also possible). Indeed, just note that if $\sum_{i=1}^{n}d_{i}=0$, then $c_{i,k}:=d_{i}e_{k}$ satisfies the necessary restrictions in the definition of PDI kernel.

The next result connects the definition of PDI kernels with positive definite kernels.

\begin{lem}\label{PDItoPD} Let $\mathfrak{I}: (X\times Y) \times (X\times Y) \to \mathbb{R}$  be a symmetric kernel  and $x_{0} \in X$, $y_{0} \in Y$.  The kernel $K^{\mathfrak{I}  }: (X\times Y) \times (X\times Y) \to \mathbb{R}$ defined as 
	\begin{align*}
		K^{\mathfrak{I}  }& ((x,y),(x^{\prime}, y^{\prime})):=\\
		& \left [ \right . \mathfrak{I}((x,y),(x^{\prime}, y^{\prime})) - \mathfrak{I}((x_{0},y),(x^{\prime}, y^{\prime}))-\mathfrak{I}((x,y_{0}),(x^{\prime}, y^{\prime}))+\mathfrak{I}((x_{0},y_{0}),(x^{\prime}, y^{\prime}))\\
		&-\mathfrak{I}((x,y),(x_{0}, y^{\prime}))+\mathfrak{I}((x_{0},y),(x_{0}, y^{\prime}))+\mathfrak{I}((x,y_{0}),(x_{0}, y^{\prime}))-\mathfrak{I}((x_{0},y_{0}),(x_{0}, y^{\prime}))\\
		&-\mathfrak{I}((x,y),(x^{\prime}, y_{0}))+\mathfrak{I}((x_{0},y),(x^{\prime}, y_{0}))+\mathfrak{I}((x,y_{0}),(x^{\prime}, y_{0}))-\mathfrak{I}((x_{0},y_{0}),(x^{\prime}, y_{0}))\\
		&+\mathfrak{I}((x,y),(x_{0}, y_{0}))-\mathfrak{I}((x_{0},y),(x_{0}, y_{0}))-\mathfrak{I}((x,y_{0}),(x_{0}, y_{0}))+\mathfrak{I}((x_{0},y_{0}),(x_{0}, y_{0}))\left . \right ]
	\end{align*}
	is PD if and only if $\mathfrak{I}$ is PDI.
\end{lem}

\begin{lem}\label{PDIsimpli} Let $\mathfrak{I}: (X\times Y) \times (X\times Y) \to \mathbb{R}$  be a symmetric kernel.  The kernel $\mathfrak{I}^{\prime}: (X\times Y) \times (X\times Y) \to \mathbb{R}$ given by
	\begin{align*}
		\mathfrak{I}^{\prime} &((x,y),(x^{\prime}, y^{\prime})):= \mathfrak{I}((x,y),(x^{\prime}, y^{\prime}))\\
		&- \mathfrak{I}((x,y),(x, y^{\prime}))/2 - \mathfrak{I}((x^{\prime},y),(x^{\prime}, y^{\prime}))/2 - \mathfrak{I}((x,y),(x^{\prime}, y))/2 - \mathfrak{I}((x,y^{\prime}),(x^{\prime}, y^{\prime}))/2\\
		& + \mathfrak{I}((x,y),(x, y))/4+ \mathfrak{I}((x,y^{\prime}),(x, y^{\prime}))/4   + \mathfrak{I}((x^{\prime},y),(x^{\prime}, y))/4+ \mathfrak{I}((x^{\prime},y^{\prime}),(x^{\prime}, y^{\prime}))/4
	\end{align*}
	is PDI if and only if $\mathfrak{I}$ is PDI.
\end{lem}

Because of Lemma \ref{PDIsimpli}, we may suppose that the  \textbf{Projection Kernels} of an PDI kernel $\mathfrak{I}$ are zero, that is
\begin{equation}\label{eqmargkernels0}
	\mathfrak{I}((x,y),(x,y^{\prime}))=0 \text{ and } \mathfrak{I}((x,y),(x^{\prime},y))=0
\end{equation}
for every $x,x^{\prime} \in X$ and $y,y^{\prime} \in Y$. This condition will be more important compared to demanding that an CND kernel $\gamma$ satisfies $\gamma(x,x)=0$, because it greatly simplifies the  integrability conditions on Lemma \ref{integmargin} and Theorem \ref{integmargincor}. In some examples in Section \ref{Bernstein  functions of two variables} and Section \ref{Examples of PDI-Characteristic kernels} will be convenient to withdraw this condition.

Because of the previous Lemmas we can obtain a geometrical interpretation of PDI kernels by using the properties of RKHS.

\begin{theorem}\label{PDIgeometryrkhs}  Let $\mathfrak{I}: (X\times Y) \times (X\times Y) \to \mathbb{R}$  be an PDI kernel whose projection kernels are zero, fixed elements $x_{0} \in X$, $y_{0}\in X$ and the  positive definite kernel $K^{\mathfrak{I}}: (X\times Y) \times (X\times Y) \to \mathbb{R}$ defined  on Lemma \ref{PDItoPD}. The following equality is satisfied:
	$$
	2\mathfrak{I}((x,y)(x^{\prime}, y^{\prime})) + 2\mathfrak{I}((x,y^{\prime})(x^{\prime}, y)) = \| K^{\mathfrak{I}  }_{x,y} + K^{\mathfrak{I}  }_{x^{\prime},y^{\prime}} -K^{\mathfrak{I}  }_{x,y^{\prime}}-K^{\mathfrak{I}  }_{x^{\prime},y} \|_{\mathcal{H}_{K^{\mathfrak{I}  }}} ^{2}. 
	$$
\end{theorem}

In  particular, if the kernel $\mathfrak{I}$ on Theorem \ref{PDIgeometryrkhs} also satisfies the symmetry relation
\begin{equation}\label{2symmetry}
	\mathfrak{I}((x,y),(x^{\prime}, y^\prime))= \mathfrak{I}((x,y^{\prime}),(x^{\prime}, y)), \quad x,x^{\prime} \in X, y,y^{\prime} \in Y
\end{equation}
then  $4\mathfrak{I}((x,y),(x^{\prime}, y^\prime))= \|K^{\mathfrak{I}  }_{x,y} + K^{\mathfrak{I}  }_{x^{\prime},y^{\prime}} -K^{\mathfrak{I}  }_{x,y^{\prime}}-K^{\mathfrak{I}  }_{x^{\prime},y}\|_{\mathcal{H}_{K^{\mathfrak{I}  }}}^{2}$.    

A symmetric kernel that also satisfies the symmetry on Equation \ref{2symmetry} will be called  \textbf{2-Symmetric}. By the relation in Equation \ref{CNDprojgeral}, if  $\mathfrak{I}$ is an PDI kernel whose projection kernels are zero   and is 2-symmetric, then the kernels
\begin{equation}\label{eqcndker}
	\mathfrak{I}_{x,x^{\prime}}(y,y^{\prime}):=\mathfrak{I}((x,y),(x^{\prime}, y^{\prime})), \quad  \mathfrak{I}_{y,y^{\prime}}(x,x^{\prime}):=\mathfrak{I}((x,y),(x^{\prime}, y^{\prime}))
\end{equation}
are CND on $Y$ and $X$ respectively, which by Equation \ref{condequa} implies that  the kernel $\mathfrak{I}$ is a nonnegative function.

A question that comes to mind is if there exists a relation that connects PDI kernels with  PD/CND kernels in a similar way as  Equation \ref{schoenmetriccond}. Below we present a brief argument that such relation does not exist. 

\begin{lem} \label{badPDItoPD} Let $\mathfrak{I}: (X\times Y)\times (X\times Y) \to \mathbb{R} $ be an PDI kernel that is  $2$-symmetric  and whose projection kernels are zero and also a function $f: [0, \infty) \to \mathbb{R}$. The kernel
	$$
	f(\mathfrak{I}((x,y),(x^{\prime}, y^{\prime}))), \quad x,x^{\prime} \in X, \quad y,y^{\prime} \in Y
	$$
	is positive definite if and only if this is a constant kernel.
\end{lem}

Indeed, let $x\neq x^{\prime}$ and $y \neq y^{\prime}$, then if the kernel is positive definite the interpolation matrix  at the points $(x,y), (x,y^{\prime}), (x^{\prime}, y), (x^{\prime},y^{\prime})$ is 
$$
A:=\begin{pmatrix}
	f(0) & f(0) & f(0) &  f(c)\\
	f(0) & f(0) & f(c) & f(0)\\
	f(0) &  f(c) & f(0)& f(0)\\
	f(c) & f(0) &  f(0)& f(0) 
\end{pmatrix}
$$
and must be positive semidefinite, where $c = \mathfrak{I}((x,y),(x^{\prime}, y^{\prime}))$. However, for $v_{1}= (-1,1,-1,1)$ and $v_{2}= (-1,1,1,-1)$ we have that $4(f(0) - f(c )) = \langle A v_{1}, v_{1}\rangle  \geq 0 $ and $4(f(c) - f(0 )) = \langle A v_{2}, v_{2}\rangle  \geq 0 $, then $f(0) =f(c)$.



For the rest of this Section we will be concerned on how  PDI kernels can used to define  pseudometrics on the space of probabilities with fixed marginals (in particular, to analyse independence of probabilities) in a similar way that CND kernels can be used to compare probabilities as done in \cite{energybernstein}. For this we assume that $X$ and $Y$ are Hausdorff spaces. In order to describe  the set of measures that we will be able to compare, we separate the analysis in two stages. First, we describe which marginals are possible to work with  (Lemma \ref{integmargin} and Lemma  \ref{integmarginkroeprod}),  and then we describe which measures we can compare, with the restriction that its marginals are well behaved (Corollary \ref{integmargincor} and Corollary \ref{integmarginkroeprodcor}).

The major inequality we use to prove such results is a direct consequence that the kernels in  Equation \ref{eqcndker} are CND,  as they implies that for every $x,x^{\prime} , z \in X$ and $y, y^{\prime}, w \in Y$, we have that 
$$
\sqrt{\mathfrak{I}((x,y),(x^{\prime},y^{\prime}))}\leq \sqrt{\mathfrak{I}((x,y),(z,y^{\prime}))} +\sqrt{\mathfrak{I}((z,y),(x^{\prime},y^{\prime}))},
$$
$$
\sqrt{\mathfrak{I}((x,y),(z,y^{\prime}))}\leq \sqrt{\mathfrak{I}((x,y),(z,w))} +\sqrt{\mathfrak{I}((x,w),(z,y^{\prime}))},
$$
$$
\sqrt{\mathfrak{I}((z,y),(x^{\prime},y^{\prime}))}\leq \sqrt{\mathfrak{I}((z,y),(x^{\prime},w))} +\sqrt{\mathfrak{I}((z,w),(x^{\prime},y^{\prime}))},
$$
hence
\begin{equation}\label{4rootsinequality}
	\begin{split}
		\sqrt{\mathfrak{I}((x,y),(x^{\prime},y^{\prime}))} & \leq \sqrt{\mathfrak{I}((x,y),(z,w))} +\sqrt{\mathfrak{I}((x,y^{\prime}),(z,w))}\\ 
		& \quad  +\sqrt{\mathfrak{I}((x^{\prime},y),(z,w))} +\sqrt{\mathfrak{I}((x^{\prime},y^{\prime}),(z,w))}.
	\end{split}
\end{equation}

\begin{lem}\label{integmargin} Let $\mathfrak{I}$ be a continuous, $2-$symmetric PDI kernel on $X\times Y$ whose projection kernels are zero. For nonnegative measures $\mu \in \mathfrak{M}(X)\setminus{\{0\}}$ and $\nu \in \mathfrak{M}(Y)\setminus{\{0\}}$, the following conditions are equivalent:
	\begin{enumerate}
		\item [$(i)$] $\mathfrak{I} \in L^{1}((\mu \times \nu)\times (\mu \times \nu)  )$;
		\item [$(ii)$]The function  $ \mathfrak{I}((\cdot,\cdot),(x^{\prime},y^{\prime}))  \in L^{1}(\mu \times \nu)$ for some $(x^{\prime},y^{\prime}) \in X\times Y$;
		\item [$(iii)$] There exists sets $X_{\mu, \nu}^{\mathfrak{I}}\times Y_{\mu, \nu}^{\mathfrak{I}} \subset X\times Y$, for which $\mu(X-X_{\mu, \nu}^{\mathfrak{I}})= \nu(Y-Y_{\mu, \nu}^{\mathfrak{I}})=0$ and $	\mathfrak{I}((\cdot, \cdot ),(x^{\prime},y^{\prime}))  \in L^{1}(\mu \times \nu)$ for every fixed $(x^{\prime}, y^{\prime}) \in X_{\mu, \nu}^{\mathfrak{I}} \times Y_{\mu, \nu}^{\mathfrak{I}}$.
	\end{enumerate}
\end{lem}


Unlike relation $(iii)$ in Lemma \ref{estimativa},  relation $(iii)$ on Lemma \ref{integmargin} depends on the measures $\mu$ and $\nu$, and this make it  difficult to analyze a double integral of the type
$$
\int_{X\times Y} \int_{X\times Y} \mathfrak{I}((x,y),(x^{\prime}, y^{\prime}))d\lambda (x,y)d\lambda^{\prime}(x^{\prime}, y^{\prime})
$$
if the marginals of $|\lambda|$ and $|\lambda^{\prime}|$ are too distinct. Initially, we circumvent this problem by  restricting the  analyses of this double integral by demanding additional conditions on the marginals of $|\lambda|$ and $|\lambda^{\prime}|$. On the examples we present at Section \ref{Examples of PDI-Characteristic kernels}, relation $(iii)$ will not depend on the measures $\mu, \nu$.

Given two measures $\sigma_{1}, \sigma_{2} \in \mathcal{M}(Z)$, where $\sigma_{2}$ is nonnegative,  we say that $|\sigma_{1}| \subset \sigma_{2}$ if there exists a constant $c>0$ such that $c\sigma_{2} - |\sigma_{1}|$ is a nonnegative measure. This condition is equivalent to $\sigma_{1}<< \sigma_{2}$ (absolutely continuity) and that the Radon-Nikodym derivative $d\sigma_{1}/ d\sigma_{2}$ is a function in $L^{\infty}(\sigma_{2})$. For a nonnegative measure $\lambda \in \mathfrak{M}(X\times Y)$, the measure $\lambda_{X} \in \mathfrak{M}(X)$ is defined as $\lambda_{X}(A):= \lambda(A\times Y)$. Similar for  $\lambda_{Y}$.

\begin{cor}\label{integmargincor} Let $\mathfrak{I}$ be a continuous, $2-$symmetric PDI kernel on $X\times Y$ whose projection kernels are zero,  a measure $\lambda \in \mathfrak{M}(X\times Y)$ and nonnegative measures $\mu \in \mathfrak{M}(X)\setminus{\{0\}}$ and $\nu \in \mathfrak{M}(Y)\setminus{\{0\}}$. If $\mathfrak{I} \in L^{1}((\mu \times \nu)\times (\mu \times \nu))$ and $|\lambda|_{X}  \subset \mu$ and $|\lambda|_{Y} \subset \nu$ the following conditions are equivalent:
	\begin{enumerate}
		\item [$(i)$] $\mathfrak{I} \in L^{1}(|\lambda|\times |\lambda|  )$;
		\item [$(ii)$]The function  $ \mathfrak{I}((\cdot, \cdot),(x^{\prime},y^{\prime}))  \in L^{1}(|\lambda|)$ for some $(x^{\prime},y^{\prime}) \in X\times Y$;
		\item [$(iii)$]It holds that $\mathfrak{I}((\cdot, \cdot ),(x^{\prime},y^{\prime})) \in  L^{1}( |\lambda|)$ for every fixed $(x^{\prime}, y^{\prime}) \in X_{\mu, \nu}^{\mathfrak{I}} \times Y_{\mu, \nu}^{\mathfrak{I}}$.			
	\end{enumerate}
\end{cor}

As a consequence of Corollary \ref{integmargin} and   Equation \ref{4rootsinequality} $\mathfrak{I}((\cdot,\cdot),(x^{\prime},\cdot)), \mathfrak{I}((\cdot,\cdot),(\cdot,y^{\prime}))$ are elements of $L^{1}( |\lambda|\times |\lambda|)$. The sets $X_{\mu, \nu}^{\mathfrak{I}}$, $Y_{\mu, \nu}^{\mathfrak{I}}$ on Corollary \ref{integmargincor} are the same as the ones in Lemma \ref{integmargin}. The following Theorem is the most important result in this Section.

\begin{theorem}\label{principalmetric} Let $\mathfrak{I}$ be a continuous, $2-$symmetric PDI kernel on $X\times Y$ whose projection kernels are zero and  nonnegative measures $\mu \in \mathfrak{M}(X)\setminus{\{0\}}$ and $\nu \in \mathfrak{M}(Y)\setminus{\{0\}}$ such that $\mathfrak{I} \in L^{1}((\mu \times \nu)\times (\mu \times \nu)  )$. The following subset of $\mathfrak{M}(X\times Y)$ is a vector space
	\begin{align*}
		\mathfrak{M}_{\mu, \nu}(\mathfrak{I}):=\{\lambda \in \mathfrak{M}(X\times Y), & \quad  \mathfrak{I} \in L^{1}(|\lambda|\times |\lambda|), \quad |\lambda|_{X} \subset \mu, |\lambda|_{Y} \subset  \nu,\\ 
		&\quad \lambda(X, \cdot) \text{ and }  \lambda( \cdot, Y) \text{ are the zero measure}   \}.
	\end{align*}
	The real valued function 
	$$
	(\lambda, \lambda^{\prime}) \to \int_{X\times Y} \int_{X\times Y} \mathfrak{I}((x,y),(x^{\prime}, y^{\prime}))d\lambda (x,y)d\lambda^{\prime}(x^{\prime}, y^{\prime}),
	$$
	is well defined for $\lambda, \lambda^{\prime} \in \mathfrak{M}_{\mu, \nu}(\mathfrak{I})$ and is a semi-inner product on it.   
\end{theorem}

On Theorem \ref{principalmetric}, the semi-inner product can be rewritten as
\begin{align*}
	\int_{X\times Y} \int_{X\times Y}&\mathfrak{I}((x,y),(x^{\prime}, y^{\prime}))d\lambda(x, y)d\lambda^{\prime}(x^{\prime}, y^{\prime}) \\
	&= \int_{X\times Y} \int_{X\times Y} K^{\mathfrak{I}}((x,y),(x^{\prime}, y^{\prime}))d\lambda(x,y)d\lambda^{\prime}(x^{\prime}, y^{\prime})= \langle K^{\mathfrak{I}}_{\lambda}, K^{\mathfrak{I}}_{\lambda^{\prime}} \rangle_{\mathcal{H}_{K^{\mathfrak{I}}}},
\end{align*}
where $K^{\mathfrak{I}}_{\lambda}$ is the kernel mean embedding of the kernel $K^{\mathfrak{I}}$ defined on Lemma  \ref{PDItoPD} for a $(x_{0}, y_{0}) \in X_{\mu, \nu}^{\mathfrak{I}}\times Y_{\mu, \nu}^{\mathfrak{I}}$ with the measure $\lambda \in \mathfrak{M}_{\mu, \nu}(\mathfrak{I})$.

\begin{defn} A continuous, $2-$symmetric PDI kernel  $\mathfrak{I}$ on $X\times Y$ whose projection kernels are zero is called \textbf{PDI-Characteristic} if 
	$$
	(\lambda, \lambda^{\prime}) \to \int_{X\times Y} \int_{X\times Y} \mathfrak{I}((x,y),(x^{\prime}, y^{\prime}))d\lambda (x,y)d\lambda^{\prime}(x^{\prime}, y^{\prime}),
	$$
	is an inner product on $\mathfrak{M}_{\mu, \nu}(\mathfrak{I})$ for every  nonnegative measures $\mu \in \mathfrak{M}(X)\setminus{\{0\}}$, $\nu \in \mathfrak{M}(Y)\setminus{\{0\}}$ such that $\mathfrak{I} \in L^{1}((\mu \times \nu)\times (\mu \times \nu))$.  
\end{defn}

If  $\mathfrak{I}$ is an PDI kernel and    $P \in \mathfrak{M}(X)$ and  $Q \in \mathfrak{M}(Y)$ are probabilities such that $\mathfrak{I} \in L^{1}((P \times Q)\times (P \times Q))$, consider the  set 
$$
\varGamma(P,Q)_{\mathfrak{I}}:=\{\lambda \in  \mathfrak{M}(X\times Y), \quad \lambda \text{ is a probability}, \lambda_{X}=P, \lambda_{Y}=Q, \quad \mathfrak{I} \in L^{1}(|\lambda| \times |\lambda|) \},
$$
then
$$
(\lambda, \lambda^{\prime}) \to \sqrt{ \int_{X\times Y} \int_{X\times Y} \mathfrak{I}((x,y),(x^{\prime}, y^{\prime}))d[\lambda-\lambda^{\prime} ](x,y)d[\lambda - \lambda^{\prime}](x^{\prime}, y^{\prime})  }   
$$
is a pseudometric  for $\lambda, \lambda^{\prime} \in \varGamma(P,Q)_{\mathfrak{I}}$. It is a metric if and only if the kernel $\mathfrak{I}$ is PDI-Characteristic. As a consequence, if $\mathfrak{I}$ is PDI-Characteristic a probability $\lambda \in \varGamma(P,Q)_{\mathfrak{I}}$ satisfy
$$
\int_{X\times Y} \int_{X\times Y} \mathfrak{I}((x,y),(x^{\prime}, y^{\prime}))d[\lambda-P\otimes Q ](x,y)d[\lambda - P\otimes Q](x^{\prime}, y^{\prime})=0 
$$
if and only if $\lambda =P\otimes Q$.

Because of relation $(iii)$ in Lemma \ref{integmargin} does not seem to hold for every $(x^{\prime}, y^{\prime}) \in X\times Y$ for all  kernels $\mathfrak{I}$, it seems difficult to characterize the set of nonnegative measures $\mu \in \mathfrak{M}(X)\setminus{\{0\}}$ and $\nu \in \mathfrak{M}(Y)\setminus{\{0\}}$ such that $\mathfrak{I} \in L^{1}((\mu \times \nu)\times (\mu \times \nu)  )$. However, we do not have an example of an PDI kernel $\mathfrak{I}$ and nonnegative measures $\mu, \nu$ for which $X_{\mu, \nu}^{\mathfrak{I}} \times Y_{\mu, \nu}^{\mathfrak{I}} \neq X\times Y$ and $\mathfrak{M}_{\mu, \nu}(\mathfrak{I})\neq \{0\}$.

One exception is when $\mathfrak{I}$ is the Kronecker product of CND kernels, which we give a proof below. The setting of product of CND kernels will also allow us to obtain results on a more general setting as we do not have to assume that the projection kernels of the  PDI kernel related to it are zero.


For a Hausdorff set  $Z$ we define the set $\delta(Z):= \{ c\delta_{z},\quad  c \in \mathbb{R}, z\in Z \}$. 

\begin{lem}\label{integmarginkroeprod} Let $\gamma: X \times X \to [0, \infty)$,  $\varsigma: Y \times Y \to [0, \infty)$  be continuous,  CND  metrizable kernels with bounded diagonal. Then,    nonnegative measures $\mu \in \mathfrak{M}(X)\setminus{\delta(X)}$, $ \nu  \in \mathfrak{M}(Y)\setminus{\delta(Y)}$ satisfies  
	$$
	\gamma\otimes\varsigma \in L^{1}((\mu\times \nu)\times (\mu \times \nu))
	$$ 
	if and only if 
	$$
	\gamma(\cdot ,x^{\prime}) \in L^{1}(\mu), \quad \varsigma(\cdot  , y^{\prime})  \in L^{1}(\mu)
	$$
	for every $(x^{\prime}, y^{\prime}) \in X \times Y$. 	\end{lem}

On Lemma \ref{integmarginkroeprod} we must exclude measures in $\delta(X)$ and $\delta(Y)$ in order to avoid problematic integration. For instance, suppose that $X=Y= \mathbb{R}^{d}$, $\gamma(z,w)= \varsigma(z,w)=\|z-w\|$ and  $\nu$ is a nonnegative measure for which $z \to \|z\|$ is not in  $L^{1}(\nu\times \nu)$. Then, we have that $\gamma \otimes \varsigma \in L^{1}((\delta_{0} \times \nu)\times (\delta_{0} \times \nu))$ (because $\gamma \otimes \varsigma$ is equal to zero $\delta_{0} \times \nu$ almost everywhere) but $(x,y) \to  \|x-x^{\prime}\| \| y - y^{\prime}\|  \in L^{1}(\delta_{0} \times \nu)$ only for $(x^{\prime}, y^{\prime})\in \{0\}\times\mathbb{R}^{d}$. This problem does not affect the definition of PDI-Characteristic kernels, as the vector spaces  $\mathfrak{M}_{\delta_{x}\times \nu}(\mathfrak{I})$ and $\mathfrak{M}_{ \mu \times \delta_{y}}(\mathfrak{I})$ are always equal to $\{0\}$. In 

\begin{cor}\label{integmarginkroeprodcor} Let $\gamma: X \times X \to [0, \infty)$,  $\varsigma: Y \times Y \to [0, \infty)$  be continuous,  CND metrizable kernels with bounded diagonal. Then,    a  measure $\lambda \in \mathfrak{M}(X \times Y)$   such that  $|\lambda|_{X}$, $|\lambda|_{Y}$ are not degenerate and $\gamma  \in L^{1}(|\lambda|_{X} \times |\lambda|_{X} )$, $\varsigma  \in L^{1}(|\lambda|_{Y} \times |\lambda|_{Y} )$ satisfies $ \gamma\otimes\varsigma \in L^{1}(|\lambda| \times |\lambda|) $  if and only if
	$$
	\gamma(\cdot,x^{\prime})\varsigma(\cdot, y^{\prime})  \in L^{1}(|\lambda|)
	$$
	for every $(x^{\prime}, y^{\prime}) \in X \times Y$. The following subset of $\mathfrak{M}(X\times Y)$ is a vector space
	\begin{align*}
		\mathfrak{M}(\gamma \otimes \varsigma): =\{\lambda \in &  \mathfrak{M}(X\times Y) ,  \quad  \gamma \otimes \varsigma \in L^{1}(|\lambda|\times |\lambda|), \quad \gamma  \in L^{1}(|\lambda|_{X} \times |\lambda|_{X} ),\\ 
		&\varsigma  \in L^{1}(|\lambda|_{Y} \times |\lambda|_{Y} ), \quad \lambda(X, \cdot) \text{ and }  \lambda( \cdot, Y) \text{ are the zero measure}   \}.
	\end{align*}
	The real valued function 
	$$
	(\lambda, \lambda^{\prime}) \to \int_{X\times Y} \int_{X\times Y} \gamma(x,x^{\prime}) \varsigma(y,  y^{\prime})d\lambda (x,y)d\lambda^{\prime}(x^{\prime}, y^{\prime}),
	$$
	is well defined for $\lambda, \lambda^{\prime} \in \mathfrak{M}(\gamma \otimes \varsigma)$ and is a semi-inner product on it. It is an inner product if and only if $\gamma$ and $\varsigma$ are CND-Characteristic. 
\end{cor}

The concept of PDI-Characteristic kernel is revisited in Section \ref{Examples of PDI-Characteristic kernels}, where Lemma  \ref{integmarginkroeprod} and Corollary \ref{integmarginkroeprodcor} are considered special cases. A proof of those results can be found at \cite{sejdinovic2013equivalence}, under the additional assumption that $\gamma$ and $\varsigma$ are zero at the diagonal.   

Corollary \ref{integmarginkroeprodcor}  was generalized in  \cite{bottcher2019distance} to a product of $n$ CND-Characteristic kernels. We expect that PDI kernels can be generalized to a product of several spaces, however, it seems that the number of variables grows exponentially with the number of spaces. We plan to analyze this question in the future.

\section{\textbf{Bernstein  functions of two variables} }\label{Bernstein  functions of two variables}

The aim of this section is to characterize radial PDI kernels on all Euclidean spaces in the same sense that Schoenberg characterized positive definite radial kernels in \cite{schoenbradial}  and CND radial kernels in \cite{neuschoen} (Theorem \ref{reprcondneg}). We also connect those kernels with the set of completely monotone functions of two variables.

Next Theorem is   an intermediate result for our purposes, it is based on Theorem $2.3$  in \cite{unifying} where the focus was on positive definite isotropic kernels on real spheres. We emphasize that the proof is not constructive, but instead is   based on Theorem \ref{reprcondneg}  and the fact that   the representation is unique.

\begin{theorem}\label{basicradialuni}
	Let $f: X \times X \times [0, \infty) \to \mathbb{R}$ be a  function such that $f(u,v, \cdot )$ is continuous on $ [0,\infty)$ for every $u,v \in X$. The kernel
	$$
	f(u,v, \|x-y\| ), \quad u, v \in X, \quad x,y \in  \mathbb{R}^{d}
	$$ is PDI  for every  $d\in \mathbb{N}$ and its projection kernels are zero if and only if the kernel can be represented as
	$$
	f(u,v, \|x-y\|)= \int_{[0,\infty)}\frac{(1-e^{-r\|x-y\|^{2}})}{r}(1 +r)d\sigma_{u,v}(r) 
	$$
	where for every $A \in \mathscr{B}([0, \infty))$ the kernel $\sigma_{u,v}(A)$ is CND in $X$ and $\sigma_{u,u}$ is the zero measure for every $u \in X$. The representation  is unique.
\end{theorem}

A function $h:(0, \infty)\times (0, \infty) \to \mathbb{R}$ is  \textbf{completely monotone with two variables} if $h \in C^{\infty}((0, \infty)^{2})$ and $(-1)^{|\alpha|}\partial^{\alpha}h(t) \geq 0$, for every $\alpha \in \mathbb{Z}_{+}^{2}$ and $t \in (0,\infty)^{2}$. Similar to the Hausdorff-Bernstein-Widder Theorem on completely monotone functions (one variable), the following equivalence holds, Section $4.2$ in \cite{bochnerlivro}:

\begin{theorem}\label{Bochnercomplsev} A  function $g:(0, \infty)^{2} \to \mathbb{R}$ is  completely monotone  with two variables if and only if it can be represented as
	$$
	h(t_{1},t_{2})=\int_{[0,\infty)^{2}}e^{-r_{1}t_{1} - r_{2}t_{2} }d\sigma(r_{1},r_{2}), \quad t_{1}, t_{2} \in (0, \infty) 
	$$
	where $\sigma$ is a Radon Borel nonnegative measure (possibly unbounded) on $[0,\infty)^{2}$ for which the previous integrals are well defined. Further, the representation is unique.
\end{theorem}

A continuous function $g:[0, \infty)^{2} \to \mathbb{R}$ is  called a  \textbf{Bernstein function  with two variables} if $g \in C^{\infty}((0, \infty)^{2})$ and $\partial^{\beta}g(t)$ is a completely monotone function for $\beta = (1,1)$. Based on the equivalence of relations $(ii)$ and $(iii)$ in  Theorem \ref{reprcondneg}, we provide a characterization for those functions.  

\begin{theorem}\label{bernssev} A continuous function $g:[0, \infty)^{2}\to \mathbb{R}$  is  a Bernstein function   with two variables if and only if it can be represented as
	\begin{align*}
		g(t_{1},t_{2})= \int_{[0,\infty)^{2}}\left (\frac{1-e^{-r_{1}t_{1}}}{r_{1}} \right ) \left (\frac{1-e^{-r_{2}t_{2}}}{r_{2}} \right )(1+r_{1})(1+r_{2})d\sigma(r_{1},r_{2}) + g(0, t_{2})& \\
		+ g(t_{1},0) - g(0,0)&.
	\end{align*}
	where  the measure $\sigma \in \mathfrak{M}([0,\infty)^{2})$ is nonnegative. Further, the representation is unique.
\end{theorem}

If $g(t, 0)= g(0, t)=0$ for every $t \in [0, \infty)$, we say that $g$ is zero at the boundary (in particular, only the integral part in Theorem \ref{bernssev} appears). The representation in 	Theorem \ref{bernssev} implies the following characterization of  radial PDI kernels on all Euclidean spaces

\begin{theorem}\label{basicradial}
	Let $g: [0, \infty)^{2} \to \mathbb{R}$ be a  continuous function  that is zero at the boundary. The following conditions are equivalent:
	\begin{enumerate}
		\item [$(i)$] The kernel
		$$
		g( \|x-y\|^{2},\|u-v\|^{2} ), \quad  x,y \in  \mathbb{R}^{d}, \quad u, v \in \mathbb{R}^{d^{\prime}}
		$$ is PDI  for every  $d, d^{\prime }\in \mathbb{N}$.
		\item [$(ii)$] The kernel can be represented as
		\begin{align*}
			g( \|x-y\|^{2}, \|u-v\|^{2})=\int_{[0,\infty)^{2}} \frac{(1-e^{-r_{1}\|x-y\|^{2}})}{r_{1}}   \frac{(1-e^{-r_{2}\|u-v\|^{2}})}{r_{2}} \prod_{i=1}^{2}(1 +r_{i})d\sigma(r_{1},r_{2}) 
		\end{align*}
		where  the measure  $\sigma\in \mathfrak{M}([0,\infty)^{2})$ is nonnegative.
		\item [$(iii)$] The function $g$ is a Bernstein function   of two variables.
	\end{enumerate}
\end{theorem}

The assumption that $g$  is zero at the boundary equivalent to demanding that the projections of the  kernel 	$g( \|x-y\|^{2},\|u-v\|^{2} )$ are zero. Note that $g(t_{1}, t_{2}) = g(t_{1})g(t_{2})$ is a Bernstein function with two variables if and only if each $g_{i}$ is a Bernstein function.

From the following simple inequality				
\begin{equation}\label{ineqexp}
	(1-e^{-sa})\leq  \max \left (1, \frac{a}{b} \right ) (1-e^{-sb}), \quad s \in [0, \infty), \quad a,b> 0, 
\end{equation}
we obtain that
\begin{equation}\label{consineqexp}
	g(t_{1}, t_{2})\leq  \max (1, t_{1}/s_{1} )  \max (1, t_{2}/s_{2} )g(s_{1}, s_{2}), \quad t_{1}, t_{2}, s_{1}, s_{2} \in (0, \infty). 	
\end{equation}
Also, since 
$$
\frac{(1-e^{-s(a+b)})}{s}\leq \frac{(1-e^{-sa})}{s} +\frac{(1-e^{-sb})}{s}, \quad a,b,s \in [0, \infty)
$$
we obtain that for every $t_{1}, t_{2}, s_{1}, s_{2} \in [0, \infty)$
\begin{equation}\label{convexexp}
	g(t_{1}+s_{1}, t_{2} + s_{2})\leq  g(t_{1}, t_{2}) + g(t_{1}, s_{2})+g(s_{1}, s_{2})+g(s_{1}, t_{2}).  	
\end{equation}	
These results are key inequalities for the proofs in Section \ref{Examples of PDI-Characteristic kernels}.

Next result connects a special family of radial PDI kernels  with  completely monotone functions  of order $2$, which appeared at Theorem \ref{compleelltimes}.

\begin{cor}\label{2times} Let $f:[0, \infty) \to \mathbb{R}$ be a continuous function. The following conditions are equivalent:
	\begin{enumerate}
		\item [$(i)$] The kernel
		$$
		f( \|x-y\|^{2} + \|u-v\|^{2} ) - f(  \|u-v\|^{2} ) -f( \|x-y\|^{2}  ) + f(0),\quad  x,y \in  \mathbb{R}^{d}, \quad u, v \in \mathbb{R}^{d^{\prime}}
		$$ is PDI  for every  $d, d^{\prime }\in \mathbb{N}$.
		\item [$(ii)$] The kernel
		$$
		f( \|x-y\|^{2} + \|u-v\|^{2} ),\quad  x,y \in  \mathbb{R}^{d}, \quad u, v \in \mathbb{R}^{d^{\prime}}
		$$ is PDI  for every  $d, d^{\prime }\in \mathbb{N}$.
		\item [$(iii)$] The function $f$ can be represented as
		$$
		f(t)= a_{0} + a_{1}t + a_{2}t^{2}  + \int_{(0,\infty)}(e^{-rt} - e_{2}(r)\omega_{2}(rt))\frac{1+r^{2}}{r^{2}} d\sigma(r) 
		$$
		where the real scalar $a_{2}$ is nonnegative and $\sigma$ is a nonnegative  measure on $\mathfrak{M}((0, \infty))$.
		\item [$(iv)$] The function $f$ is  a completely monotone function of order $2$, that is, $f \in C^{\infty}((0, \infty)) $ and $f^{(2)}$ is a completely monotone function.
	\end{enumerate}
\end{cor}

As an example, the function $f(t)= t^{a}$, for $a \in (1,2)$ is a completely monotone function of order $2$ and also the function $f(t)= t\log (t)$. Note that for a function that satisfy any of the equivalent conditions in Corollary \ref{2times}, the following relation  holds
\begin{align*}
	f( \|x-y\|^{2}+ &\|u-v\|^{2})
	- 	f(  \|u-v\|^{2})-	f( \|x-y\|^{2}) + f(0) \\
	&=\int_{[0,\infty)}\left (\frac{1-e^{-r\|x-y\|^{2}}}{r} \right ) \left ( \frac{1-e^{-r\|u-v\|^{2}}}{r} \right )(1 +r^{2})d\sigma(r).
\end{align*}

\section{\textbf{Examples of PDI-Characteristic kernels}}\label{Examples of PDI-Characteristic kernels}

In this Section we present two families of PDI kernels and characterize the set of measures we can compare with them and under which conditions we obtain PDI-Characteristic kernels.

For the first family, let    $\gamma: X \times X \to [0, \infty)$ and  $\varsigma: Y \times Y \to [0, \infty)$ be  CND kernels and  $g: [0, \infty)\times [0, \infty) \to \mathbb{R}$ a continuous  Bernstein function of $2$ variables. Similar to the result on Equation \ref{bern+cnd}, the kernel
\begin{equation}\label{PDIradialcndcnd}
	[\mathfrak{I}_{g}^{\gamma, \varsigma}]((x,y),(x^{\prime}, y^{\prime})):=	g(\gamma(x,x^{\prime}),\varsigma(y,y^{\prime})) ,
\end{equation}
is PDI in $X \times Y$. Indeed, let $x_{1}, \ldots, x_{n} \in X$, $y_{1}, \ldots, y_{m} \in Y$ and real scalars $c_{i,k}$ with the necessary restrictions, then by Theorem \ref{basicradial}
\begin{align*}
	\sum_{i,j=1}^{n}\sum_{k,l=1}^{m}&c_{i, k}c_{j, l}g(\gamma(x_{i}, x_{j}),\varsigma(y_{k}, y_{l})) =\sum_{i,j=1}^{n}\sum_{k,l=1}^{m}c_{i, k}c_{j, l} \Biggl{[} -g(\gamma(x_{i}, x_{j}),0) - g(0,\varsigma(y_{k}, y_{l}))    \\
	&    + g(0,0)+ \int_{[0,\infty)^{2}}\left ( \frac{1-e^{-r_{1}\gamma(x_{i}, x_{j})}}{r_{1}}\right )\left (\frac{1-e^{-r_{2}\varsigma(y_{k}, y_{l})}}{r_{2}}\right )\prod_{i=1}^{2}(1+r_{i})d\sigma(r_{1},r_{2}) \Biggl{]}   \\
	&=\sum_{i,j=1}^{n}\sum_{k,l=1}^{m}c_{i, k}c_{j, l} \int_{[0,\infty)^{2}}\frac{e^{-r_{1}\gamma(x_{i}, x_{j})}}{r_{1}}\frac{e^{-r_{2}\varsigma(y_{k}, y_{l})}}{r_{2}}\prod_{i=1}^{2}(1+r_{i})d\sigma(r_{1},r_{2})\geq 0.
\end{align*}

As the function $g(t_{1}, t_{2}):=t_{1}t_{2}$ is a Bernstein function of two variables, we can generalize  Lemma \ref{integmarginkroeprod} and Corollary \ref{integmarginkroeprodcor}  for a family of kernels related to those on Theorem \ref{basicradial}. 

\begin{lem}\label{integmarginbernstein} Let $g: [0, \infty) \times [0, \infty) \to \mathbb{R}$ be a continuous Bernstein function of   $2$ variables that is zero at the boundary  and $\gamma: X \times X \to [0, \infty)$,  $\varsigma: Y \times Y \to [0, \infty)$  be continuous metrizable CND kernels with bounded diagonal. Then, for  nonnegative measures $(\mu, \nu) \in (\mathfrak{M}(X)\setminus{\delta(X)})\times (\mathfrak{M}(Y)\setminus{\delta(Y)})$  the following conditions are equivalent:
	\begin{enumerate}
		\item [$(i)$] $\mathfrak{I}_{g}^{\gamma, \varsigma} \in L^{1}((\mu \times \nu )\times (\mu \times \nu ) )$;
		\item [$(ii)$]The function  $(x,y) \to  \mathfrak{I}_{g}^{\gamma, \varsigma}((x,y),(x^{\prime},y^{\prime}))  \in L^{1}(\mu \times \nu )$ for some $(x^{\prime},y^{\prime}) \in X\times Y$;
		\item [$(iii)$] The function  $(x,y) \to 	\mathfrak{I}_{g}^{\gamma, \varsigma}((x, y ),(x^{\prime},y^{\prime}))  \in L^{1}(\mu \times \nu)$   for every $(x^{\prime},y^{\prime}) \in X\times Y$.
	\end{enumerate}
\end{lem}

\begin{cor}\label{integmargincorbernstein} Let $g: [0, \infty) \times [0, \infty) \to \mathbb{R}$ be a continuous Bernstein function of   $2$ variables that is zero at the boundary  and $\gamma: X \times X \to [0, \infty)$,  $\varsigma: Y \times Y \to [0, \infty)$  be continuous metrizable CND kernels with bounded diagonal. For  a measure $\lambda \in \mathfrak{M}(X\times Y)$ such that $\mathfrak{I}_{g}^{\gamma, \varsigma} \in L^{1}((|\lambda_{X}| \times |\lambda_{Y}| )\times (|\lambda_{X}| \times |\lambda_{Y}| ) )$,  the following conditions are equivalent:
	\begin{enumerate}
		\item [$(i)$] $\mathfrak{I}_{g}^{\gamma, \varsigma} \in L^{1}(|\lambda|\times |\lambda|  )$;
		\item [$(ii)$]The function  $(x,y) \to  \mathfrak{I}_{g}^{\gamma, \varsigma}((x,y),(x^{\prime},y^{\prime}))  \in L^{1}(|\lambda|)$ for some $(x^{\prime},y^{\prime}) \in X\times Y$;
		\item [$(iii)$] The function  $(x,y) \to 	\mathfrak{I}_{g}^{\gamma, \varsigma}((x, y ),(x^{\prime},y^{\prime}))  \in L^{1}(|\lambda|)$   for every $(x^{\prime},y^{\prime}) \in X\times Y$.
	\end{enumerate}
\end{cor}

As we obtained better  integrability relations in  Lemma \ref{integmarginbernstein} and  Corollary \ref{integmargincorbernstein} compared to Lemma \ref{integmargin}, we do not need to define the semi-inner product using various vector spaces  as done in Theorem \ref{principalmetric}, instead, we can use only one vector space. We emphasize  that we are not assuming that the projections kernels of the PDI kernel  $\mathfrak{I}_{g}^{\gamma, \varsigma}$  are zero.

\begin{theorem}\label{principalmetricbernstein} Let $g: [0, \infty) \times [0, \infty) \to \mathbb{R}$ be a continuous Bernstein function of   $2$ variables that is zero at the boundary  and $\gamma: X \times X \to [0, \infty)$,  $\varsigma: Y \times Y \to [0, \infty)$  be continuous metrizable CND kernels with bounded diagonal. The following subset of $\mathfrak{M}(X\times Y)$ is a vector space
	\begin{align*}
		\mathfrak{M}(\mathfrak{I}_{g}^{\gamma, \varsigma}):=\{\lambda \in \mathfrak{M}&(X\times Y),  \quad \mathfrak{I}_{g}^{\gamma, \varsigma} \in L^{1}((|\lambda|_{X}\times|\lambda|_{Y}    )\times (|\lambda|_{X}\times|\lambda|_{Y}    ))  , \\ 
		& \mathfrak{I}_{g}^{\gamma, \varsigma} \in L^{1}(|\lambda|\times |\lambda|), \quad \lambda(X, \cdot) \text{ and }  \lambda( \cdot, Y) \text{ are the zero measure}   \}.
	\end{align*}
	The real valued function 
	$$
	(\lambda, \lambda^{\prime}) \to \int_{X\times Y} \int_{X\times Y} \mathfrak{I}_{g}^{\gamma, \varsigma}((x,y),(x^{\prime}, y^{\prime}))d\lambda (x,y)d\lambda^{\prime}(x^{\prime}, y^{\prime}),
	$$
	is well defined for $\lambda, \lambda^{\prime} \in \mathfrak{M}(\mathfrak{I}_{g}^{\gamma, \varsigma})$ and is a semi-inner product on it.   
\end{theorem}

Because of Theorem \ref{principalmetricbernstein}, the concept of PDI-Characteristic kernel can, equivalently, be defined using only one  vector space $\mathfrak{M}(\mathfrak{I}_{g}^{\gamma, \varsigma})$, instead of multiple ones using $	\mathfrak{M}_{\mu, \nu}(\mathfrak{I}_{g}^{\gamma, \varsigma})$.

Below, we present a criterion for when the kernels $\mathfrak{I}_{g}^{\gamma, \varsigma}$ are PDI-Characteristic

\begin{theorem}\label{characPDIbern2}Let $g: [0, \infty) \times [0, \infty) \to \mathbb{R}$ be a continuous Bernstein function of   $2$ variables   that is zero at the boundary,  $\gamma: X \times X \to [0, \infty)$,  $\varsigma: Y \times Y \to [0, \infty)$  be continuous metrizable CND kernels with bounded diagonal. Then,  the kernel $g(\gamma , \varsigma )$ is $PDI$-Characteristic if and only  one of the following relations about the representation of Theorem \ref{bernssev} is satisfied
	\begin{enumerate}
		\item [$(i)$] If $\gamma$ and $\varsigma$ are not CND-Characteristic: $\sigma((0,\infty)^{2})>0$ (or equivalently, $g(t_{1}, t_{2})$ is not of the type $t_{1}\psi(t_{2}) + t_{2}\varphi(t_{1})$);  
		\item [$(ii)$] If $\gamma$ is CND-Characteristic and $\varsigma$ is not:  $\sigma([0, \infty) \times (0,\infty))>0$ (or equivalently, $g(t_{1}, t_{2})$ is not of the type $ t_{2}\varphi(t_{1})$); 
		\item [$(iii)$] If $\varsigma$ is CND-Characteristic and $\gamma$ is not: $\sigma((0, \infty) \times [0,\infty))>0$ (or equivalently, $g(t_{1}, t_{2})$ is not of the type $ t_{1}\psi(t_{2})$); 
		\item [$(iv)$] If $\gamma$  and $\varsigma$ are CND-characteristic:  $\sigma([0, \infty) \times [0,\infty))>0$ (or equivalently, $g$ is not  a constant function), 
	\end{enumerate}
	where $\psi, \varphi :  [0,\infty) \to \mathbb{R}$ are Bernstein functions.
\end{theorem}

The second family is based on the PDI radial kernels mentioned after Corollary \ref{2times}, for which we have several nontrivial examples. Let    $\gamma: X \times X \to [0, \infty)$ and  $\varsigma: Y \times Y \to [0, \infty)$ be  CND kernels and  $ \psi: [0, \infty) \to \mathbb{R}$ be a  function in $CM_{2}$. Then, the kernel
\begin{equation}\label{PDIcm2cndcnd}
	[\mathfrak{I}_{\psi}^{\gamma + \varsigma}]((x,y),(x^{\prime}, y^{\prime})):=	\psi(\gamma(x,x^{\prime})+\varsigma(y,y^{\prime})),
\end{equation}
is PDI in $X\times Y$. Indeed, let $x_{1}, \ldots, x_{n} \in X$, $y_{1}, \ldots, y_{m} \in Y$ and real scalars $c_{i,k}$ with the necessary restrictions, then by Theorem \ref{compleelltimes} 
\begin{align*}
	&\sum_{i,j=1}^{n}\sum_{k,l=1}^{m}c_{i, k}c_{j, l}\psi(\gamma(x_{i}, x_{j})+\varsigma(y_{k}, y_{l})) =\sum_{i,j=1}^{n}\sum_{k,l=1}^{m}c_{i, k}c_{j, l} \Biggl{[}  \sum_{k=0}^{2}a_{k}[\gamma(x_{i}, x_{j})+\varsigma(y_{k}, y_{l})]^{k} \\ 
	& + \int_{(0,\infty)} \left ( e^{-r[\gamma(x,x^{\prime}) + \varsigma(y,y^{\prime})]} - \omega_{2}(r[\gamma(x,x^{\prime}) + \varsigma(y,y^{\prime})])e_{2}(r) \right )\frac{1+r^{\ell}}{r^{\ell}} d\sigma(r) \Biggl{]}   \\
	&=\sum_{i,j=1}^{n}\sum_{k,l=1}^{m}c_{i, k}c_{j, l} \left [ 2a_{2}\gamma(x_{i}, x_{j})\varsigma(y_{k}, y_{l}) + \int_{(0,\infty)}  e^{-r[\gamma(x,x^{\prime}) + \varsigma(y,y^{\prime})]}  \frac{1+r^{\ell}}{r^{\ell}} d\sigma(r)\right ]\geq 0,
\end{align*}
where the conclusion comes from the fact that $\gamma \otimes \varsigma$ is an PDI kernel and that $r[\gamma + \varsigma]$ is an CND kernel in $X\times Y$ for every $r >0$.

As we have several examples of functions in $CM_{2} \cap C^{1}([0, \infty))$,  we prove  additional results concerning the integrability conditions of the kernels involved.  These results are not consequence of 	Lemma \ref{integmarginbernstein} and Corollary \ref{integmargincorbernstein},  for instance if $\psi(t)= t^{2}$, the integrability of $\psi(\gamma + \varsigma)$ is different from $\psi(\gamma + \varsigma) - \psi(\gamma) - \psi(\sigma) + \psi(0)$. Also,  note that the projection kernels of $\mathfrak{I}_{\psi}^{\gamma + \varsigma}$ are not zero.


\begin{lem}\label{integmargincompl2}Let $\psi: [0, \infty) \to \mathbb{R}$ be a function either in $CM_{2} \cap C^{1}([0, \infty))$ or in $CM_{1}$, $\gamma: X \times X \to [0, \infty)$,  $\varsigma: Y \times Y \to [0, \infty)$  be continuous metrizable CND kernels with bounded diagonal. Then, for a measure  $\lambda \in \mathfrak{M}(X \times Y)$  the following conditions are equivalent :
	\begin{enumerate}
		\item [$(i)$] $\psi(\gamma +  \varsigma)\in L^{1}(|\lambda|\times |\lambda|)$;
		\item [$(ii)$]The function  $\psi(\gamma(\cdot, x^{\prime}) +  \varsigma(\cdot , y^{\prime}))  \in L^{1}(|\lambda|)$ for some $(x^{\prime},y^{\prime}) \in X\times Y$;
		\item [$(iii)$] The function  $\psi(\gamma(\cdot, x^{\prime}) +  \varsigma(\cdot , y^{\prime}))  \in L^{1}(|\lambda|)$   for every $(x^{\prime},y^{\prime}) \in X\times Y$.
		\item[$(iv)$] $\psi(\gamma +  \varsigma)\in L^{1}([|\lambda|_{X} \times |\lambda|_{Y}] \times [|\lambda|_{X} \times |\lambda|_{Y}])$
		\item[$(v)$] $\psi(\gamma) \in L^{1}(|\lambda|_{X} \times |\lambda|_{X})$ and $\psi(\varsigma) \in L^{1}(|\lambda|_{Y} \times |\lambda|_{Y})$
	\end{enumerate}
	
\end{lem}

\begin{theorem}\label{principalmetriccompl2}	Let $\psi: [0, \infty) \to \mathbb{R}$ be a function  either in $CM_{2} \cap C^{1}([0, \infty))$ or in $CM_{1}$, $\gamma: X \times X \to [0, \infty)$,  $\varsigma: Y \times Y \to [0, \infty)$  be continuous metrizable CND kernels with bounded diagonal.  The following subset of $\mathfrak{M}(X\times Y)$ is a vector space
	\begin{align*}
		\mathfrak{M}(\mathfrak{I}_{\psi}^{\gamma + \varsigma} ):=\{   \psi(\gamma + \varsigma) \in L^{1}(|\lambda|\times |\lambda|), \quad \lambda(X, \cdot) \text{ and }  \lambda( \cdot, Y) \text{ are the zero measure}   \}. \quad &
	\end{align*}
	The real valued function 
	$$
	(\lambda, \lambda^{\prime}) \to \int_{X\times Y} \int_{X\times Y} \psi(\gamma(x,x^{\prime}) + \varsigma(y,y^{\prime}))d\lambda (x,y)d\lambda^{\prime}(x^{\prime}, y^{\prime}),
	$$
	is well defined for $\lambda, \lambda^{\prime} \in \mathfrak{M}(\mathfrak{I}_{\psi}^{\gamma + \varsigma} )$ and is a semi-inner product on it.  It is an inner product if and only if one of the conditions is satisfied
	\begin{enumerate}
		\item[(i)] $\psi \in CM_{1}$ and $\psi$ is not a linear polynomial. 
		\item[(ii-1)] $\psi \in CM_{2} \cap C^{1}([0, \infty))$, $\psi$ is not a quadratic polynomial and  $\gamma$, $\varsigma$ are not CND-Characteristic. 
		\item[(ii-2)] $\psi \in CM_{2} \cap C^{1}([0, \infty))$, $\psi$ is not a linear polynomial and $\gamma$, $\varsigma$ are CND-Characteristic. 
	\end{enumerate}
\end{theorem}	

As similarly done in \cite{energybernstein}, it is possible to prove Theorem \ref{principalmetriccompl2} for an arbitrary function in $CM_{2}$, but at the cost of imposing the stronger assumption that $(\gamma  + \varsigma)^{2} \in L^{1}(|\lambda|\times |\lambda|)$ in the definition of $	\mathfrak{M}(\mathfrak{I}_{\psi}^{\gamma + \varsigma} )$.  As mentioned after Lemma \ref{compleelltimes}, the functions  $t\log(t)$ and $t^{a}$, $a \in (1,2)$ are in $CM_{2}\cap C^{1}([0, \infty))$.	In particular, the kernels
$$
(d_{S^{d}}(x,x^{\prime})+ \|y-y^{\prime}\|)^{3/2}, \quad (d_{S^{d}}(x,x^{\prime})+ \|y-y^{\prime}\|)\log(d_{S^{d}}(x,x^{\prime})+ \|y-y^{\prime}\|)
$$
are PDI-Characteristic in $S^{d}\times \mathcal{H}$ for any Hilbert space $\mathcal{H}$, where $S^{d}$ is the unit sphere in $\mathbb{R}^{d+1}$ and $d_{S^{d}}(x,x^{\prime})= \arccos (\langle x,x^{\prime}\rangle )$ is the geodesic distance.

We emphasize that for the two family of kernels presented in this Section, $\mathfrak{I}$ is PDIK-Characteristic if and only if 
$$
D(\pi, \pi^{\prime})= \sqrt{  \int_{X\times Y} \int_{X\times Y} \mathfrak{I}((x,y),(x^{\prime}, y^{\prime})) d[\pi - \pi^{\prime}] (x,y)d[\pi - \pi^{\prime}](x^{\prime}, y^{\prime}) }
$$
is a metric on the convex set
$$
\Gamma(P,Q)_{\mathfrak{I}}:=\{ \lambda \in \mathfrak{M}(X\times Y), \quad \lambda  \text{ is a probability}, \quad \mathfrak{I} \in L^{1}(\lambda \times \lambda), \quad \lambda_{X}=P, \lambda_{Y}=Q \}
$$
for every  probabilities  $P\in \mathfrak{M}(X)\setminus{\delta(X)}$, $Q \in \mathfrak{M}(Y)\setminus{\delta(Y)}$  such that $\mathfrak{I} \in L^{1}([P \times Q]\times [P\times Q])$.

The way it is stated Lemma \ref{integmargincompl2} and Theorem \ref{principalmetriccompl2}, it is not including important examples of functions in $CM_{1}$ (which is a subset of $CM_{2}$) that are not differentiable at $0$, such as $\psi(t)=-t^{a}$, $a \in (0,1)$. However, since the CND kernel $\gamma + \varsigma$ is metrizable but cannot be CND-Characteristic, Theorem $3.3$ in \cite{energybernstein} implies that if $\psi \in CM_{1}$,  then
$$
(\lambda, \lambda^{\prime}) \to \int_{X\times Y} \int_{X\times Y} \psi(\gamma(x,x^{\prime}) + \varsigma(y,y^{\prime}))d\lambda(x,y)d\lambda^{\prime}(x^{\prime},y^{\prime})
$$
is an inner product in $\mathfrak{M}_{1}(X\times Y;-\psi(\gamma + \varsigma))$ if and only if $\psi$ is not a linear polynomial. 


\section{\textbf{Distance covariance for PDI kernels}}\label{Distance covariance}
On this brief Section we present the notion of distance covariance for PDI kernels, which satisfies the same aspects as the one defined in 

\begin{lem}\label{cndtopdsquare}Let $\gamma: X \times X \to \mathbb{R}$ be a symmetric  kernel.
	\begin{enumerate}\item[(i)]The kernel $K_{\gamma}: X^{2}\times X^{2} \to \mathbb{R}$ defined as
		$$
		K_{\gamma}((x,z),(y,w)):=-\gamma(x,y)-\gamma(z,w) + \gamma(x,w) + \gamma(z,y) 
		$$	
		is PD if and only if $\gamma$ is CND. 
		\item[(ii)]If $\gamma$ is  a continuous CND kernel with bounded diagonal, then, for any probabilities  $P,Q \in \mathfrak{M}(X)$ that satisfies any of the $3$ equivalences in Lemma \ref{estimativa}, the function $(x,y) \in X^{2} \to K_{\gamma}([x,y], [x,y]) \in L^{1}(P\times Q)$, moreover
		\begin{align*}
			\int_{X^{2}} \int_{X^{2}} K_{\gamma}((x,z),(y,w)
			&d[P\otimes Q](x,z) d[P\otimes Q](y,w)\\
			&= \int_{X\times X}-\gamma(x,y)d[P-Q](x)d[P-Q](y).
		\end{align*}		
	\end{enumerate}
\end{lem}

Inspired by Lemma \ref{cndtopdsquare} we improve Lemma \ref{PDItoPD} and obtain a connection between a pseudometric defined by an PDI kernel in $X\times Y$  with a maximum mean discrepancy  on the set $[X\times Y]^{2}$. 

\begin{lem}\label{dcovtopdsquare}	Let $\mathfrak{I}: (X \times Y) \times (X\times Y) \to \mathbb{R}$ be a symmetric  kernel.
	\begin{enumerate}
		\item[(i)] The kernel $K_{\mathfrak{I}}: [X\times Y]^{2} \times [X \times Y]^{2} \to \mathbb{R}$ defined as
		\begin{align*}
			&K_{\mathfrak{I}}([x,y, z,w ] ,[x^{\prime},y^{\prime}, z^{\prime},w^{\prime}] ):=\\
			& \left [ \right . \mathfrak{I}((x,y),(x^{\prime}, y^{\prime})) - \mathfrak{I}((z,y),(x^{\prime}, y^{\prime}))-\mathfrak{I}((x,w),(x^{\prime}, y^{\prime}))+\mathfrak{I}((z,w),(x^{\prime}, y^{\prime}))\\
			&-\mathfrak{I}((x,y),(z^{\prime}, y^{\prime}))+\mathfrak{I}((z,y),(z^{\prime}, y^{\prime}))+\mathfrak{I}((x,w),(z^{\prime}, y^{\prime}))-\mathfrak{I}((z,w),(z^{\prime}, y^{\prime}))\\
			&-\mathfrak{I}((x,y),(x^{\prime}, w^{\prime}))+\mathfrak{I}((z,y),(x^{\prime}, w^{\prime}))+\mathfrak{I}((x,w),(x^{\prime}, w^{\prime}))-\mathfrak{I}((z,w),(x^{\prime}, w^{\prime}))\\
			&+\mathfrak{I}((x,y),(z^{\prime}, w^{\prime}))-\mathfrak{I}((z,y),(z^{\prime}, w^{\prime}))-\mathfrak{I}((x,w),(z^{\prime}, w^{\prime}))+\mathfrak{I}((z,w),(z^{\prime}, w^{\prime}))\left . \right ]
		\end{align*}

		is PD if and only if $\mathfrak{I}$ is PDI. 
		\item[(ii)]If $\mathfrak{I}$ is an PDI kernel that satisfies the requirements in  Theorem \ref{principalmetricbernstein} or Theorem \ref{principalmetriccompl2}.   Then, for any $\lambda, \lambda^{\prime } \in  \varGamma(P,Q)_{\mathfrak{I}}$ we have that the function $K_{\mathfrak{I}}([x,y, z,w], [x ,y, z,w])$ is an element of  $L^{1}(\lambda \otimes \lambda^{\prime}) $, moreover
		\begin{align*}
			&\int_{[X\times Y]^{2}}\int_{[X\times Y]^{2}}K_{\mathfrak{I}}([x,y, z,w ] ,[x^{\prime},y^{\prime}, z^{\prime},w^{\prime}] )d[\lambda\otimes \lambda^{\prime}](x,y, z,w )d[\lambda\otimes \lambda^{\prime}](x^{\prime},y^{\prime}, z^{\prime},w^{\prime} ) \\
			&= \int_{X\times Y} \int_{X\times Y} \mathfrak{I}((x,y),(x^{\prime}, y^{\prime}))d[\lambda+\lambda^{\prime}-2P\otimes Q ](x,y)d[\lambda+\lambda^{\prime}-2P\otimes Q](x^{\prime}, y^{\prime}). 
		\end{align*}
	\end{enumerate}
	
\end{lem}

In the case that $\mathfrak{I}$ is the Kronecker product of CND kernels $\gamma$, $\varsigma$, then the $16$ kernels that appears in the definition of $K_{\mathfrak{I}}$ in relation $(i)$ at Lemma \ref{dcovtopdsquare} can be simplified to a product of 4 kernels involving $\gamma$ and $4$ kernels involving $\varsigma$. More precisely $K_{\gamma \otimes \varsigma}([x,y, z,w ] ,[x^{\prime},y^{\prime}, z^{\prime},w^{\prime}] )$ is equal to 
$$
\big{[} \gamma(x,x^{\prime})+ \gamma(z,z^{\prime}) - \gamma(x,z^{\prime}) - \gamma(z,x^{\prime})  \big{]}  \big{[} \varsigma(y,y^{\prime})+ \varsigma(w,w^{\prime}) -\varsigma(y,w^{\prime}) - \varsigma(w,y^{\prime}) \big{]}. 
$$

Additionally, if $\lambda^{\prime}= P \otimes Q$ we can use Fubini-Tonelli to change the order of integration and integrate in the variables related to $\lambda^{\prime}$ and  obtain that
\begin{align*}
	&\int_{X\times Y} \int_{X\times Y} \gamma(x,x^{\prime}) \varsigma(y,y^{\prime})d[\lambda-P\otimes Q ](x,y)d[\lambda-P\otimes Q](x^{\prime}, y^{\prime})\\
	&=\int_{[X\times Y]^{2}}\int_{[X\times Y]^{2}}K_{\gamma \otimes \varsigma}([x,y, z,w ] ,[x^{\prime},y^{\prime}, z^{\prime},w^{\prime}] ) d[\lambda\otimes \lambda^{\prime}](x,y, z,w )d[\lambda\otimes \lambda^{\prime}](x^{\prime},y^{\prime}, z^{\prime},w^{\prime} ) \\
	&=\int_{X\times Y}\int_{X\times Y}  K^{\gamma}_{P}(x,x^{\prime}) K^{\varsigma}_{Q}(y,y^{\prime})  d\lambda(x,y)d\lambda(x^{\prime},y^{\prime} ),
\end{align*}
where 
$$
K^{\gamma}_{P}(x,x^{\prime})=\gamma(x,x^{\prime}) - a_{P}(x) - a_{P}(x^{\prime})  + D(P), 
$$	
and
$$
a_{P}(x) := \int_{X} \gamma(x,x^{\prime})dP(x^{\prime}), \quad D(P)= \int_{X}\int_{X}\gamma(x,x^{\prime})dP(x) dP(x^{\prime}).
$$
Defined similarly   for $K^{\varsigma}_{Q}$, $a_{Q}(y)$ and $D(Q)$, using the kernel $\varsigma$. As defined in 	
$$	
dcov(\lambda):=	\int_{X\times Y} \int_{X\times Y} K^{\gamma}_{P}(x,x^{\prime})K^{\varsigma}_{Q}(y,y^{\prime}) d\lambda(x,y)d\lambda(x^{\prime},y^{\prime}),
$$

Because of Lemma \ref{dcovtopdsquare}, $dcov$ is a nonnegative function and it is a criteria for independence if and only if $\gamma \otimes \varsigma$ is PDI-Characteristic.	

The function $dcov$ can be generalized to the PDI kernels in Theorem \ref {principalmetricbernstein} and Theorem \ref{principalmetriccompl2}, by replacing the positive definite kernel $K^{\gamma}_{P}(x,x^{\prime})K^{\varsigma}_{Q}(y,y^{\prime})$ 
(defined in  $X \times  Y$), with the positive definite kernel (also defined in $X \times Y$)

\begin{alignat*}{3}
	K_{P,Q}^{\mathfrak{I}}((x,y),(x^{\prime}, y^{\prime})):= [ &  \mathfrak{I}((x,y),(x^{\prime}, y^{\prime}))  & -b_{P}(y,x^{\prime}, y^{\prime}) & - b_{Q}(x,x^{\prime}, y^{\prime})	 & + c_{P,Q}(x^{\prime}, y^{\prime})\\
	& - b_{P}(x,y,y^{\prime}) & + d_{P,P}(y,y^{\prime}) & + c_{P,Q}(x,y^{\prime}) & - e_{P,P,Q}(y^{\prime})\\
	& -b_{Q}(x,y,x^{\prime}) & + c_{P,Q}(x^{\prime},y) & + d_{Q,Q}(x,x^{\prime}) & - e_{P,Q,Q}(x^{\prime})\\
	& +c_{P,Q}(x,y) & - e_{P,P,Q}(y) & - e_{P,Q,Q}(x) & + f_{P,P,Q,Q}],
\end{alignat*}
where
\begin{align*}
	b_{P}(y,x^{\prime}, y^{\prime}):=& \int_{X}\mathfrak{I}((z,y),(x^{\prime}, y^{\prime}))dP(z)\\
	c_{P,Q}(x,y):=& \int_{X \times Y}\mathfrak{I}((x,y),(z, w))dP(z)dQ(w)\\
	d_{Q,Q}(x,x^{\prime}):=&\int_{Y \times Y}\mathfrak{I}((x,w),(x^{\prime}, w^{\prime}))dQ(w)dQ(w^{\prime})\\
	e_{P,Q,Q}(x):=&\int_{X \times Y \times Y}\mathfrak{I}((x,w),(z, w^{\prime}))dP(z)dQ(w)dQ(w^{\prime})\\
	f_{P,P,Q,Q} :=&\int_{X\times Y} \int_{X \times Y}\mathfrak{I}((z,w),(z^{\prime}, w^{\prime}))dP(z)dP(z^{\prime})dQ(w)dQ(w^{\prime})
\end{align*}

If  $\mathfrak{I}$ is an PDI kernel that satisfies the requirements of Theorem \ref{principalmetric}, we can also define the distance covariance in this setting, however the  PD kernel $K^{\mathfrak{I}}_{P,Q}$ is only defined in $X_{P, Q}^{\mathfrak{I}} \times Y^{\mathfrak{I}}_{P ,Q}$.

\section{\textbf{Proofs}}\label{Proofs}

\subsection{\textbf{Section \ref{positive definite independent kernels}}} \label{proofsecpositive definite independent kernels}

\begin{proof}[\textbf{Proof of Lemma \ref{PDItoPD}}]Suppose that  $K^{\mathfrak{I}}$ is PD, let  $x_{1}, \ldots, x_{n} \in X$, $y_{1}, \ldots , y_{m} \in Y$  and real scalars $c_{i, k}$ with the restrictions
	$$
	\sum_{i=1}^{n}c_{i,k}=0 , \quad \sum_{l=1}^{m}c_{j, l}=0, 
	$$
	for every $1\leq k \leq m$, $1\leq j \leq n$. Then 
	$$
	0 \leq \sum_{i,j=1}^{n}\sum_{k,l=1}^{m}c_{i, k}c_{j, l}K^{\mathfrak{I}}((x_{i},y_{k}),(x_{j}, y_{l}))  = \sum_{i,j=1}^{n}\sum_{k,l=1}^{m}c_{i, k}c_{j, l}\mathfrak{I}((x_{i},y_{k}),(x_{j}, y_{l})),
	$$
	because in the definition of $K^{\mathfrak{I}}$, all $16$ double double sums are zero except for the first term.\\  
	Conversely, let $x_{1}, \ldots, x_{n} \in X$, $y_{1}, \ldots , y_{m} \in Y$  and arbitrary real scalars $c_{i, k}$. Define the scalars $c_{0, k}= - \sum_{i=1}^{n}c_{i,k}$, $c_{i,0}= - \sum_{k=1}^{m}c_{i,k}$ and  $c_{0,0}= \sum_{i=1}^{n}\sum_{k=1}^{m}c_{i,k} $, which satisfy 
	$$
	\sum_{i=0}^{n}c_{i,k}=0 , \quad \sum_{l=0}^{m}c_{j, l}=0, 
	$$
	for every $0\leq k \leq m$, $0\leq j \leq n$. Note that

	\begin{align*}
		&\sum_{i,j=1}^{n}\sum_{k,l=1}^{m}c_{i, k}c_{j, l}K^{\mathfrak{I}}((x_{i},y_{k}),(x_{j}, y_{l}))=\\
		&\bigg{[}  \sum_{i,j=1}^{n}\sum_{k,l=1}^{m}c_{i, k}c_{j, l}\mathfrak{I}((x_{i},y_{k}),(x_{j}, y_{l}))  + \sum_{j=1}^{n}\sum_{k,l=1}^{m}c_{0,k}c_{j,l}\mathfrak{I}((x_{0},y_{k}),(x_{j}, y_{l}))\\
		& +\sum_{i,j=1}^{n}\sum_{l=1}^{m}c_{i, 0}c_{j, l}\mathfrak{I}((x_{i},y_{0}),(x_{j}, y_{l}))  + \sum_{j=1}^{n}\sum_{l=1}^{m}c_{0,0}c_{j,l}\mathfrak{I}((x_{0},y_{0}),(x_{j}, y_{l}))\\
		&+\sum_{i=1}^{n}\sum_{k,l=1}^{m}c_{i,k}c_{0, l}\mathfrak{I}((x_{i},y_{k}),(x_{0}, y_{l}))+\sum_{k,l=1}^{m}c_{0, k}c_{0, l}\mathfrak{I}((x_{0},y_{k}),(x_{0}, y_{l}))\\
		&+\sum_{i=1}^{n}\sum_{l=1}^{m}c_{i, 0}c_{0, l}\mathfrak{I}((x_{i},y_{0}),(x_{0}, y_{l}))+\sum_{l=1}^{m}c_{0, 0}c_{0, l}\mathfrak{I}((x_{0},y_{0}),(x_{0}, y_{l}))\\
		&+\sum_{i,j=1}^{n}\sum_{k=1}^{m}c_{i, k}c_{j, 0}\mathfrak{I}((x_{i},y_{k}),(x_{j}, y_{0}))+\sum_{j=1}^{n}\sum_{k=1}^{m}c_{0, k}c_{j, 0}\mathfrak{I}((x_{0},y_{k}),(x_{j}, y_{0}))\\
		&+\sum_{i,j=1}^{n}c_{i, 0}c_{j, 0}\mathfrak{I}((x_{i},y_{0}),(x_{j}, y_{0}))+\sum_{j=1}^{n}c_{0, 0}c_{j, 0}\mathfrak{I}((x_{0},y_{0}),(x_{j}, y_{0}))\\
		&+\sum_{i=1}^{n}\sum_{k=1}^{m}c_{i, k}c_{0, 0}\mathfrak{I}((x_{i},y_{k}),(x_{0}, y_{0}))+\sum_{k=1}^{m}c_{0, k}c_{0, 0}\mathfrak{I}((x_{0},y_{k}),(x_{0}, y_{0}))\\
		&+\sum_{i=1}^{n}c_{i, 0}c_{0, 0}\mathfrak{I}((x_{i},y_{0}),(x_{0}, y_{0}))+c_{0, 0}c_{0, 0}\mathfrak{I}((x_{0},y_{0}),(x_{0}, y_{0})) \bigg{]}\\
		&=\sum_{i,j=0}^{n}\sum_{k,l=0}^{m}c_{i, k}c_{j, l}\mathfrak{I}((x_{i},y_{k}),(x_{j}, y_{l})) \geq 0. 
	\end{align*}
	
\end{proof}

\begin{proof}[\textbf{Proof of Lemma \ref{PDIsimpli}}] Let $x_{1}, \ldots, x_{n} \in X$, $y_{1}, \ldots , y_{m} \in Y$  and real scalars $c_{i, k}$ with the restrictions
	$$
	\sum_{i=1}^{n}c_{i,k}=0 , \quad \sum_{l=1}^{m}c_{j, l}=0, 
	$$
	for every $1\leq k \leq m$, $1\leq j \leq n$. Then 
	$$
	\sum_{i,j=1}^{n}\sum_{k,l=1}^{m}c_{i, k}c_{j, l}\mathfrak{I}((x_{i},y_{k}),(x_{j}, y_{l}))=\sum_{i,j=1}^{n}\sum_{k,l=1}^{m}c_{i, k}c_{j, l}\mathfrak{I}^{\prime}((x_{i},y_{k}),(x_{j}, y_{l})) ,
	$$
	which proves our assertion.
\end{proof}

\begin{proof}[\textbf{Proof of Theorem \ref{PDIgeometryrkhs}}] The proof follows by a verification that both sides are equal using the properties of RKHS, Lemma \ref{PDItoPD} and Equation \ref{eqmargkernels0}. Indeed, 
	\begin{align*}
		&\| K^{\mathfrak{I}}_{x,y} + K^{\mathfrak{I}}_{x^{\prime},y^{\prime}} -K^{\mathfrak{I}}_{x,y^{\prime}}-K^{\mathfrak{I}}_{x^{\prime},y} \|_{\mathcal{H}_{ K^{\mathfrak{I}} } }^{2} = 2\langle K^{\mathfrak{I}}_{x,y} , K^{\mathfrak{I}}_{x^{\prime},y^{\prime}}\rangle +2\langle K^{\mathfrak{I}}_{x,y^{\prime}} , K^{\mathfrak{I}}_{x^{\prime},y}\rangle \\ 
		&+\langle K^{\mathfrak{I}}_{x,y} , K^{\mathfrak{I}}_{x,y}\rangle +\langle K^{\mathfrak{I}}_{x^{\prime},y^{\prime}} , K^{\mathfrak{I}}_{x^{\prime},y^{\prime}}\rangle +\langle K^{\mathfrak{I}}_{x,y^{\prime}} , K^{\mathfrak{I}}_{x,y^{\prime}}\rangle +\langle K^{\mathfrak{I}}_{x^{\prime},y} , K^{\mathfrak{I}}_{x^{\prime},y}\rangle\\
		&- 2\langle K^{\mathfrak{I}}_{x,y} , K^{\mathfrak{I}}_{x^{\prime},y}\rangle -2\langle K^{\mathfrak{I}}_{x,y} , K_{x,y^{\prime}}\rangle - 2\langle  K^{\mathfrak{I}}_{x^{\prime},y}, K^{\mathfrak{I}}_{x^{\prime},y^{\prime}} \rangle -2\langle K^{\mathfrak{I}}_{x^{\prime},y^{\prime}} , K^{\mathfrak{I}}_{x,y^{\prime}}\rangle. 
	\end{align*}
	
	The $10$ inner products above are equal to
	\begin{align*}
		\langle K^{\mathfrak{I}}_{x,y} , K^{\mathfrak{I}}_{x,y}\rangle &= 2\mathfrak{I}((x_{0},y_{0}),(x,y))  + 2\mathfrak{I}((x,y_{0}),(x_{0},y)). \\
		\langle K^{\mathfrak{I}}_{x^{\prime},y^{\prime}} , K^{\mathfrak{I}}_{x^{\prime},y^{\prime}}\rangle &= 2\mathfrak{I}((x_{0},y_{0}),(x^{\prime},y^{\prime})) + 2\mathfrak{I}((x^{\prime},y_{0}),(x_{0},y^{\prime})).\\
		\langle K^{\mathfrak{I}}_{x^{\prime},y} , K^{\mathfrak{I}}_{x^{\prime},y}\rangle &= 2\mathfrak{I}((x_{0},y_{0}),(x^{\prime},y)) + 2\mathfrak{I}((x^{\prime},y_{0}),(x_{0},y)).\\
		\langle K^{\mathfrak{I}}_{x,y^{\prime}} ,K^{\mathfrak{I}}_{x,y^{\prime}}\rangle &= 2\mathfrak{I}((x_{0},y_{0}),(x,y^{\prime})) + 2\mathfrak{I}((x,y_{0}),(x_{0},y^{\prime})).
	\end{align*}
	
	\begin{align*}
		&-2\langle K^{\mathfrak{I}}_{x^{\prime},y^{\prime}} , K^{\mathfrak{I}}_{x^{\prime},y}\rangle= 2\mathfrak{I}((x^{\prime},y),(x_{0},y^{\prime}))+2\mathfrak{I}((x_{0},y),(x^{\prime},y^{\prime}))\\
		&-2  \mathfrak{I}((x_{0},y_{0}),(x^{\prime},y^{\prime})) -2 \mathfrak{I}((x^{\prime},y),(x_{0},y_{0})) -2 \mathfrak{I}((x^{\prime},y_{0}),(x_{0},y^{\prime})) -2 \mathfrak{I}((x_{0},y),(x^{\prime},y_{0})).
	\end{align*}
	
	\begin{align*}
		&-2\langle K^{\mathfrak{I}}_{x,y^{\prime}} , K^{\mathfrak{I}}_{x^{\prime},y^{\prime}}\rangle =2\mathfrak{I}((x,y^{\prime}),(x^{\prime},y_{0}))+2\mathfrak{I}((x,y_{0}),(x^{\prime},y^{\prime}))\\
		&-2\mathfrak{I}((x_{0},y_{0}),(x^{\prime},y^{\prime})) -2\mathfrak{I}((x,y^{\prime}),(x_{0},y_{0})) -2\mathfrak{I}((x, y_{0}),(x_{0},y^{\prime})) -2\mathfrak{I}((x_{0},y^{\prime}),(x^{\prime},y_{0})).
	\end{align*}
	
	\begin{align*}
		&-2\langle K^{\mathfrak{I}}_{x,y} , K^{\mathfrak{I}}_{x,y^{\prime}}\rangle=2\mathfrak{I}((x,y),(x_{0},y^{\prime}))+2\mathfrak{I}((x_{0},y),(x,y^{\prime}))\\
		&-2\mathfrak{I}((x_{0},y_{0}),(x,y^{\prime})) -2\mathfrak{I}((x,y),(x_{0},y_{0})) -2\mathfrak{I}((x,y_{0}),(x_{0},y^{\prime})) -2\mathfrak{I}((x_{0},y),(x,y_{0})).
	\end{align*}
	
	\begin{align*}
		&-2\langle K^{\mathfrak{I}}_{x,y} , K^{\mathfrak{I}}_{x^{\prime},y}\rangle =2\mathfrak{I}((x,y),(x^{\prime},y_{0}))+2\mathfrak{I}((x,y_{0}),(x^{\prime},y))\\
		&-2\mathfrak{I}((x_{0},y_{0}),(x^{\prime},y)) -2\mathfrak{I}((x,y),(x_{0},y_{0})) -2\mathfrak{I}((x,y_{0}),(x_{0},y)) -2\mathfrak{I}((x_{0},y),(x^{\prime},y_{0})).   
	\end{align*}

	\begin{align*}
		&2\langle K^{\mathfrak{I}}_{x,y} , K^{\mathfrak{I}}_{x^{\prime},y^{\prime}}\rangle=2\mathfrak{I}((x,y),(x^{\prime}, y^{\prime}))\\
		&-2\mathfrak{I}((x,y),(x_{0}, y^{\prime})) - 2\mathfrak{I}((x,y),(x^{\prime}, y_{0})) - 2\mathfrak{I}((x_{0},y),(x^{\prime}, y^{\prime})) - 2\mathfrak{I}((x,y_{0}),(x^{\prime}, y^{\prime}))\\
		& + 2\mathfrak{I}((x_{0},y_{0}),(x^{\prime}, y^{\prime}))+ 2\mathfrak{I}((x,y),(x_{0}, y_{0}))  + 2\mathfrak{I}((x_{0},y),(x^{\prime}, y_{0}))+ 2\mathfrak{I}((x,y_{0}),(x_{0}, y^{\prime})).    
	\end{align*}
	
	\begin{align*}
		&2\langle K^{\mathfrak{I}}_{x,y^{\prime}} , K^{\mathfrak{I}}_{x^{\prime},y}\rangle =  2\mathfrak{I}((x,y^{\prime}),(x^{\prime}, y))\\
		&- 2\mathfrak{I}((x,y^{\prime}),(x_{0}, y)) - 2\mathfrak{I}((x,y^{\prime}),(x^{\prime}, y_{0})) - 2\mathfrak{I}((x_{0},y^{\prime}),(x^{\prime}, y)) - 2\mathfrak{I}((x,y_{0}),(x^{\prime}, y))\\
		& + 2\mathfrak{I}((x_{0},y_{0}),(x^{\prime}, y))+ 2\mathfrak{I}((x,y^{\prime}),(x_{0}, y_{0}))  + 2\mathfrak{I}((x_{0},y^{\prime}),(x^{\prime}, y_{0}))+ 2\mathfrak{I}((x,y_{0}),(x_{0}, y)). 
	\end{align*}
	Summing all the above relations we obtain the desired equality.
\end{proof}

\begin{proof}[\textbf{Proof of Lemma \ref{integmargin}}] Relation $(i)$ implies relation $(ii)$ by the Fubini-Tonelli Theorem.\\
	Now suppose that $(ii)$ is valid for a pair $(z,w) \in X \times Y$ and we will prove $(i)$. Equation  \ref{4rootsinequality}   and Holder's inequality implies that in order that to prove that $\mathfrak{I} \in L^{1}((\mu \times \nu)\times (\mu \times \nu)  )$, it is sufficient that the $4$ functions appearing on the right hand side of Equation \ref{4rootsinequality}  are elements of  $L^{2}((\mu \times \nu)\times (\mu \times \nu)  )$, which occur by the hypothesis and because the measures $\mu, \nu$ are finite.\\
	If relation $(iii)$ is valid, it  is immediate that  relation $(ii)$ is valid. Conversely, if $(ii)$ is valid for a pair $(z,w) \in X \times Y$, by Fubini-Tonelli the complement of the  sets 
	$$
	X_{\mu, \nu}^{\mathfrak{I}}:= \{x \in X , \mathfrak{I}((x, \cdot),(z,w)) \in L^{1}(\nu) \}, \quad Y_{\mu, \nu}^{\mathfrak{I}}:= \{y \in Y , \mathfrak{I}((\cdot, y),(z,w)) \in L^{1}(\mu) \}  
	$$  
	have $\mu$   and $\nu $ zero measure respectively. Then, the complement of the set $X_{\mu, \nu}^{\mathfrak{I}}\times Y_{\mu, \nu}^{\mathfrak{I}}$ has $\mu\times \nu$ zero measure and if $(x^{\prime}, y^{\prime}) \in X_{\mu, \nu}^{\mathfrak{I}}\times Y_{\mu, \nu}^{\mathfrak{I}}$, the inequality on Equation \ref{4rootsinequality} implies that the function $
	\mathfrak{I}((\cdot, \cdot ),(x^{\prime}, y^{\prime}))$ is in $L^{1}(\mu\times \nu)$.
\end{proof}

On the proof of Lemma \ref{integmargin}, we choose the set $X_{\mu, \nu}^{\mathfrak{I}}\times Y_{\mu, \nu}^{\mathfrak{I}}$ based on the point $(z,w)$ such that $(x,y) \to \mathfrak{I}((x,y),(z,w)) \in L^{1}(\mu\times \nu)$. It may occur to the reader that this set depends on the point $(z,w)$. We complement the proof of Lemma \ref{integmargin} by showing that it does not depend.

\begin{proof} Our objective is to prove that if  $(z,w)$, $(z^{\prime}, w^{\prime})$ are points such that $ \mathfrak{I}((\cdot,\cdot),(z,w)) \in L^{1}(\mu\times \nu)$ and $ \mathfrak{I}((\cdot,\cdot),(z^{\prime},w^{\prime})) \in L^{1}(\mu\times \nu)$, then $A^{1}_{z,w}\times A^{2}_{z,w} = A^{1}_{z^{\prime}, w^{\prime}}\times A^{2}_{z^{\prime}, w^{\prime}}$, where
	
	$$
	A^{1}_{z,w}:= \{x \in X , \mathfrak{I}((x, \cdot),(z,w)) \in L^{1}(\nu) \}, \quad    A^{2}_{z,w}:= \{y \in Y , \mathfrak{I}((\cdot, y),(z,w)) \in L^{1}(\mu) \}, 
	$$
	$$
	A^{1}_{z^{\prime},w^{\prime}}:= \{x \in X , \mathfrak{I}((x, \cdot),(z^{\prime},w^{\prime})) \in L^{1}(\nu) \}, \quad    A^{2}_{z^{\prime},w^{\prime}}:= \{y \in Y , \mathfrak{I}((\cdot, y),(z^{\prime},w^{\prime})) \in L^{1}(\mu) \}. 
	$$
	Suppose in addition  that $(z^{\prime},w^{\prime}) \in A^{1}_{z,w} \times A^{2}_{z,w}$. By  Equation \ref{4rootsinequality} we have that
	\begin{align*}
		\sqrt{\mathfrak{I}((x,y),(z^{\prime},w^{\prime}))} & \leq \sqrt{\mathfrak{I}((x,y),(z,w))} +\sqrt{\mathfrak{I}((x,w^{\prime}),(z,w))}\\ 
		& \quad  +\sqrt{\mathfrak{I}((z^{\prime},y),(z,w))} +\sqrt{\mathfrak{I}((z^{\prime},w^{\prime}),(z,w))}.
	\end{align*}
	Since $w^{\prime} \in A^{2}_{z,w}$, if $y \in A^{2}_{z,w}$ then the previous inequality implies that $y \in A^{2}_{z^{\prime},w^{\prime}}$, then $A^{2}_{z,w} \subset A^{2}_{z^{\prime},w^{\prime}}$.  Similarly $A^{1}_{z,w} \subset A^{1}_{z^{\prime},w^{\prime}}$, which proves that $A^{1}_{z,w}\times A^{2}_{z,w} \subset A^{1}_{z^{\prime}, w^{\prime}}\times A^{2}_{z^{\prime},w^{\prime}}$. The reverse inclusion follows from the fact that $(z,w) \in A^{1}_{z,w}\times A^{2}_{z,w}$ and then $(z,w) \in A^{1}_{z^{\prime}, w^{\prime}}\times A^{2}_{z^{\prime}, w^{\prime}}$, together with the first part of the proof.\\
	Now, for arbitrary $(z,w)$, $(z^{\prime}, w^{\prime})$, since the complement of the sets  $A^{1}_{z,w} \times A^{2}_{z,w}$ and $A^{1}_{z^{\prime}, w^{\prime}}\times A^{2}_{z^{\prime}, w^{\prime}}$ has $\mu \times \nu $ zero measure, there exists  $ (z^{\prime \prime}, w^{\prime \prime}) \in (A^{1}_{z,w} \times A^{2}_{z,w}) \cap  (A^{1}_{z^{\prime}, w^{\prime}}\times A^{2}_{z^{\prime}, w^{\prime}})$. By Equation  \ref{4rootsinequality} once again, we have that $\mathfrak{I}((\cdot,\cdot),(z^{\prime \prime},w^{\prime \prime})) \in L^{1}(\mu \times \nu)$, the conclusion follows then by the first part using this intermediate pair of points.\end{proof}

\begin{proof}[\textbf{Proof of Corollary \ref{integmargincor}}] Relation $(i)$ implies relation $(ii)$ by the Fubini-Tonelli Theorem.\\
	The converse also follows from  Equation  \ref{4rootsinequality}, where the hypothesis  $\mathfrak{I} \in L^{1}((\mu \times \nu)\times (\mu \times \nu))$ is used to deal with the functions  $\sqrt{\mathfrak{I}((x,y^{\prime}),(z,w))}$ and  $\sqrt{\mathfrak{I}((x^{\prime},y),(z,w))}$.\\
	If relation $(iii)$ is valid is immediate that  $(ii)$ is valid. The converse is very similar to the proof of Lemma \ref{integmargin} as  if $(x^{\prime}, y^{\prime}) \in X_{\mu, \nu}^{\mathfrak{I}}\times Y_{\mu, \nu}^{\mathfrak{I}}$, the four functions on the right hand side of  Equation \ref{4rootsinequality} are elements of  $L^{2}(|\lambda|)$.\end{proof}

\begin{proof}[\textbf{Proof of Theorem \ref{principalmetric} }]  First we prove that $\mathfrak{M}_{\mu, \nu}(\mathfrak{I})$ is a vector space. Clearly, the measures $\lambda \in \mathfrak{M}(X\times Y)$ that satisfies that $\lambda(X, \cdot)$, $\lambda(\cdot, Y)$ are the zero measures is a vector space.  The requirements  $|\lambda|_{X} \subset \mu$ is also satisfied by a vector space, because if $\lambda_{1}, \lambda_{2}$ satisfies this condition,  then for every $c \in \mathbb{R}$ we have that  $|\lambda_{1}| +|c||\lambda_{2}|- |\lambda_{1} +c\lambda_{2}|$ is a nonnegative measure, and then $|\lambda_{1} +c\lambda_{2}|_{X} \subset \mu$. Similar for $|\lambda|_{Y}$.\\
	Now, for the condition $\mathfrak{I} \in L^{1}(|\lambda|\times |\lambda|)$ we use Corollary \ref{integmargincor}. Indeed, if  $\mathfrak{I} \in L^{1}(|\lambda_{1}|\times |\lambda_{1}|)$ and $\mathfrak{I} \in L^{1}(|\lambda_{2}|\times |\lambda_{2}|)$, since $\mathfrak{I}$ is nonnegative, in order to prove that  $\mathfrak{I} \in L^{1}(|\lambda_{1} + c\lambda_{2}|\times |\lambda_{1} + c\lambda_{2}|)$ for every $c \in \mathbb{R}$, it is sufficient to prove that  $\mathfrak{I} \in L^{1}((|\lambda_{1}| + |c||\lambda_{2}|)\times (|\lambda_{1}| + |c||\lambda_{2}|))$.\\
	By Corollary \ref{integmargincor}, since  $ \mathfrak{I}((\cdot, \cdot ),(x^{\prime},y^{\prime})) \in  L^{1}(|\lambda_{1}| ) \cap L^{1}(|\lambda_{2}| )$ for every fixed $(x^{\prime}, y^{\prime}) \in X_{\mu, \nu}^{\mathfrak{I}} \times Y_{\mu, \nu}^{\mathfrak{I}}$, we have that   $ \mathfrak{I}((\cdot, \cdot ),(x^{\prime},y^{\prime})) \in L^{1}((|\lambda_{1}| + |c||\lambda_{2}|)$  for every fixed $(x^{\prime}, y^{\prime}) \in X_{\mu, \nu}^{\mathfrak{I}} \times Y_{\mu, \nu}^{\mathfrak{I}}$. Applying  Corollary \ref{integmargincor} once again and we obtain that $\mathfrak{I} \in L^{1}((|\lambda_{1}| + |c||\lambda_{2}|)\times (|\lambda_{1}| + |c||\lambda_{2}|))$.\\
	For the second statement, if $(z,w) \in X_{\mu, \nu}^{\mathfrak{I}}\times Y_{\mu, \nu}^{\mathfrak{I}}$ then the four functions on the right hand side of Equation  \ref{4rootsinequality} belongs to $L^{2}(|\lambda|\times |\lambda^{\prime}|)$, where in particular
	\begin{align*}
		0\leq &\int_{X\times Y}\int_{X\times Y}\mathfrak{I}((x,y^{\prime}),(z,w))d|\lambda|(x,y)d|\lambda^{\prime}|(x^{\prime}, y^{\prime})\\
		&=   \int_{ Y}\int_{X}\mathfrak{I}((x,y^{\prime}),(z,w))d|\lambda|_{X}(x)d|\lambda^{\prime}|_{Y}( y^{\prime})< \infty 
	\end{align*}
	because $|\lambda|_{X} \times |\lambda^{\prime}|_{Y} \subset \mu\times \nu $ and $(\cdot, \cdot) \to \mathfrak{I}((x,y^{\prime}),(z,w)) \in L^{1}(\mu\times \nu)$. Then, the real valued function is well defined.\\
	The bilinearity is immediate, while for the property that $(\lambda, \lambda) \geq 0$ we use  Lemma \ref{PDItoPD}. Let $(x_{0}, y_{0}) \in X_{\mu, \nu}^{\mathfrak{I}}\times Y_{\mu, \nu}^{\mathfrak{I}}$ and the PD kernel $K^{\mathfrak{I}}$  provided by the Lemma. Note that since \\ $\sqrt{K^{\mathfrak{I}}((x,y),(x,y))}=\sqrt{4\mathfrak{I}((x,y),(x_{0},y_{0}))} \in L^{2}(|\lambda|)$, by the kernel mean embedding on Lemma \ref{initialextmmddominio}, we have that 
	$$
	\int_{X\times Y} \int_{X\times Y} K^{\mathfrak{I}}((x,y),(x^{\prime}, y^{\prime}))d\lambda(x,y)d\lambda(x^{\prime}, y^{\prime}) = \langle K^{\mathfrak{I}}_{\lambda}, K^{\mathfrak{I}}_{\lambda} \rangle_{\mathcal{H}_{K^{\mathfrak{I}}}} \geq 0
	$$
	On the other hand, the $9$ functions on the right hand side of the definition of the kernel $K^{\mathfrak{I}}$ on Lemma \ref{PDItoPD}  are in $L^{1}(|\lambda|\times|\lambda|)$ (the other $7$ are zero because its projections kernels are zero), but only the first one the double integral with respect to $d\lambda(x,y)d\lambda(x^{\prime}, y^{\prime})$ is (necessarily) nonzero. Indeed, we present the argument for the functions $\mathfrak{I}((x,y),(x_{0}, y^{\prime}))$ and $\mathfrak{I}((x,y),(x_{0}, y_{0}))$, as the others $6$ functions are similar.\\
	By the choice of the point $(x_{0}, y_{0})$, since $ \mathfrak{I}((\cdot,\cdot),(x_{0}, y_{0})) \in L^{1}(|\lambda|) $, we have that the double integral is well defined and
	\begin{align*}
		\int_{X\times Y} \int_{X\times Y} &\mathfrak{I}((x,y),(x_{0}, y_{0}))d\lambda(x,y)d\lambda(x^{\prime}, y^{\prime})\\
		&= \left [\int_{X\times Y}  \mathfrak{I}((x,y),(x_{0}, y_{0}))d\lambda(x,y)\right ] \left [\int_{X\times Y}1d\lambda(x^{\prime}, y^{\prime}) \right ] =0,
	\end{align*}
	because $\lambda(X\times Y)=0$.\\
	For the function $\mathfrak{I}((x,y),(x_{0}, y^{\prime}))$, it is integrable because of the inequality
	$$
	\sqrt{ \mathfrak{I}((x,y),(x_{0}, y^{\prime}))} \leq \sqrt{\mathfrak{I}((x,y),(x_{0}, y_{0}))} + \sqrt{\mathfrak{I}((x,y_{0}),(x_{0}, y^{\prime}))}.
	$$
	The integral
	$$
	\int_{X\times Y}\mathfrak{I}((x,y),(x_{0}, y^{\prime}))d\lambda(x^{\prime}, y^{\prime}) =0, \quad (x,y) \in X_{\mu, \nu}^{\mathfrak{I}}\times Y_{\mu, \nu}^{\mathfrak{I}}
	$$
	because $\lambda(X, \cdot)$ is the zero measure. Then
	$$
	\int_{X\times Y} \int_{X\times Y}\mathfrak{I}((x,y),(x_{0}, y^{\prime}))d\lambda(x, y)d\lambda(x^{\prime}, y^{\prime})=0.
	$$
\end{proof}

\begin{proof}[\textbf{Proof of Lemma \ref{integmarginkroeprod}}] 
	The converse is immediate, because we can use the $3$ equivalence relations in Lemma  \ref{estimativa}.\\
	For the opposite relation, we assume  that $\gamma$ and $\varsigma$ are zero at the diagonal (equivalently, the projections kernels of $\gamma \otimes \varsigma$ are zero).  Because of Lemma \ref{integmargin}, it is sufficient that we prove that the set $X_{\mu, \nu}^{\mathfrak{I}}=X$ and $Y_{\mu, \nu}^{\mathfrak{I}}=Y$. \\
	Given arbitrary $(x^{\prime}, y^{\prime})$ such that $\mathfrak{I}((\cdot, \cdot), (x^{\prime}, y^{\prime})) \in L^{1}(\mu \times \nu)$, by the definition and invariance of $X_{\mu, \nu}^{\mathfrak{I}}$ we have that 
	\begin{align*}
		X_{\mu, \nu}^{\mathfrak{I}}&= \{x \in X, \quad \gamma(x,x^{\prime})\varsigma(\cdot ,y^{\prime}) \in L^{1}(\nu)\}\\
		&= \{x \in X, \quad \gamma(x,x^{\prime})=0 \} \cup \{x \in X, \quad \gamma(x,x^{\prime})\neq 0,  \varsigma(\cdot ,y^{\prime}) \in L^{1}(\nu) \}.
	\end{align*}
	If the set $\{x \in X, \quad \gamma(x,x^{\prime})\neq 0,  \varsigma(\cdot ,y^{\prime}) \in L^{1}(\nu) \}$ is nonempty, then $ \varsigma(\cdot ,y^{\prime}) \in L^{1}(\nu)$ and consequently $X_{\mu, \nu}^{\mathfrak{I}} = \{x \in X, \quad \gamma(x,x^{\prime})=0 \} \cup \{x \in X, \quad \gamma(x,x^{\prime})\neq 0\} = X$. On the other hand, if this set is empty, we have that $X_{\mu, \nu}^{\mathfrak{I}}= \{x \in X, \quad \gamma(x,x^{\prime})=0 \} =\{x^{\prime}\}$, but $\mu(X-X_{\mu, \nu}^{\mathfrak{I}})=0$, so $\mu = c\delta_{x^{\prime}}$, which does not occur by the hypothesis.  \\
	For the general case, first note that
	$$
	0 \leq \gamma(x, x^{\prime})- \gamma(x,x)/2 - \gamma(x^{\prime},x^{\prime})/2 \leq \gamma(x,x^{\prime}), \quad x,x^{\prime} \in X
	$$
	and similarly for $\varsigma$. Hence, if   $\gamma\otimes\varsigma \in L^{1}((\mu\times \nu)\times (\mu \times \nu))$, then the PDI kernel
	$$
	(\gamma(x, x^{\prime})- \gamma(x,x)/2 - \gamma(x^{\prime},x^{\prime})/2 )\otimes(\varsigma(y, y^{\prime})- \varsigma(y,y)/2 - \varsigma(y^{\prime},y^{\prime})/2),
	$$
	is an element of $ L^{1}((\mu\times \nu)\times (\mu \times \nu))$ and its projections kernels are zero. By the first part of the proof we obtain that $x \to  \gamma(x, x^{\prime}) - \gamma(x,x)/2 - \gamma(x^{\prime}, x^{\prime}) \in L^{1}(\mu)$ for every $x^{\prime} \in X$. Since $\gamma$ is bounded at the diagonal, we obtain that $\gamma(\cdot, x^{\prime}) \in L^{1}(\mu)$ for every $x^{\prime} \in X$. The proof for $\varsigma$ follows by similar arguments
	.\end{proof}

\begin{proof}[\textbf{Proof of  Corollary \ref{integmarginkroeprodcor}}]
	If $\gamma$ and $\varsigma$ are zero at the diagonal, the first claim is a consequence of Corollary \ref{integmargincor}, as $X_{\mu, \nu}=X$ and $Y_{\mu, \nu}=Y$ by Lemma \ref{integmarginkroeprod}. In the general case, since 
	$$
	0 \leq (\gamma(x,x^{\prime}) - \gamma(x^{\prime},x^{\prime})/2 - \gamma(x,x)/2)(\varsigma(y,y^{\prime}) - \varsigma(y^{\prime},y^{\prime})/2 - \varsigma(y,y)/2)	\leq \gamma(x,x^{\prime})\varsigma(y,y^{\prime}), 
	$$
	we obtain by the first part of the proof that  
	$$
	(x,y) \to (\gamma(x,x^{\prime}) - \gamma(x^{\prime},x^{\prime})/2 - \gamma(x,x)/2)(\varsigma(y,y^{\prime}) - \varsigma(y^{\prime},y^{\prime})/2 - \varsigma(y,y)/2) \in L^{1}(|\lambda|)
	$$ 	
	for every $x^{\prime} \in X$ and $y^{\prime} \in Y$.  By the hypothesys that $\gamma$  and $\varsigma$ are bounded at the diagonal and the integrability of $ \gamma$ and $\varsigma$ with respect to the marginals of $|\lambda|$, we obtain that $(x,y) \to \gamma(x,x^{\prime})\varsigma(y,y^{\prime}) \in L^{1}(|\lambda|)$ for every $x^{\prime} \in X$ and $y^{\prime} \in Y$.\\
	The fact that $	\mathfrak{M}(\gamma \otimes \varsigma)$ is a vector space is a direct consequence that the equivalence in Lemma \ref{integmarginkroeprod} and the first part of the Corollary holds for every $x^{\prime} \in X$ and $y^{\prime} \in Y$.\\
	For the final claim, the double integral is well defined because $\gamma \otimes \varsigma \in L^{1}((|\lambda| + |\lambda^{\prime}|) \times (|\lambda| + |\lambda^{\prime}|))$. It is a semi-inner product because
	\begin{align*}
		&\int_{X\times Y} \int_{X\times Y} \gamma(x,x^{\prime}) \varsigma(y,  y^{\prime})d\lambda (x,y)d\lambda(x^{\prime}, y^{\prime})\\
		&= \int_{X\times Y} \int_{X\times Y} (\gamma(x,x^{\prime}) - f(x^{\prime}) -f(x)) (\varsigma(y,y^{\prime}) - g(y^{\prime}) -  g(y))d\lambda (x,y)d\lambda(x^{\prime}, y^{\prime}) \geq 0.	
	\end{align*}
	where $f(z)= \gamma(z,z)/2$ and $g(w)= \varsigma(w,w)/2$.\\
	Now, we prove that if either $\gamma$ or $\varsigma$ is not CND-Characteristic then the semi-inner product is not an inner product. Indeed, if $\gamma$ is not CND-Characteristic there exist a nonzero measure  $\lambda_{1} \in \mathfrak{M}_{1}(X; \gamma)$  such that:
	$$\int_{X} \int_{X}\gamma(x,x^{\prime})d\lambda_{1}(x)d\lambda_{1}(x^{\prime})=0.
	$$
	Pick an arbitrary nonzero and   $\lambda_{2} \in \mathfrak{M}_{1}(Y; \varsigma)$ and define the measure  $\lambda= \lambda_{1}\times\lambda_{2}$, which is a nonzero element of $\mathfrak{M}(\gamma\otimes \varsigma)$. The conclusion that $(\lambda, \lambda)=0$ comes from Fubini-Tonelli. If $varsigma$ is not CND-Characteristic the proof is identical.\\
	We prove the  converse  separately, as the arguments are different and some properties of RKHS will be needed.	\end{proof}

First, note that if $K: X \times  X \to \mathbb{R}$ is a positive definite kernel and $(\phi_{i})_{i \in \mathcal{I}}$ is a complete  orthonormal basis for $\mathcal{H}_{K}$, then the following pointwise convergence  holds
\begin{equation}\label{Mercergensum}
	K(x,x^{\prime})= \sum_{i \in \mathcal{I}}\phi_{i}(x)\phi_{i}(x^{\prime}), \quad x, x^{\prime} \in X.
\end{equation}

This occur because since $K_{x^{\prime}} \in \mathcal{H}_{K}$ for every $x^{\prime} \in X$, there exists $a_{i}(x^{\prime}) \in \mathbb{R}$ for which  $K_{x^{\prime}}= \sum_{i \in \mathcal{I}}a_{i}(x^{\prime})\psi_{i}$, with convergence in $\mathcal{H}_{K}$. Since 
$$
a_{i}(x^{\prime})= \langle K_{x^{\prime}}, \psi_{i} \rangle_{\mathcal{H}_{K}} = \psi_{i}(x^{\prime}) 
$$
and convergence in $\mathcal{H}_{K}$ implies pointwise convergence for every point in $X$, Equation  \ref{Mercergensum} is valid. 

In Section $7$ in \cite{gaussinfi} it is proved that if $X$ is a Hausdorff space, $\lambda \in \mathfrak{M}(X)$ and  there exists a continuous CND metrizable kernel on $X$, then  the RKHS of any continuous PD kernel on $X$ is separable when restricted to    $Supp(\lambda)$. This property,  Equation \ref{Mercergensum} and  Equation \ref{Kgamma}  implies that if $\gamma: X \times X \to \mathbb{R}$ is a continuous CND metrizable kernel that is zero at the diagonal and $ \lambda \in \mathfrak{M}_{1}(X; \gamma)$, then the pointwise convergence holds
\begin{equation}
	\gamma(x,x^{\prime})= \sum_{k=0}^{\infty}(\phi_{k}(x) - \phi_{k}(x^{\prime}))^{2}, \quad  x,x^{\prime} \in supp(\lambda)	
\end{equation}

\begin{proof}[\textbf{Continuation of the Proof of Corollary \ref{integmarginkroeprodcor}}]
	Without loss of generalization we may assume that $\gamma$, $\varsigma$ are zero at the diagonal.\\
	Suppose that $\gamma$ and $\varsigma$ are CND-Characteristic and that  $\lambda \in \mathfrak{M}(\gamma\otimes \varsigma)$  is such that $(\lambda, \lambda)=0$. Since 	
	\begin{align*}
		&\int_{X\times Y} \int_{X\times Y} \gamma(x,x^{\prime})\varsigma(y,y^{\prime})d\lambda(x,y)d\lambda(x^{\prime}, y^{\prime})\\
		&=\int_{X\times Y} \int_{X\times Y}\left [ \sum_{k=0}^{\infty}(\phi_{k}(x) - \phi_{k}(x^{\prime}))^{2} \right ]\varsigma(y,y^{\prime})d\lambda(x,y)d\lambda(x^{\prime}, y^{\prime})\\
		&=\sum_{k=0}^{\infty} \int_{X\times Y} \int_{X\times Y}(\phi_{k}(x) - \phi_{k}(x^{\prime}))^{2} \varsigma(y,y^{\prime})d\lambda(x,y)d\lambda(x^{\prime}, y^{\prime})	
	\end{align*}
	and for each $k \in \mathbb{Z}_{+}$ the double integral is nonnegative because the kernel is PDI,  we obtain that each term in this series is equal to zero. \\
	Without loss of generalization we  may assume that there exists a $z \in X$ for which $\phi_{k}(z)=0$ for all $k \in \mathbb{Z}_{+}$.  From the following inequalities 
	\begin{align*}
		[\phi_{k}(x)^{2} + \phi_{k}(x^{\prime})^{2} ]\varsigma(y,y^{\prime}) &\leq [\gamma(x,z) + \gamma(x^{\prime},z)][ 2\varsigma(y,w) + 2\varsigma(y^{\prime},w)]\\
		|\phi_{k}(x)\phi_{k}(x^{\prime})|\varsigma(y,y^{\prime})&\leq  [\gamma(x,z) + \gamma(x^{\prime},z)][ 2\varsigma(y,w) + 2\varsigma(y^{\prime},w)],
	\end{align*}
	we obtain that the functions on the left hand side are in $L^{1}(|\lambda|\times |\lambda|)$, where the functions on the right hand side are elements of $L^{1}(|\lambda|\times |\lambda|)$ because of the equivalence in the first part of the Corollary (functions $\gamma(x,z)\varsigma(y,w)$, $\gamma(x^{\prime},z)\varsigma(y^{\prime},w)$ ) and also because    $\gamma \in L^{1}(|\lambda|_{X}\times |\lambda|_{X})$, $\varsigma \in L^{1}(|\lambda|_{Y} \times |\lambda|_{Y})$ (functions $\gamma(x,z)\varsigma(y^{\prime},w)\varsigma(y,w)$, $\gamma(x^{\prime},z)\varsigma(y,w)$).\\
	However, 
	$$
	\int_{X\times Y} \phi_{k}(x^{\prime})^{2}  \varsigma(y,y^{\prime})d\lambda(x,y)=0, \quad k \in \mathbb{Z}_{+}, \quad y^{\prime} \in Y
	$$
	because $ \lambda(X, \cdot)$ is the zero measure. Hence,
	$$
	\int_{X\times Y} \int_{X\times Y}\phi_{k}(x)\phi_{k}(x^{\prime}) \varsigma(y,y^{\prime})d\lambda(x,y)d\lambda(x^{\prime}, y^{\prime})=0, \quad k \in \mathbb{Z}_{+}.
	$$
	Define the function
	$$
	\lambda_{k}(A):= \int_{X \times A} \phi_{k}(x)d\lambda(x,y), \quad   A \in \mathscr{B}(Y).
	$$
	Note that $\lambda_{k} \in \mathfrak{M}(Y)$, because $|\phi_{k}(x)| \leq 1 +\phi_{k}(x)^{2} \leq 1 + \gamma(x,z)$ and $\gamma \in L^{1}(|\lambda|_{X}\times |\lambda|_{X})$. Also, by the previous relations $\lambda_{k} \in \mathfrak{M}_{1}(Y; \varsigma)$. But
	$$
	0=  \int_{X\times Y}\int_{X\times Y}\phi_{k}(x) \phi_{k}(x^{\prime})  \varsigma(y,y^{\prime})d\lambda(x,y)d\lambda(x^{\prime},y^{\prime})= \int_{Y}\int_{Y}\varsigma(y,y^{\prime})d\lambda_{k}(y)d\lambda_{k}(y^{\prime}),
	$$
	as the kernel $\gamma$ is CND-Characteristic, $\lambda_{k} $ is the zero measure for every $k \in \mathbb{Z}_{+}$. \\
	In particular, for every $A   \in \mathscr{B}(Y)$ 
	\begin{align*}
		&\int_{X}\int_{X}\gamma(x,x^{\prime})d\lambda_{A}(x)d\lambda_{A}(x^{\prime})=\int_{X\times A}\int_{X\times A}\gamma(x,x^{\prime})d\lambda(x,y)d\lambda(x^{\prime},y^{\prime})\\
		&= \sum_{k=0}^{\infty}\int_{X\times A}\int_{X\times A}(\phi_{k}(x) - \phi_{k}(x^{\prime}))^{2}d\lambda(x,y)d\lambda(x^{\prime},y^{\prime})\\
		&=-2\sum_{k=0}^{\infty}\int_{X\times A}\int_{X\times A}\phi_{k}(x) \phi_{k}(x^{\prime})d\lambda(x,y)d\lambda(x^{\prime},y^{\prime})=-2\sum_{k=0}^{\infty} \lambda_{k}(A)^{2}=0,
	\end{align*}
	but since $\gamma$ is CND-Characteristic, we obtain that the measure  $\lambda_{A}$, which is an element of $\mathfrak{M}_{1}(X;\gamma)$, is the zero measure for every $A \in \mathscr{B}(Y)$, hence, $\lambda$ is the zero measure.
\end{proof}

\subsection{\textbf{Section \ref{Bernstein  functions of two variables}} }

In some sense, the following result characterizes  measure valued conditionally negative definite kernels defined on a finite set $\{1, \ldots , n\}$. The proof is very similar to Lemma $3.6$ in \cite{jeanrad}, where the focus was on  positive definite kernels. 

\begin{lem}\label{improve} Let $\sigma_{i, j}$, $i, j =1, \ldots,n$, be  measures on $\mathfrak{M}(X)$ fulfilling the following properties: $\sigma_{i, i}$ is the zero measure and $\sigma_{i,j}= \sigma_{j,i}$ for every $i,j$; the measure $ -\sum_{i, j=1}^{n} c_{i} c_{j} \sigma_{i ,j}$ is nonnegative whenever $\sum_{i=1}^{n}c_{i}=0$. If
	$$
	\sigma_{S}:= \sum_{i,j=1}^{n}\sigma_{i,j},
	$$
	then the measure $\sigma_{S}$ is nonnegative, there exist functions $g_{i,j}$, $i,j=1,\ldots,n$ in $L^{1}( \sigma_{S})$ so that
	$$
	d\sigma_{i, j} = g_{i,j}d\sigma_{S}.
	$$
	Further, we may assume that for each $r\in X$ the matrix $[g_{i, j}(r)]_{i,j =1}^{n}$ is CND and $g_{i, i}(r)=0$ for every $i$.
\end{lem}

\begin{proof}[\textbf{Proof of Lemma \ref{improve}}] Our assumptions on the measures $\sigma_{i, j}$ asserts that the matrix $[\sigma_{i, j}(A)]_{i, j =1}^{n}$ is CND for all Borel sets $A \in \mathscr{B}(X)$. In particular,
	$$
	2 \sigma_{i ,j}(A) \geq -\sigma_{i, i}(A) - \sigma_{j, j}(A)=0,
	$$
	for all $i,j=1,\ldots,n$, which implies that all measures $\sigma_{i,j}$ are nonnegative, hence $\sigma_{S}$ is nonnegative as well. A direct  application of the Radon-Nykodin  Theorem  provides the existence of functions $h_{i, j}$, $i,j=1,\ldots,n$, in the space $L^{1}( \sigma_{S})$, so that
	$$ \sigma_{i,j}(A) = \int_{A} h_{i, j}(r) d\sigma_{S}(r), $$
	for all $i,j=1,\ldots,n$ and  $A \in \mathscr{B}(X)$. 	Since $\sigma_{i ,j}= \sigma_{j, i}$, we can assume that $h_{i, j}(r)=h_{j, i}(r)$ for all $r \in X$. In order to prove the last statement of the theorem, first observe that if $v=(v_{1}, \ldots, v_{n})$ is a nonzero vector on $\mathbb{R}^{n}$, with $\sum_{i=1}^{n}v_{i}=0$, and  $A \in \mathscr{B}(X)$, then
	$$
	\int_{A} \langle H(r) v , v \rangle  d\sigma_{S}(r) = \langle \sigma(A) v, v \rangle  \leq 0.
	$$
	Hence, since the matrices $H(r):=[h_{i ,j}(r)]_{i, j=1}^{n}$ and   $\sigma(A):=[\sigma_{i ,j}(A)]_{i, j=1}^{n}$ are symmetric, for each such $v$,
	$$
	B_{v}:= \{ r \in X: \langle H(r) v , v \rangle >0 \}
	$$
	is a $\sigma_{S}$-null subset of $X$. We now infer that
	$$
	B:= \bigcup_{v \in \Lambda} B_{v}
	$$
	is also a $\sigma_{S}$-null subset of $X$, where $\Lambda = \{v \in \mathbb{R}^{n}\setminus{\{0\}} \text{ and } \sum_{i=1}^{n}v_{i}=0\}$. Indeed, if $s \in B$, then $s\in B_v$ for some $v \in \Lambda$. Since
	$\langle H(s) v , v \rangle > 0$ and the function $w \in \mathbb{R}^{n} \to \langle H(s) w , w \rangle$ is continuous, we can select a neighborhood $U_v$ of $v$ on $\Lambda$ so that $
	\langle H(s) u , u \rangle >0, \quad u \in U_v$. \\
	Now, pick a dense subset $\{v_1, v_2, \ldots\}$ of $\Lambda$. Since $v_m \in U_v$ for some $m$, it follows that $\langle H(s) v_{m} , v_{m} \rangle >0 $ for at least one $m$. In other words,
	$s\in B_{v_m}$, for at least one $m$. Therefore, $B\subset \bigcup _{m \in \mathbb{N}} B_{v_m}$ and, consequently, $\sigma_{S}(B)=0$. The arguments above reveal that $H(r)$ is CND, except for $r$ in a $\sigma_{S}$-null subset $B$ of $X$. Since  the functions $h_{i,i}$ are zero  almost everywhere, we can define the define the functions $g_{i,j}$  that fulfills the desired properties by defining it to be zero on those sets of zero measure and $h_{i,j}$ otherwise. 
\end{proof}

\begin{proof}[\textbf{Proof of Theorem \ref{basicradialuni}}] Suppose that the kernel is PDI and its projection kernels are zero. By Equation \ref{eqcndker}, for every $u,v \in X$ the kernel $\gamma_{u,v}(x,y):=f(u,v, \|x-y\|)$ is CND on $\mathbb{R}^{d}$ for every $d \in \mathbb{N}$. Theorem \ref{reprcondneg} implies that  the kernel admits the desired representation and that it is unique, it only remains to prove the properties of the kernel  $\sigma_{u,v}$. For that, let  $u_{1}, \ldots , u_{n} \in X$ and scalars $c_{1}, \ldots , c_{n} \in \mathbb{R} $ with the restriction that $\sum_{i=1}^{n}c_{i}=0$, by  Equation \ref{CNDprojgeral}, the kernel 
	$$
	-\sum_{i,j=1}^{n}c_{i}c_{j}f(u_{i}, u_{j}, \|x-y\|), \quad x,y \in \mathbb{R}^{d}
	$$   
	is  CND for every $d \in \mathbb{N}$. Since the representation on Theorem \ref{reprcondneg} is unique, we obtain that $\sum_{i,j=1}^{n}c_{i}c_{j}\sigma_{u_{i}, u_{j}}$ is a nonpositive measure, which proves our claim.\\
	For the converse,  let $x_{1}, \dots , x_{m} \in \mathbb{R}^{d}$, $u_{1}, \ldots, u_{n} \in X$ and scalars $c_{i,k} \in \mathbb{R}$ with the restriction that  $\sum_{i=1}^{n}c_{i,k}= \sum_{l=1}^{m}c_{j,l}=0$ for every $1\leq k \leq m$ and  $1\leq j \leq n$.\\  By the hypothesis and Lemma  \ref{improve} we can write $d\sigma_{u_{i}, u_{j}}= g_{i,j}d\sigma_{S}$, where $\sigma_{S}$ is a nonnegative measure in $\mathfrak{M}([0, \infty))$, $g_{i,j} \in L^{1}(\sigma_{S})$, for every $r \in [0, \infty)$ the matrix $[g_{i,j}(r)]_{i,j=1}^{n}$ is CND and $g_{i,i}=0$ for every $i$. In particular
	\begin{align*}
		\sum_{i,j=1}^{n}\sum_{k,l=1}^{m}c_{i, k}c_{j, l}&f(u_{i}, u_{j}, \|x_{k}-x_{l}\|)\\
		&= \int_{[0, \infty)} \left [\sum_{i,j=1}^{n}\sum_{k,l=1}^{m}c_{i, k}c_{j, l} \frac{(1-e^{-r\|x_{k}-x_{l}\|^{2}})}{r}g_{u_{i}, u_{j}}(r) \right ](1 +r)d\sigma_{S}(r). 
	\end{align*}
	As proved at the beginning of Section \ref{positive definite independent kernels},  the Kronecker product of CND kernels is PDI, so  we have  that 
	$$
	\sum_{i,j=1}^{n}\sum_{k,l=1}^{m}c_{i, k}c_{j, l} \frac{(1-e^{-r\|x_{k}-x_{l}\|^{2}})}{r}g_{u_{i}, u_{j}}(r) \geq 0, \quad r \in [0,\infty)
	$$
	which proves our claim, since the measure $(1+r)d\sigma_{S}$ is nonnegative.
\end{proof}

\begin{proof} [\textbf{Proof of Theorem \ref{bernssev}}]
	Without loss of generalization we may suppose that $g$ is zero at the boundary. Suppose that the  function $g$ admits the representation 
	\begin{align*}
		g(t_{1}, t_{2})= \int_{[0,\infty)^{2}}\left (\frac{1-e^{-r_{1}t_{1}}}{r_{1}} \right ) \left ( \frac{1-e^{-r_{2}t_{2}}}{r_{2}} \right )\prod_{i=1}^{2}(1 +r_{i})d\sigma(r_{1},r_{2}).
	\end{align*}
	Differentiation under the integral sign implies that  $\partial^{\beta}g$, where $\beta=(1,1)$, is  well defined on $(0,\infty)^{2}$ and is equal to  
	$$
	\int_{[0,\infty)^{2}}e^{-r_{1}t_{1}}e^{-r_{2}t_{2}}\prod_{i=1}^{2}(1 +r_{i})d\sigma(r_{1},r_{2}),
	$$
	which is a completely monotone function of two variables by Theorem \ref{Bochnercomplsev}. It is immediate that $g$ is zero at the boundary. Also, the uniqueness in Theorem \ref{Bochnercomplsev} implies that the representation for Bernstein functions of two variables is unique.\\
	Our proof for the converse is very similar to the proof of  Theorem $3.2$ in \cite{bers}. Suppose that $\partial^{\beta}g$, $\beta=(1,1)$, is a completely monotone function of two variables,  Theorem \ref{Bochnercomplsev} implies that 
	\begin{align*}
		\partial^{\beta}g(t_{1}, t_{2}) &=\int_{[0, \infty)^{2}} e^{-r_{1}t_{1} - r_{2}t_{2}}d\sigma^{\prime}(r_{1}, r_{2}),
	\end{align*}
	for some nonnegative Borel measure $\sigma^{\prime}$ (possibly unbounded) on $[0, \infty)^{2}$. By the fundamental Theorem of calculus in two variables, we have that for $t_{1},t_{2}, c_{1}, c_{2} >0$
	$$
	\int_{[c_{1},t_{1}]}\int_{[c_{2}, t_{2}]}\partial^{\beta}g(z,w)dzdw=g(t_{1},t_{2}) + g(c_{1},c_{2}) - g(c_{1},t_{2}) - g(t_{1},c_{2}),  
	$$
	by letting $c_{1},c_{2}\to 0$ and using the fact that $g$ is continuous on $[0, \infty)\times [0, \infty)$  and $\partial^{\beta}g$ is a nonnegative function we get that
	\begin{align*}
		g(t_{1}, t_{2})&= g(t_{1},t_{2}) + g(0,0) - g(0,t_{2}) - g(t_{1},0)= \int_{(0,t_{1}]}\int_{(0, t_{2}]}\partial^{\beta}g(z,w)dzdw\\
		&=\int_{(0,t_{1}]}\int_{(0, t_{2}]}\left [\int_{[0, \infty)^{2}} e^{-r_{1}z - r_{2}w }d\sigma^{\prime}(r_{1},r_{2})\right ]dzdw\\
		&= \int_{[0, \infty)^{2}} \left [ \int_{(0,t_{1}]}\int_{(0, t_{2}]} e^{-r_{1}z - r_{2}w }dzdw\right ]d\sigma^{\prime}(r_{1},r_{2})\\
		&=\int_{[0, \infty)^{2}} \frac{(1-e^{-r_{1}t_{1}})}{r_{1}}\frac{(1-e^{-r_{2}t_{2}})}{r_{2}}d\sigma^{\prime}(r_{1},r_{2}).
	\end{align*}
	Similar to the proof of Theorem $3.2$ in \cite{bers}, the measure $d\sigma(r_{1},r_{2}):=\frac{1}{1+r_{1}}\frac{1}{1+r_{2}}d\sigma^{\prime}(r_{1},r_{2})$ is finite because 
	$$
	\frac{1}{1+s} \leq \frac{1-e^{-s}}{s}\leq 2\frac{1}{1+s}, \quad s\geq 0.
	$$
	
\end{proof}

\begin{proof}[\textbf{Proof of Theorem \ref{basicradial}}] Suppose that the kernel is PDI for every $d \in \mathbb{N}$. By Theorem \ref{basicradialuni} we have that 
	$$
	f( \|x-y\|, \|u-v\|)= \int_{[0,\infty)}\frac{(1-e^{-r_{1}\|x-y\|^{2}})}{r_{1}} (1+r_{1})d\sigma_{u,v}(r_{1}). 
	$$
	Note that since the representation on Theorem \ref{reprcondneg} is unique, if $\|u-v\|=\|u^{\prime} - v^{\prime}\|$ then  $\sigma_{u,v}= \sigma_{u^{\prime}, v^{\prime}}$. But, as also mentioned on Theorem \ref{basicradialuni}, the kernel $\sigma_{u,v}(A)$  is CND for every $A \in \mathscr{B}([0, \infty))$ and is zero on the diagonal. By   Theorem \ref{reprcondneg} once again, we have that 
	$$
	\sigma_{u,v}(A)=  \int_{[0,\infty)}\frac{(1-e^{-r_{2}\|u-v\|^{2}})}{r_{2}}(1+r_{2})d\sigma_{A}(r_{2})
	$$
	where $\sigma_{A}$ is a nonnegative measure in $\mathfrak{M}([0, \infty))$. Note that  $\sigma_{\emptyset}$ is the zero measure. Since  $\sigma_{u,v}$ is a measure we have that if $(A_{n})_{n \in \mathbb{N}}$ is a disjoint sequence of Borel measurable sets in $[0, \infty)$, then
	\begin{align*}
		\int_{[0,\infty)}\frac{(1-e^{-r_{2}\|u-v\|^{2}})}{r_{2}}(1+r_{2})& d\sigma_{\cup_{n \in \mathbb{N}}A_{n}}(r_{2})= \sigma_{u,v}(\cup_{n \in \mathbb{N}}A_{n}) = \sum_{n \in \mathbb{N}}\sigma_{u,v}(A_{n})\\
		&= \sum_{n \in \mathbb{N}}\int_{[0,\infty)}\frac{(1-e^{-r_{2}\|u-v\|^{2}})}{r_{2}}(1+r_{2})d\sigma_{A_{n}}(r_{2})\\
		&= \int_{[0,\infty)}\frac{(1-e^{-r_{2}\|u-v\|^{2}})}{r_{2}}(1+r_{2})d\left [\sum_{n \in \mathbb{N}} \sigma_{A_{n}}\right ](r_{2}),
	\end{align*}
	since the representation is unique we obtain that $\sigma_{\cup_{n \in \mathbb{N}}A_{n}} =\sum_{n \in \mathbb{N}} \sigma_{A_{n}}$.\\ 
	The function $A \times B \to \sigma_{A}(B)$ is a nonnegative bimeasure, which by Theorem $1.10$ in \cite{berg0} there exists a nonnegative measure $\sigma \in \mathfrak{M}([0,\infty)^{2})$ such that $\sigma(A\times B) = \sigma_{A}(B)$, for every $A,B \in \mathscr{B}([0,\infty)^{2})$. \\
	Gathering all this information we obtain that
	$$
	f(\|x-y\|, \|u-v\|)=\int_{[0,\infty)^{2}}\frac{(1-e^{-r_{1}\|x-y\|^{2}})}{r_{1}}\frac{(1-e^{-r_{2}\|u-v\|^{2}})}{r_{2}}\prod_{i=1}^{2}(1+r_{i})d\sigma(r_{1},r_{2}), 
	$$
	which concludes that $(i)$ implies $(ii)$. The fact that $(ii)$ implies $(i)$ can be proved in a similar way as the converse of Theorem \ref{basicradialuni}, using the fact that the Kronecker product of CND kernels is PDI.\\
	The equivalence between $(ii)$ and $(iii)$ is a direct consequence of Theorem \ref{bernssev}. 
\end{proof}

In order to prove Corollary \ref{2times}    we will need the following special representation of completely monotone functions with $n$ variables.

\begin{lem}\label{sumcm} Let $g:(0, \infty) \to \mathbb{R}$, the function   $g(t_{1} +t_{2})$ is completely monotone with $2$ variables if and only if $g$ is completely monotone with one variable. 
\end{lem}

\begin{proof} For the converse, if $g$ is completely monotone, then
	$$	
	(-1)^{|\alpha|}\partial^{\alpha}[g(t_{1} + t_{2}))] = (-1)^{|\alpha|}g^{(|\alpha|)} (t_{1} + t_{2}) \geq 0.	
	$$	
	Conversely, if the function $g(t_{1} + t_{2})$ is completely monotone with $2$ variables, then for any $c >0$  
	$$	
	(-1)^{k }g^{(k)} (t_{1} + c) = (-1)^{k}\partial_{1}^{k}[g(t_{1} + c))] \geq 0, \quad k \in \mathbb{Z}_{+}	
	$$
	which proves that $g$ is completely monotone.
\end{proof}

In particular, for  a function $g$ that satisfies Lemma \ref{sumcm}, we have that

$$
g(t_{1}  +t_{2})= \int_{[0, \infty)}e^{-r(t_{1} + t_{2})}d\sigma(r)
$$
that is, the measure that represents $g(t_{1} + t_{2})$ has support on $\{(r, r), \quad  r \in [0, \infty)  \} \subset [0, \infty)^{2}$. 

The Cauchy functional equation is concerned with which functions $z: \mathbb{R} \to \mathbb{R} $ satisfies the following relation
$$
z(t_{1} + t_{2}) - z(t_{1}) - z(t_{2}) + z(0)=0, \quad t_{1}, t_{2} \in \mathbb{R}.
$$ 

Under some very weak assumptions on the function $z$, which continuity is a special case,  it can be proved that $z(t) = at+ b$, for some $a, b\in \mathbb{R}$. For the proof of Corollary \ref{2times} we need a  consequence  of this result: assume that  $z$ is continuous but it is defined on $[0, \infty)$ and the functional equation only holds for $t_{1}, t_{2} \in [0, \infty)$. Likewise, the function $z$ still is a polynomial of degree at most one, which can be proved by  defining the function $z$ on $t \in (-\infty, 0)$ by the relation $z(t):= -z(-t) + 2z(0)$. This function is continuous and satisfies the original Cauchy functional equation.

\begin{proof}[\textbf{Proof of Corollary \ref{2times}}] 
	Relation $(iii)$ and $(iv)$ are equivalent because of Theorem \ref{compleelltimes}.\\	
	Relation $(i)$ and $(ii)$ are equivalent by  Lemma	\ref{PDIsimpli}.\\	
	If relation $(iii)$ holds,   let $x_{1}, \ldots, x_{n} \in \mathbb{R}^{d}$, $y_{1}, \ldots , y_{m} \in \mathbb{R}^{d^{\prime}}$  and real scalars $c_{i, k}$ with the restrictions of Definition \ref{PDI}. Then, relation $(ii)$ holds because 
	\begin{align*}
		\sum_{i,j=1}^{n} \sum_{k,l=1}^{m}c_{i,k}c_{j,l}&f(\|x_{i} - x_{j}\|^{2} + \|y_{k}-y_{j}\|^{2})=2\sum_{i,j=1}^{n} \sum_{k,l=1}^{m}c_{i,k}c_{j,l}  \|x_{i}-x_{j}\|^{2}\|y_{k}-y_{j}\|^{2}\\
		& + \int_{(0, \infty)} \left [\sum_{i,j=1}^{n} \sum_{k,l=1}^{m}c_{i,k}c_{j,l}e^{-r(\|x_{i}-x_{j}\|^{2}+\|y_{k}-y_{j}\|^{2})} \right ]\frac{1+r^{2}}{r^{2}}d\sigma(r)\geq 0. 	
	\end{align*}
	Now, suppose that relation $(i)$ holds, since the sated kernel is PDI and satisfies relation $(i)$ of Theorem \ref{basicradial}, we obtain that the function $g(t_{1},t_{2}) = f( t_{1} + t_{2})
	- 	f( t_{1})-	f(t_{2}) + f(0)$ is a Bernstein function with two variables. Note that the function $\partial_{1}^{1}\partial_{2}^{1}g(t_{1},t_{2})$  is completely monotone and it only depends on  the value of $t_{1}+t_{2}$, because
	\begin{align*}
		\partial_{1}^{1}\partial_{2}^{1}g(t_{1},t_{2})&= \lim_{h,h^{\prime} \to 0}\frac{g(t_{1}+ h, t_{2} +h^{\prime}) - g(t_{1} , t_{2} +h^{\prime}) - g(t_{1} +h , t_{2} ) + g(t_{1} , t_{2})}{hh^{\prime}}\\
		&=\lim_{h,h^{\prime} \to 0}\frac{ f(t_{1}+t_{2} + h +h^{\prime}) -f(t_{1}+t_{2} +h^{\prime}) -f(t_{1}+t_{2} + h ) +f(t_{1}+t_{2})    }{hh^{\prime}}.  
	\end{align*}
	By the comment made after Lemma \ref{sumcm}, there exists a nonnegative measure $\sigma^{\prime}$  defined on $[0, \infty)$ (not necessarily bounded) for which
	$\partial_{1}^{1}\partial_{2}^{1}g(t_{1},t_{2}) = \int_{[0, \infty)}e^{-r(t_{1}+t_{2})}d\sigma^{\prime}(r)$. Using this measure, we may adapt the proof of Theorem \ref{bernssev} to obtain that 
	$$
	f(t_{1}+t_{2} ) -f(t_{1} )-f(t_{2} )+f(0)=g(t_{1},t_{2})  =  \int_{[0, \infty)} \frac{(1-e^{-rt_{1}})}{r}\frac{(1-e^{-rt_{2}})}{r} (1+r)^{2}d\sigma(r), 
	$$
	where $\sigma \in \mathfrak{M}([0, \infty))$ is nonnegative. On the other hand, define the function 
	$$
	h(s):= \int_{(0, \infty)} (e^{-rs} - e_{2}(r)\omega_{2,\infty}(rs))\frac{ (1+r)^{2}}{r^{2}} d\sigma(r) + s^{2}\frac{\sigma(\{0\})}{2} 
	$$
	which is a well defined $CM_{2}$ by Theorem \ref{compleelltimes}. Note that
	$$
	h(t_{1}+t_{2}) - h(t_{1}) - h(t_{2})  + h(0)=  \int_{[0, \infty)} \frac{(1-e^{-rt_{1}})}{r}\frac{(1-e^{-rt_{2}})}{r} (1+ r)^{2}d\sigma(r).
	$$
	But then, the continuous  function $z(s):=h(s) - f(s)$ satisfies the Cauchy functional equation on $[0, \infty)$, implying that $z$ is a polynomial of degree at most one, hence $f \in C^{\infty}((0, \infty))$  and also is a completely monotone function of order $2$. \end{proof}

\subsection{\textbf{Section \ref{Examples of PDI-Characteristic kernels}}} \label{proofsecExamples of PDI-Characteristic kernels}

It is possible to prove Lemma \ref{integmarginbernstein} by showing that the sets $X_{\mu, \nu}^{\mathfrak{I}}, Y_{\mu, \nu}^{\mathfrak{I}}$  on Lemma \ref{integmargin} are respectively  equal to $X, Y$, but we prove it by a direct argument.

\begin{proof}[\textbf{Proof of Lemma \ref{integmarginbernstein}}] If  $\mathfrak{I}_{g}^{\gamma, \varsigma} \in L^{1}((\mu\times \nu)\times (\mu\times \nu))$, by Fubini-Toneli there exists a pair $(x^{\prime}, y^\prime)$ such that $(x,y) \to  \mathfrak{I}_{g}^{\gamma, \varsigma}((x,y),(x^{\prime}, y^\prime)) \in L^{1}(\mu\times \nu)$.\\
	Suppose that relation $(ii)$ holds and define the sets
	$$
	A_{1}:= \{ x\in X , \quad \gamma(x,x^{\prime}) \leq 1\}, \quad A_{2}:=(A_{1})^{c},
	$$
	$$
	B_{1}:= \{ y\in Y , \quad \varsigma(y,y^\prime) \leq 1\}, \quad B_{2}:=(B_{1})^{c}.
	$$
	In order to prove that for an arbitrary $(z, w) \in X\times Y$, the function $(x,y) \to  \mathfrak{I}_{g}^{\gamma, \varsigma}((x,y),(z, w)) \in L^{1}(\mu\times \nu)$, we analyze the  integral in the four regions $A_{i}\times B_{j}$, $1\leq i ,j \leq 2$. Note that $x \to \max  (1, \gamma(x,z)/\gamma(x,x^{\prime}))$ is a bounded function on $A_{2}$ because $\gamma$ has bounded diagonal and the following inequality 
	\begin{equation}\label{descond}
		\begin{split}	
			0&\geq  2\gamma(x,z) -4 \gamma(x,x^{\prime}) -4\gamma(x^{\prime},z) + \gamma(x,x) + 4\gamma(x^{\prime}, x^{\prime}) + \gamma(z,z)\\
			& \geq 2\gamma(x,z) -4 \gamma(x,x^{\prime}) -4\gamma(x^{\prime},z),
		\end{split}
	\end{equation} 
	which holds because  $\gamma$ is CND. Similar for $\varsigma$ on $B_{2}$. In particular, we obtain the  integrability  of  $(x,y) \to  \mathfrak{I}_{g}^{\gamma, \varsigma}((x,y),(z, w))$ on the set $A_{2}\times B_{2}$ by using Equation \ref{consineqexp}.\\
	On the other hand, Equation  \ref{descond} provide us that  $\gamma(x,z) \leq 2 + 2\gamma(x^{\prime}, z):= M_{1}$,   for every  $x \in A_{1}$. In particular, we obtain the  integrability  of  $(x,y) \to  \mathfrak{I}_{g}^{\gamma, \varsigma}((x,y),(z, w))$ on the set $A_{1}\times B_{1}$, because the function is bounded in this set.\\
	On the remaining sets we mix these two relations.  For instance, if  $(x,y) \in A_{1} \times B_{2} $
	$$		
	0\leq g(\gamma(z,x), \varsigma(w,y))\leq \max \left (1, \frac{\varsigma(y,w)}{\varsigma(y, y^{\prime})} \right ) g(M_{1} , \varsigma( y, y^{\prime})).
	$$	
	In order to prove the integrability in $A_{1} \times B_{2}$,  we prove a stronger relation that we use later in the text, that relation $(ii)$ implies that 
	\begin{equation} \label{suff}
		y \to g(t_{1} , \varsigma(y^{\prime}, y)) \in L^{1}(\nu), \quad t_{1} \geq 0
	\end{equation} 
	The case $t_{1}=0$ is immediate. If $t_{1} > 0$, since   
	$  \mathfrak{I}_{g}^{\gamma, \varsigma}((\cdot,\cdot),(x^{\prime}, y^\prime)) \in L^{1}(\mu\times \nu)$, then for $\mu$ almost every $x \in X$ the function $  \mathfrak{I}_{g}^{\gamma, \varsigma}((x,\cdot),(x^{\prime}, y^\prime)) \in L^{1}( \nu)$. Since $\mu \in \mathfrak{M}(X)\setminus{\delta(X)}$, there must exist a $x\neq x^{\prime}$ satisfying this integrability relation and the conclusion is obtained by Equation \ref{consineqexp}. \\
	The fact that relation $(iii)$ implies relation $(i)$ is a direct consequence of Equation \ref{convexexp}, because
	\begin{align*}
		0 &\leq g(\gamma(x,x^{\prime}), \varsigma(y,y^{\prime})) \leq  g(2 \gamma(x,z) + 2 \gamma(z,x^{\prime}), 2\varsigma(y,w) + 2\varsigma(w,y^{\prime}) )\\
		&\leq 4[ g( \gamma(x,z), \varsigma(y,w) ) + g( \gamma(x,z) ,  \varsigma(w,y^{\prime}) )+g(   \gamma(z,x^{\prime}), \varsigma(y,w) )+g( \gamma(z,x^{\prime}),  \varsigma(w,y^{\prime}) ) ].   	
	\end{align*}
	
\end{proof}

\begin{proof}[\textbf{Proof of Corollary \ref{integmargincorbernstein} }]
	Relation $(i)$ implies relation $(ii)$ by Fubini-Tonelli.\\
	In order to prove that relation $(ii)$ implies relation $(iii)$, we also use the sets $A_{i}\times B_{j}$. The argument for $A_{1}\times B_{1}$ and $A_{2}\times B_{2}$ are the same. If   $(x,y) \in A_{1} \times B_{2} $
	$$		
	0\leq g(\gamma(z,x), \varsigma(w,y))\leq \max \left (1, \frac{\varsigma(y,w)}{\varsigma(y, y^{\prime})} \right ) g(M_{1} , \varsigma( y, y^{\prime})),
	$$	
	so, in order to prove the integrability in $A_{1} \times B_{2}$,  it is sufficient to prove that $y \to g(t_{1} , \varsigma(y^{\prime}, y)) \in L^{1}(|\lambda|_{Y})$, for any $t_{1} \geq 0$. But, this is a direct consequence of the hypothesis that   
	$ \mathfrak{I}_{g}^{\gamma, \varsigma} \in L^{1}((|\lambda|_{X}\times |\lambda|_{Y})\times (|\lambda|_{X}\times |\lambda|_{Y}))$ and the result in Equation \ref{suff}. \\
	The proof that relation $(iii)$ implies relation $(i)$ follows by the same method of Lemma \ref{integmarginbernstein} together with the hypothesis that   
	$ \mathfrak{I}_{g}^{\gamma, \varsigma} \in L^{1}((|\lambda|_{X}\times |\lambda|_{Y})\times (|\lambda|_{X}\times |\lambda|_{Y}))$. 
\end{proof}

\begin{proof}[\textbf{Proof of Theorem \ref{principalmetricbernstein}}] 
	The fact that $\mathfrak{M}(\mathfrak{I}_{g}^{\gamma, \varsigma})$ is a vector space is a direct consequence of the fact that relation $(iii)$ in Lemma \ref{integmarginbernstein} and Corollary \ref{integmargincorbernstein} holds for every $(x^{\prime}, y^{\prime}) \in X\times Y$.\\
	The fact that the real valued function is well defined is a consequence that $\mathfrak{M}(\mathfrak{I}_{g}^{\gamma, \varsigma})$ is a vector space, as $L^{1}((|\lambda| + |\lambda^{\prime}|)\times (|\lambda| + |\lambda^{\prime}|)) \subset L^{1}(|\lambda|\times |\lambda^{\prime}|) $. By the representation in Theorem \ref{bernssev}, if $\lambda \in \mathfrak{M}(\mathfrak{I}_{g}^{\gamma, \varsigma})$, then  
	\begin{align*}
		\int_{X \times Y} \int_{X \times Y}&\mathfrak{I}_{g}^{\gamma, \varsigma}((x,y),(x^{\prime}, y^{\prime}))d\lambda(x,y)d\lambda(x^{\prime}, y^{\prime})\\
		=\int_{[0, \infty)^{2}}&\left [\int_{X \times Y} \int_{X \times Y}\frac{(1-e^{-r_{1}\gamma(x, x^{\prime})})}{r_{1}}\frac{(1-e^{-r_{2}\varsigma(y, y^{\prime})})}{r_{2}} d\lambda(x,y)d\lambda(x^{\prime}, y^{\prime})\right ]\\
		&   \otimes (1+r_{1})(1+r_{2})  d\sigma(r_{1}, r_{2}).\end{align*} 
	As the continuous kernel $(1-e^{-r_{1}\gamma(x, x^{\prime})})/r_{1}$, defined in $X$, is CND, metrizable, nonnegative  and has bounded diagonal for every $r_{1} \in [0, \infty)$ (similar for the other kernel), Corollary \ref{integmarginkroeprodcor} implies that the inner double integral is nonnegative for every $r_{1}, r_{2} \in [0, \infty)$, hence $(\lambda, \lambda) \geq 0$. 
\end{proof}

In order to prove Theorem \ref{characPDIbern2}, we rewrite the integral representation in Theorem \ref{bernssev}  as
\begin{align*}
	&\int_{[0,\infty)^{2}}\frac{(1-e^{-r_{1}t_{1}})}{r_{1}}\frac{(1-e^{-r_{2}t_{2}})}{r_{2}}(1 +r_{1})(1+r_{2})d\sigma(r_{1},r_{2})\\
	&=\int_{(0,\infty)^{2} \cup \{0\}^{2} \cup [\{0\}\times (0, \infty)] \cup [(0, \infty)\times \{0\}] }\frac{(1-e^{-r_{1}t_{1}})}{r_{1}}\frac{(1-e^{-r_{2}t_{2}})}{r_{2}}(1 +r_{1})(1+r_{2})d\sigma(r_{1},r_{2})\\
	&=g_{0}(t_{1},t_{2}) +  \sigma(\{0\}^{2})t_{1}t_{2} + t_{1}\psi(t_{2}) + t_{2}\phi(t_{1}),
\end{align*}		

where we use the notation
$$
\psi(t_{2}):=\int_{(0, \infty)}\frac{(1-e^{-r_{2}t_{2}})}{r_{2}}(1+r_{2})d\sigma(\{0\}\times r_{2})$$
$$ \varphi(t_{1}):=\int_{(0, \infty)}\frac{(1-e^{-r_{1}t_{1}})}{r_{1}}(1+r_{1})d\sigma(r_{1}\times \{0\})
$$ 
$$
g_{0}(t_{1}, t_{2}):= \int_{(0,\infty)^{2}}\frac{(1-e^{-r_{1}t_{1}})}{r_{1}}\frac{(1-e^{-r_{2}t_{2}})}{r_{2}}(1 +r_{1})(1+r_{2})d\sigma(r_{1},r_{2}).
$$

\begin{proof}[\textbf{Proof of Theorem \ref{characPDIbern2}}] We separate the proof in $2$ cases, with the second one subdivided in $4$ subcases.\\
	$\bullet$ Case $1)$  $\sigma((0, \infty)^{2}) >0$:\\
	If $\lambda$ is a nonzero element in  $\mathfrak{M}(\mathfrak{I}_{g}^{\gamma, \varsigma})$, then 
	\begin{align*}
		&\int_{X \times Y} \int_{X \times Y}\mathfrak{I}_{g}^{\gamma, \varsigma}((x,y),(x^{\prime}, y^{\prime}))d\lambda(x,y)d\lambda(x^{\prime}, y^{\prime})\\
		&\geq \int_{X \times Y} \int_{X \times Y} g_{0}((x,y),(x^{\prime}, y^{\prime}))d\lambda(x,y)d\lambda(x^{\prime}, y^{\prime})\\
		&=\int_{(0, \infty)^{2}}\left [\int_{X \times Y} \int_{X \times Y}e^{-r_{1}\gamma(x, x^{\prime})}e^{-r_{2}\varsigma(y, y^{\prime})} d\lambda(x,y)d\lambda(x^{\prime}, y^{\prime})\right ] \prod_{i=1}^{2}\frac{1+r_{i}}{r_{i}}d\sigma(r_{1}, r_{2})>0
	\end{align*}	
	because the continuous kernel $r_{1}\gamma(x, x^{\prime})+r_{2}\varsigma(y, y^{\prime})$ is CND, metrizable and has bounded diagonal for every $r_{1}, r_{2} \in (0, \infty)$, so Theorem \ref{gengaussian} implies that the inner double integral is a positive number for every $r_{1}, r_{2}$. This proves that the kernel is PDI-Characteristic whenever $\sigma((0, \infty)^{2}) >0$.\\
	$\bullet$ Case $2)$  $\sigma((0, \infty)^{2}) =0$:\\
	In this case we may write the kernel as
	$$
	\mathfrak{I}_{g}^{\gamma, \varsigma}((x,y),(x^{\prime}, y^{\prime}))= \gamma(x,x^{\prime})\psi(\varsigma(y, y^{\prime})) + \varsigma(y,y^{\prime})\varphi(\gamma(x,x^{\prime})) + \sigma(\{0\}^{2})\gamma(x,x^{\prime})\varsigma(y,y^{\prime}).  
	$$
	By Corollary \ref{integmarginkroeprodcor} and the comment made after Equation \ref{bern+cnd}, the  kernel $\gamma(x,x^{\prime})\psi(\varsigma(y, y^{\prime})) $ is PDI-Characteristic if and only if $\gamma$ is CND-Characteristic and $\psi$ is not the zero function (equivalentely, $\sigma(\{0\}\times (0, \infty))>0 $ ). \\
	Similarly, the  kernel $\varsigma(y,y^{\prime})\varphi(\gamma(x,x^{\prime})) $ is PDI-Characteristic if and only if $\varsigma$ is CND-Characteristic and $\varphi$ is not the zero function (equivalentely, $\sigma( (0, \infty)\times \{0\})>0 $ ). \\
	The kernel $\sigma(\{0\}^{2})\gamma(x,x^{\prime})\varsigma(y,y^{\prime})$ is PDI-Characteristic if and only if $\gamma$ and $\varsigma$ are CND-Characteristic  and $\sigma( \{0\}^{2})>0 $.\\
	We conclude the proof of the Theorem by showing that the kernel $\mathfrak{I}_{g}^{\gamma, \varsigma}$ is PDI-Characteristic if and only if at least one of the $3$ kernels that describe it is PDI-Characteristic. We only need to prove the counterpositive.\\
	$\bullet$ Subcase $1)$ $\varphi$ and $\psi$ are not the zero function:\\
	In this case we must have that $\gamma$ and $\varsigma$ are not CND-Characteristic. Hence, there exists nonzero measures  $\lambda_{1} \in \mathfrak{M}_{1}(X; \gamma)$ and   $\lambda_{2} \in \mathfrak{M}_{1}(Y; \varsigma)$, such that:
	$$\int_{X} \int_{X}\gamma(x,x^{\prime})d\lambda_{1}(x)d\lambda_{1}(x^{\prime})=0, \quad \int_{Y} \int_{Y}\varsigma(y,y^{\prime})d\lambda_{2}(y)d\lambda_{2}(y^{\prime})=0.
	$$
	Define the measure  $\lambda= \lambda_{1}\times\lambda_{2}$, which is a nonzero element of $\mathfrak{M}(\mathfrak{I}_{g}^{\gamma, \varsigma})$ because a consequence of Equation \ref{ineqexp} is that 
	$$
	0 \leq \psi(\varsigma(y,y^{\prime})) \leq (\varsigma(y,y^{\prime})+1)\psi(1). 
	$$
	The conclusion that $(\lambda, \lambda)=0$ comes from Fubini-Tonelli.\\
	$\bullet$ Subcase $2)$ $\varphi=0$ and $\psi$ is not the zero function:\\  
	The proof is essentially the same of Subcase 1. The kernel $\gamma$ is not CND-Characteristic, then pick $\lambda_{1}$ satisfying the same integral relation and an arbitrary  nonzero  $\lambda_{2} \in \mathfrak{M}_{1}(Y; \varsigma)$. The same measure $\lambda$ implies that $(\lambda, \lambda)=0$ using Fubini-Tonelli.    \\
	$\bullet$ Subcase $3)$ $\psi=0$ and $\varphi$ is not the zero function:\\
	The proof is similar to the Subcase 2 and is omitted.\\
	$\bullet$ Subcase $4)$ $\psi=\varphi=0$.\\
	On this scenario $\mathfrak{I}_{g}^{\gamma, \varsigma}= \sigma(\{0\}^{2})\gamma \varsigma$ and the claim is immediate.
\end{proof}

\begin{proof}[\textbf{Proof of Lemma \ref{integmargincompl2}}] Relation $(i)$, $(ii)$ and $(iii)$ are equivalent because the kernel $\gamma + \varsigma$ is CND and Lemma \ref{estimativa2}. \\
	In order to prove the equivalences between $(i)$, $(iv)$ and $(v)$ we assume without loss of generalization that $a_{0}=a_{1}=0$. Note that $(\gamma^{2} + \varsigma^{2})/2 \leq (\gamma + \varsigma)^{2} \leq 2 \gamma^{2} + 2 \varsigma^{2} $, and 
	\begin{align*}
		[e^{-r\gamma} - \omega_{2}(r\gamma)] + [e^{-r\varsigma} - \omega_{2}(r\varsigma)] &\leq (e^{-r[\gamma + \varsigma]} - \omega_{2}(r[\gamma + \varsigma])) \\
		&\leq \frac{M_{2}2^{2}}{M_{1}}([e^{-r\gamma} - \omega_{2}(r\gamma)] + [e^{-r\varsigma} - \omega_{2}(r\varsigma)] ), 
	\end{align*}		
	where the left hand inequality is a consequence that $E_{2}(s)= e^{-s} - \omega_{2}(s)$ is convex, while the right hand inequality is a consequence of the comments made after Equation 	\ref{ineqcm2c1}. Hence, there are positive  constants $M_{3}, M_{4}$  such that
	$$
	0\leq M_{3}[ \psi(\gamma) + \psi(\varsigma)] \leq \psi(\gamma + \varsigma) \leq M_{4} [ \psi(\gamma) + \psi(\varsigma)],
	$$
	and the conclusion follows directly from this inequality.
\end{proof}


\begin{proof}[ \textbf{Proof of Theorem \ref{principalmetriccompl2}}] The fact that it is a vector space is a direct consequence of Lemma \ref{integmargincompl2}. \\
	The real valued function is a semi-inner product because by Fubini-Tonelli, we have that
	\begin{align*}
		\int_{X\times Y} \int_{X\times Y}& \psi(\gamma(x,x^{\prime}) + \varsigma(y,y^{\prime}))d\lambda (x,y)d\lambda(x^{\prime}, y^{\prime})\\
		=& 2a_{2} \int_{X\times Y} \int_{X\times Y} \gamma(x,x^{\prime}) \varsigma(y,y^{\prime})	 d\lambda (x,y)d\lambda(x^{\prime}, y^{\prime}) \\
		& +\int_{(0, \infty)} \left [ \int_{X\times Y} \int_{X\times Y} e^{-r[\gamma(x,x^{\prime}) + \varsigma(y,y^{\prime})]} 	 d\lambda (x,y)d\lambda(x^{\prime}, y^{\prime})\right ]\frac{1+r}{r^{2}} d\sigma(r)\geq 0.
	\end{align*}		
	For the final part of the proof, note that  by Theorem \ref{gengaussian}, a  nonzero measure $\lambda \in 	\mathfrak{M}(\mathfrak{I}_{\psi}^{\gamma + \varsigma} )$ satisfies 	
	$$
	\int_{(0, \infty)} \left [ \int_{X\times Y} \int_{X\times Y} e^{-r[\gamma(x,x^{\prime}) + \varsigma(y,y^{\prime})]} 	 d\lambda (x,y)d\lambda(x^{\prime}, y^{\prime})\right ]\frac{1+r}{r^{2}} d\sigma(r)=0
	$$		
	if and only if $\sigma$ is the zero measure. This occurs because $\gamma + \sigma$ is a CND metrizable kernel with bounded diagonal. The conclusion is then a consequence of Corollary \ref{integmarginkroeprodcor}.		 
\end{proof}

\subsection{\textbf{Section \ref{Distance covariance}}}

\begin{proof}[\textbf{Proof of Lemma \ref{cndtopdsquare}}]If $K_{\gamma}$ is PD then for any fixed $w \in X$ the kernel $(x,y) \to K_{\gamma}((x,w), (y,w))$ is also PD. As this is the same kernel that appears in the equivalence    presented in Equation \ref{Kgamma}, we obtain that $\gamma$ is an CND kernel. Conversely, let $x_{1}, \ldots, x_{n} \in X$ and constants $c_{i, \alpha} \in \mathbb{R}$, $1\leq i, \alpha \leq n$, since
	
	\begin{align*}
		\sum_{\alpha, \beta =1}^{n}\sum_{i,j=1}^{n}c_{\alpha, i}c_{\beta, j}K_{\gamma}((x_{\alpha}, x_{i}), (x_{\beta}, x_{j}))=&-\sum_{l,k=1}^{n}(d_{l}^{0}d_{k}^{0} +d_{l}^{1}d_{k}^{1} -d_{l}^{0}d_{k}^{1} - d_{l}^{1}d_{k}^{0}   )  \gamma(x_{l}, x_{k})\\
		& =-\sum_{l,k=1}^{n}(d_{l}^{0}- d_{l}^{1})(d_{k}^{0}-d_{k}^{1})  \gamma(x_{l}, x_{k}),	
	\end{align*}	
	where $d_{l}^{0} = \sum_{i=1}^{n}c_{l,i}$  and  $d_{l}^{1} = \sum_{\alpha=1}^{n}c_{\alpha,l}$. But, $\gamma$ is CND and $\sum_{l =1}^{n}d_{l}^{0} - d_{l}^{1} = 0$, which implies that $K_{\gamma}$	is PD and it settles the proof of relation $(i)$.\\
	For the proof of $(ii)$, note that $K_{\gamma}([x,y], [x,y])= -\gamma(x,x)- \gamma(y,y)+2\gamma(x,y)$. By Lemma \ref{estimativa} all those functions are in $L^{1}(P \times Q)$, which implies that $K_{\gamma}([x,y], [x,y]) \in L^{1}(P \times Q)$. For the final equality, note that
	\begin{align*}
		& 	\int_{X^{2}} \int_{X^{2}} -\gamma(x,y)
		d[P\otimes Q](x,z) d[P\otimes Q](y,w) =  \int_{X}\int_{X}  -\gamma(x,y),
		dP(x) dP(y)\\
		& 	\int_{X^{2}} \int_{X^{2}} -\gamma(z,w)
		d[P\otimes Q](x,z) d[P\otimes Q](y,w) =  \int_{X}\int_{X}  -\gamma(x,y),
		dQ(x) dQ(y)\\
		& 	\int_{X^{2}} \int_{X^{2}} \gamma(x,w)
		d[P\otimes Q](x,z) d[P\otimes Q](y,w) =  \int_{X}\int_{X}  \gamma(x,y),
		dP(x) dQ(y)\\
		& 	\int_{X^{2}} \int_{X^{2}}\gamma(z,y)
		d[P\otimes Q](x,z) d[P\otimes Q](y,w) =  \int_{X}\int_{X}  \gamma(x,y),
		dQ(x) dP(y).	
	\end{align*}
	by summing all those equalities we obtain the desired relation.

\end{proof}

\begin{proof}[\textbf{Proof of Lemma \ref{dcovtopdsquare}}]
	If $K_{\mathfrak{I}}$ is PD, then for any fixed $z \in X, w \in Y$ the kernel $(x,y),(x^{\prime}, y^{\prime}) \in [X\times Y]^{2} \to K_{\mathfrak{I}}([x,y,z,w], [x,y,z,w])   $ also is PD. Since this  is the  same kernel that appears in Lemma \ref{PDItoPD}, we obtain that $\mathfrak{I}$ is an PDI kernel. Conversely, suppose that $\mathfrak{I}$ is PDI and let $x_{1}, \ldots , x_{n} \in X$, $y_{1}, \ldots , y_{m} \in Y$ and scalars $c_{i,k, j,l } \in \mathbb{R}$, then 
	
	\begin{align*}
		&\sum_{i,j,i^{\prime}, j^{\prime}=1}^{n}\sum_{k,l,k^{\prime}, l^{\prime}=1}^{m}c_{i,k, j,l }c_{i^{\prime},k^{\prime}, j^{\prime},l^{\prime} }K_{\mathfrak{I}}([x_{i},y_{k},x_{j},y_{l} ],[x_{i^{\prime}},y_{k^{\prime}},x_{j^{\prime}},y_{l^{\prime}} ])\\
		= \sum_{i,j=1}^{n}\sum_{k,l=1}^{m}&\mathfrak{I}((x_{i},y_{k}),(x_{j},y_{l} ))  \bigg{[} d^{0,0}_{i,k}d^{0,0}_{ j,l } -d_{i,k}^{1,0}d_{ j,l }^{0,0}-d_{i,k}^{0,1}d_{ j,l }^{0,0}+d_{i,k}^{1,1}d_{ j,l }^{0,0} \\
		&-d_{i,k}^{0,0}d_{ j,l }^{1,0}+d_{i,k}^{1,0}d_{ j,l }^{1,0}+d_{i,k}^{0,1}d_{ j,l }^{1,0}- d_{i,k}^{1,1}d_{ j,l }^{1,0} -d^{0,0}_{i,k}d^{0,1}_{ j,l } +d_{i,k}^{1,0}d_{ j,l }^{0,1} \\
		&  +d_{i,k}^{0,1}d_{ j,l }^{0,1}-d_{i,k}^{1,1}d_{ j,l }^{0,1}+d_{i,k}^{0,0}d_{ j,l }^{1,1}-d_{i,k}^{1,0}d_{ j,l }^{1,1}-d_{i,k}^{0,1}d_{ j,l }^{1,1}+ d_{i,k}^{1,1}d_{ j,l }^{1,1} \bigg{]} \\
		= \sum_{i,j=1}^{n}\sum_{k,l=1}^{m}& \left [d_{i,k}^{0,0} - d_{i,k}^{1,0} - d_{i,k}^{0,1} + d_{i,k}^{1,1}  \right ] \left [d_{j,l}^{0,0} - d_{j,l}^{1,0} - d_{j,l}^{0,1} + d_{j,l}^{1,1}  \right ]\mathfrak{I}((x_{i},y_{k}),(x_{j},y_{l} )) 
	\end{align*}
	where
	$$
	d_{i,k}^{0,0}=\sum_{J=1}^{n}\sum_{L=1}^{m}c_{i,k, J,L}   \quad  d_{i,k}^{0,1} =\sum_{J=1}^{n}\sum_{L=1}^{m}c_{i,L, J,k}
	$$
	$$ 
	d_{i,k}^{1,0}=\sum_{J=1}^{n}\sum_{L=1}^{m}c_{J,k, i,L}   \quad  d_{i,k}^{1,1} =\sum_{J=1}^{n}\sum_{L=1}^{m}c_{J,L,i,k}.
	$$
	Note that the scalars  $e_{i,k}:=d_{i,k}^{0,0} - d_{i,k}^{1,0} - d_{i,k}^{0,1} + d_{i,k}^{1,1} $, satisfy the necessary  restrictions in  the definition of an PDI kernel,hence,  we obtain that $K^{\mathfrak{I}}$ is PD in $[X\times Y]^{2}$.\\
	For the proof of $(ii)$, we focus on the kernels in Theorem \ref{principalmetricbernstein}, as the other case the arguments are simpler and similar.\\
	The function $K_{\mathfrak{I}}([x,y,z,w],[x,y,z,w]) \in L^{1}(\lambda\times \lambda^{\prime})$ because each one of the $16$ functions on the right hand side of the definition of $K_{\mathfrak{I}}$  are in $L^{1}(\lambda\times \lambda^{\prime})$. Indeed, first note that is sufficient to prove the case $\lambda= \lambda^{\prime}$, as all 16 functions are nonnegative. The function $g(\gamma(x,x), \varsigma(y,y)) \in L^{1}(\lambda\times \lambda)$ because this function is bounded, as $\gamma$, $\varsigma$ have bounded diagonals. The function $g(\gamma(z,x), \varsigma(y,y)) \in L^{1}(\lambda\times \lambda)$ because of the  inequality 
	$$
	g(\gamma(z,x), \varsigma(y,y))  \leq g(\gamma(z,x), M)
	$$
	where $M= \sup_{y \in Y}\varsigma(y,y)$ and the result in Equation \ref{suff}. The function  $g(\gamma(x,z), \varsigma(y,w)) \in L^{1}(\lambda\times \lambda)$  by the Definition of $\Gamma(P,Q)_{\mathfrak{I}}$. The other 13 functions follows by the same arguments as one of those 3 cases.\\
	For the final claim, note that
	\begin{align*}
		&\int_{[X\times Y]^{2}}\int_{[X\times Y]^{2}}K_{\mathfrak{I}}([x,y, z,w ] ,[x^{\prime},y^{\prime}, z^{\prime},w^{\prime}] )d[\lambda\otimes \lambda^{\prime}](x,y, z,w )d[\lambda\otimes \lambda^{\prime}](x^{\prime},y^{\prime}, z^{\prime},w^{\prime} ) \\
		&=  \int_{X \times Y} \int_{X\times Y} \mathfrak{I}(a,b) \bigg{[}  d\lambda(a) d\lambda(b) -dP\otimes Q(a) d\lambda(b) -dP\otimes Q(a) d\lambda(b) +d\lambda^{\prime}(a) d\lambda(b) \\
		& 	-d\lambda(a) dP\otimes Q(b) +dP\otimes Q(a) dP\otimes Q(b) +dP\otimes Q(a) dP\otimes Q(b) -d\lambda^{\prime}(a) dP\otimes Q(b)\\
		& -d\lambda(a) dP\otimes Q(b) +dP\otimes Q(a) dP\otimes Q(b) +dP\otimes Q(a) dP\otimes Q(b) -d\lambda^{\prime}(a) dP\otimes Q(b) \\
		&+d\lambda(a)d\lambda^{\prime}(b)- dP\otimes Q(a)d\lambda^{\prime}(b) - dP\otimes Q(a)d\lambda^{\prime}(b) + d\lambda^{\prime}(a)d\lambda^{\prime}(b)\bigg{]}\\
		&= \int_{X \times Y} \int_{X\times Y} \mathfrak{I}((x,y),(x^{\prime}, y^{\prime}))d[\lambda + \lambda^{\prime} -2P\otimes Q ](x,y)d[\lambda + \lambda^{\prime} -2P\otimes Q ](x^{\prime},y^{\prime}).
	\end{align*}

\end{proof}

	This work was funded by Fundação de Amparo à Pesquisa do Estado de São Paulo - FAPESP grant 2021/04226-0.

\bibliographystyle{siam} 
\bibliography{Referrences}       

\begin{thebibliography}{10}

\bibitem{berg0}
{\sc C.~Berg, J.~Christensen, and P.~Ressel}, {\em Harmonic analysis on
  semigroups: theory of positive definite and related functions}, vol.~100 of
  Graduate Texts in Mathematics, Springer, 1984.

\bibitem{bergforst}
{\sc C.~Berg and G.~Forst}, {\em Potential theory on locally compact abelian
  groups}, vol.~87, Springer Science \& Business Media, 1975.

\bibitem{berlinet2011reproducing}
{\sc A.~Berlinet and C.~Thomas-Agnan}, {\em Reproducing kernel Hilbert spaces
  in probability and statistics}, Springer Science \& Business Media, 2011.

\bibitem{bochnerlivro}
{\sc S.~Bochner}, {\em Harmonic analysis and the theory of probability},
  Courier Corporation, 2005.

\bibitem{bottcher2019distance}
{\sc B.~{Bottcher}, M.~{Keller-Ressel}, and R.~L. {Schilling}}, {\em Distance
  multivariance: New dependence measures for random vectors}, Annals of
  Statistics, 47 (2019), pp.~2757--2789.

\bibitem{gretton2005measuring}
{\sc A.~Gretton, O.~Bousquet, A.~Smola, and B.~Sch{\"o}lkopf}, {\em Measuring
  statistical dependence with hilbert-schmidt norms}, in International
  conference on algorithmic learning theory, Springer, 2005, pp.~63--77.

\bibitem{gretton2008kernel}
{\sc A.~Gretton, K.~Fukumizu, C.~H. Teo, L.~Song, B.~Sch{\"o}lkopf, and A.~J.
  Smola}, {\em A kernel statistical test of independence}, in Advances in
  neural information processing systems, 2008, pp.~585--592.

\bibitem{gaussinfi}
{\sc J.~C. Guella}, {\em {On Gaussian kernels on Hilbert spaces and kernels on
  Hyperbolic spaces}}, arXiv e-prints,  (2020), p.~arXiv:2007.14697.

\bibitem{energybernstein}
{\sc J.~C. Guella}, {\em Generalization of the energy distance by bernstein
  functions},  (2021).

\bibitem{unifying}
{\sc J.~C. Guella and V.~A. Menegatto}, {\em Schoenberg s theorem for positive
  definite functions on products: A unifying framework}, Journal of Fourier
  Analysis and Applications,  (2018), pp.~1--23.

\bibitem{jeanrad}
\leavevmode\vrule height 2pt depth -1.6pt width 23pt, {\em Conditionally
  positive definite matrix valued kernels on {E}uclidean spaces}, Constructive
  Approximation, 52 (2020), pp.~65--92.

\bibitem{guo}
{\sc K.~Guo, S.~Hu, and X.~Sun}, {\em Conditionally positive definite functions
  and {L}aplace-{S}tieltjes integrals}, Journal of Approximation Theory, 74
  (1993), pp.~249--265.

\bibitem{lyons2013}
{\sc R.~Lyons}, {\em Distance covariance in metric spaces}, Ann. Probab., 41
  (2013), pp.~3284--3305.

\bibitem{lyons2018errata}
{\sc R.~{Lyons}}, {\em Errata to distance covariance in metric spaces},
  Annals of Probability, 46 (2018), pp.~2400--2405.

\bibitem{lyons2021second}
{\sc R.~{Lyons}},{\em Second errata to
  distance covariance in metric spaces}, Annals of Probability, 49
  (2021), pp.~2668--2670.

\bibitem{MethodsofDistances2013}
{\sc S.~T. {Rachev}, L.~B. {Klebanov}, S.~V. {Stoyanov}, and F.~{Fabozzi}},
  {\em The Methods of Distances in the Theory of Probability and Statistics},
  2013.

\bibitem{bers}
{\sc R.~L. Schilling, R.~Song, and Z.~Vondracek}, {\em Bernstein functions:
  theory and applications}, vol.~37, Walter de Gruyter, 2012.

\bibitem{schoenbradial}
{\sc I.~J. Schoenberg}, {\em Metric spaces and completely monotone functions},
  Annals of Mathematics,  (1938), pp.~811--841.

\bibitem{sejdinovic2013equivalence}
{\sc D.~Sejdinovic, B.~Sriperumbudur, A.~Gretton, and K.~Fukumizu}, {\em
  Equivalence of distance-based and rkhs-based statistics in hypothesis
  testing}, The Annals of Statistics,  (2013), pp.~2263--2291.

\bibitem{Steinwart}
{\sc I.~Steinwart and A.~Christmann}, {\em Support vector machines}, Springer
  Science \& Business Media, 2008.

\bibitem{szekely2009brownian}
{\sc G.~J. {Szekely} and M.~L. {Rizzo}}, {\em Brownian distance covariance},
  The Annals of Applied Statistics, 3 (2009), pp.~1236--1265.

\bibitem{neuschoen}
{\sc J.~Von-Neumann and I.~J. Schoenberg}, {\em Fourier integrals and metric
  geometry}, Transactions of the American Mathematical Society, 50 (1941),
  pp.~226--251.

\bibitem{wendland}
{\sc H.~Wendland}, {\em Scattered data approximation}, vol.~17, Cambridge
  university press, 2005.

\end{thebibliography}

\end{document}